\documentclass[11pt,reqno]{amsart}

\usepackage{amsmath,amssymb,mathrsfs}
\usepackage{graphicx,cite,enumitem}

\newtheorem{theo}{Theorem}[section]

\newtheorem{lemm}[theo]{Lemma}

\newtheorem{rema}[theo]{Remark}
\newtheorem{prob}[theo]{Problem}

\numberwithin{equation}{section}

\begin{document}

\title[stability for inverse source problems]{Stability for the inverse source
problems in elastic and electromagnetic waves}

\author{Gang Bao}
\address{School of Mathematical Sciences, Zhejiang University, Hangzhou
310027, China.}
\email{baog@zju.edu.cn}

\author{Peijun Li}
\address{Department of Mathematics, Purdue University, West Lafayette, Indiana
47907, USA.}
\email{lipeijun@math.purdue.edu}

\author{Yue Zhao}
\address{School of Mathematics and Statistics, Central China Normal University,
Wuhan 430079, China.}
\email{zhaoyueccnu@163.com}

\thanks{The work of G. Bao is supported in part by a NSFC Innovative Group Fund
(No.11621101), an Integrated Project of the Major Research Plan of NSFC (No.
91630309), and an NSFC A3 Project (No. 11421110002), and  the Fundamental
Research Funds for the Central Universities.}



\begin{abstract}
This paper concerns the inverse source problems for the time-harmonic
elastic and electromagnetic wave equations. The goal is to determine the
external force and the electric current density from boundary measurements of
the radiated wave field, respectively. The problems are challenging due to the
ill-posedness and complex model systems.  Uniqueness and stability are
established for both of the inverse source problems. Based on either continuous
or discrete multi-frequency data, a unified increasing stability theory is
developed. The stability estimates consist of two parts: the Lipschitz type data
discrepancy and the high frequency tail of the source functions. As the upper
bound of frequencies increases, the latter decreases and thus becomes
negligible. The increasing stability results reveal
that ill-posedness of the inverse problems can be overcome by using
multi-frequency data. The method is based on integral equations and analytical
continuation, and requires the Dirichlet data only. The analysis employs
asymptotic expansions of Green's tensors and the transparent boundary conditions
by using the Dirichlet-to-Neumann maps. In addition, for the first time,
the stability is established on the inverse source problems for both the Navier
and Maxwell equations.
\end{abstract}

\maketitle

\section{Introduction}

The inverse source problems in waves arise in many scientific and industrial areas such
as antenna design and synthesis, biomedical imaging, and photo-acoustic
tomography \cite{A-IP99}. For instance, in medical imaging, such
as magnetoencephalography (MEG),  the imaging modality is a non-invasive
neurophysiological technique that measures the electric or magnetic fields
generated by neuronal activity of the brain \cite{ABF-SIAP02, FKM-IP04,
NOHTA-PMB07}. The spatial distributions of the measured fields are analyzed to
localize the sources of the activity within the brain to provide information
about both the structure and function of the brain. The inverse source problems
are also considered as a basic mathematical tool for solving many
imaging problems including reflection tomography, diffusion-based optical
tomography, lidar imaging for chemical and biological threat detection, and
fluorescence microscopy \cite{I-89}.

Motivated by these significant applications, the inverse source problems, as an
important research subject in inverse scattering theory, have continuously attracted
much attention by many researchers \cite{AM-IP06, ABF-SIAP02, BN-IP11, BN-IP13, 
BCL-JUQ16, BCL-17, BCLZ-MC14, DML-SIAP07, L-IP11, LLC-IP18, ZG-IP15}.
Consequently, a great deal of mathematical and numerical results are available, especially
for the acoustic waves or the Helmholtz equations. In general, it is known that there is no uniqueness for the
inverse source problem at a fixed frequency due to the existence of
non-radiating sources \cite{BC-JMP77, DS-IEEE82, HKP-IP05}. Therefore,
additional information is required for the source in order to obtain a unique
solution, such as to seek the minimum energy solution \cite{MD-IEEE99}.
From the computational point of view, a more challenging issue is the lack of
stability. A small variation of the data might lead to a huge error in the
reconstruction. Recently, it has been realized that the use of multi-frequency
data is an effective approach to overcome the difficulties of non-uniqueness and
instability which are encountered at a single frequency.
In \cite{BLT-JDE10}, Bao et al. initialized the mathematical study on the
stability of the inverse source problem for the Helmholtz equation by
using multi-frequency data. The increasing stability was further studied in
\cite{CIL-JDE16},\cite{LY-2016} for the inverse source problem of the three-dimensional
Helmholtz equation. Based on the Huygens principle, the method
assumes a special form of the source function, and requires both the Dirichlet
and Neumann boundary data. A different approach was developed in \cite{LY-2016} to
obtain the same increasing stability result for both the two- and
three-dimensional Helmholtz equation. The method removes the assumption on the
source function and requires the Dirichlet data only. An
attempt was made in \cite{LY-JMAA17} to extend the stability result to the
inverse random source of the one-dimensional stochastic Helmholtz equation. We refer
to \cite{ACTV-IP12, BLRX-SINUM15, EV-IP09} for the study of the inverse source
problems by using multiple frequency information. A topical review can be found in
\cite{BLLT-IP15} on the inverse source problems as well as other inverse
scattering problems by using multiple frequencies to overcome the
ill-posedness and gain increased stability. We also
refer to \cite{I-D11} on the increasing stability of determining potentials
for the Sch\"{o}dinger equation. Related results can be found in \cite{AI-IP10,
HI-IP04, I-CM07} on the increasing stability in the solution of the
Cauchy problem for the acoustic and electromagnetic wave equations.

Although a lot of work has been done on the inverse source problem for acoustic
waves, little is known on the inverse source problems for elastic and
electromagnetic waves, especially their stability.
This work initializes the mathematical study and provides the first stability
results of the inverse source problems for elastic and electromagnetic waves. Our objective
is to develop a unified stability theory of the inverse source problems for 
elastic and electromagnetic waves. It
significantly extends the previous approaches for the Helmholtz equations to
handle the more complicated Navier and Maxwell equations. Especially, more
delicate studies
are needed for sophisticated Green's tensors of these two wave equations. The
results shed light on the stability analysis of the more challenging
inverse medium and obstacle scattering problems \cite{BLLT-IP15}. In addition,
they motivate further study of the time-domain inverse problem where all
frequencies are available in order to gain better stability \cite{BZ-JAMS14}. It
should also be pointed out that the general case is widely open for the inverse
source problems on these vector wave equations in inhomogeneous media. 
General references on elastic and electromagnetic wave scattering problems may
be found in \cite{ABG-15, BMU-SIAP15, C-88, LL-59, LWWZ-IP16, MP-JASA85,
NU-IM94, T-SIAP15} and \cite{CK-98, EINT-02, HR-WM98,N-00,
RK-94,SU-ARMA92,Y-IPT98},
respectively.

For electromagnetic waves, Ammari et al.\cite{ABF-SIAP02} showed
uniqueness and stability, and presented an inversion scheme to reconstruct
dipole sources based on a low-frequency asymptotic analysis of the time-harmonic
Maxwell equations. In \cite{AM-IP06}, Albanese
and Monk discussed uniqueness and non-uniqueness of the inverse source
problems for Maxwell's equations. A monograph can be found in \cite{RK-94} on general inverse
problems for Maxwell's equations. We refer to \cite{K-IP92, L-AA15, LY-AA05,
Y-IPT98} for solving inverse source problems on hyperbolic systems by using
Carleman estimates, and to \cite{OPS-DMJ93, RS-IP89} for inverse problems which
are related to Maxwell's equations. To the best of our knowledge, there
is no stability result for the inverse source problem of Maxwell's
equations in a general setting. The questions are completely open regarding
uniqueness and stability for the inverse source problem of the elastic wave
equation. In this paper, we develop new techniques and establish a unified
increasing stability theory in the inverse source scattering problems for both
elastic and electromagnetic waves, where the wave propagation is governed by
the two- or three-dimensional Navier equation and the three dimensional
Maxwell equations, respectively.

For elastic waves, the inverse source problem is to determine the external
force that produces the measured displacement. We show a uniqueness result and
demonstrate that the increasing stability can be achieved by using the Dirichlet
boundary data only at multiple frequencies. For electromagnetic waves, the
inverse source problem is to reconstruct the electric current density from the
tangential trace of the electric field. First we discuss the uniqueness of the
problem and distinguish the detectable radiating sources from the non-radiating
sources. Then we prove that the increasing stability can be obtained to
reconstruct the radiating electric
current densities from the boundary measurement at multiple frequencies. For
each wave, we give the stability estimates for both the continuous frequency
data and the discrete frequency data. The estimates consist of two parts: the
first part is the Lipschitz type of data discrepancy and the second part is the
high frequency tail of the source function. The former is analyzed via the Green
tensor. The latter is estimated by the analytical continuation, which decreases
as the frequency of the data increases. The results reveal that the
ill-posedness of the inverse problems can be overcome and the inverse
problems are stable when multi-frequency data is used. In our analysis, the main
ingredients are to use the transparent boundary conditions and Green's tensors
for the wave equations. The transparent boundary condition establishes the
relation between the Dirichlet data and the Neumann data. The Neumann data
can not only be represented in terms of the Dirichlet data, but also be computed
once the Dirichlet data is available in practice.

\begin{figure}
\centering
\includegraphics[width=0.3\textwidth]{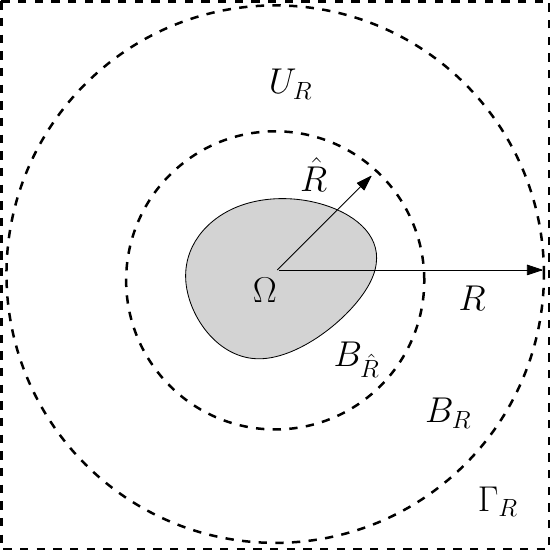}
\caption{Problem geometry of the inverse source scattering.}
\label{pg}
\end{figure}

Throughout, we assume that the source of either the external force or the
electric current density has a compact support $\Omega\subset\mathbb R^d$, $d=2$
or $3$. Let $\hat{R}>0$ be a sufficiently large constant such that
$\bar\Omega\subset B_{\hat R}=\{\boldsymbol x=(x_1, \dots, x_d)^\top\in\mathbb
R^d: |\boldsymbol x|<\hat{R}\}$. Let $R>\hat{R}$ be a constant such that
$B_{\hat R}\subset B_R=\{\boldsymbol x\in\mathbb R^d: |\boldsymbol x|<R\}$.
Denote by $\Gamma_R=\{\boldsymbol x\in \mathbb R^d: |\boldsymbol x|=R\}$ the
boundary of $B_R$ where the measurement of the wave field is taken. Let
$U_R=(-R, R)^d$ be a rectangular box in $\mathbb R^d$. Clearly we have
$\Omega\subset B_{\hat R}\subset B_R\subset U_R$. The problem
geometry is shown in Figure \ref{pg}.

The paper is organized as follows. In Section 2, we show the increasing
stability of the inverse source problem for elastic waves. Section 3 is devoted
to the inverse source problem for electromagnetic waves. The uniqueness and
non-uniqueness are discussed and the increasing stability is obtained. In both
sections, the analysis is carried for the continuous frequency data,
followed by the discussion for the discrete frequency data. The paper is
concluded with some general remarks in Section 4. To make the
paper easily accessible, some necessary notations and useful results are
provided in the appendices on the differential operators, Helmholtz
decomposition, and Sobolev spaces.

\section{Elastic waves}

This section addresses the inverse source problem for elastic waves. The
uniqueness and increasing stability are established to reconstruct the external
force from the boundary measurement of the displacement at multiple frequencies.

\subsection{Problem formulation}

Consider the time-harmonic Navier equation in a homogeneous medium:
\begin{equation}\label{ne}
\mu\Delta\boldsymbol{u}+ (\lambda + \mu)\nabla\nabla\cdot\boldsymbol{u} +
\omega^2\boldsymbol{u} = \boldsymbol{f}\quad\text{in}~\mathbb{R}^d,
\end{equation}
where $\omega>0$ is the angular frequency, $\lambda$
and $\mu$ are the Lam\'{e} constants satisfying $\mu>0$ and $\lambda + \mu>0$,
$\boldsymbol{u}\in\mathbb C^d$ is the displacement field, and
$\boldsymbol{f}\in\mathbb C^d$ accounts for the external force which is assumed
to have a compact support $\Omega\subset\mathbb R^d$.

An appropriate radiation condition is needed to complete the definition of the
scattering problem since it is imposed in the open domain. As discussed in
Appendix \ref{hd}, the displacement $\boldsymbol u$ can be decomposed into the
compressional part $\boldsymbol u_{\rm p}$ and the shear part $\boldsymbol
u_{\rm s}$:
\[
 \boldsymbol u=\boldsymbol u_{\rm p}+\boldsymbol u_{\rm s}\quad\text{in}~
\mathbb R^d\setminus\bar{\Omega}.
\]
The Kupradze--Sommerfeld radiation condition requires that $\boldsymbol
u_{\rm p}$ and $\boldsymbol u_{\rm s}$ satisfy the Sommerfeld radiation
condition:
\begin{equation}\label{rc}
\lim_{r\to\infty}r^{\frac{d-1}{2}}(\partial_r \boldsymbol{u}_{\rm p}-{\rm
i}\kappa_{\rm p}\boldsymbol{u}_{\rm
p})=0,\quad\lim_{r\to\infty}r^{\frac{d-1}{2}}(\partial_r
\boldsymbol{u}_{\rm s}-{\rm i}\kappa_{\rm s}\boldsymbol{u}_{\rm s})=0,\quad
r=|\boldsymbol x|,
\end{equation}
where $\kappa_{\rm p}, \kappa_{\rm s}$ are the compressional and shear
wavenumbers, given by
\[
 \kappa_{\rm p}=\frac{\omega}{(\lambda+2\mu)^{1/2}}=c_{\rm p}\omega,\quad
\kappa_{\rm s}=\frac{\omega}{\mu^{1/2}}=c_{\rm s}\omega,
\]
where
\begin{equation}\label{cj}
c_{\rm p}=(\lambda+2\mu)^{-1/2},\quad c_{\rm s}=\mu^{-1/2}.
\end{equation}
Note that $c_{\rm p}, c_{\rm s}$ are independent of $\omega$ and $c_{\rm
p}<c_{\rm s}$.

Given $\boldsymbol f\in L^2(\Omega)^d$, it is known that the scattering problem
\eqref{ne}--\eqref{rc} has a unique solution (cf. \cite{BCL-17}):
\begin{equation}\label{us}
 \boldsymbol u(\boldsymbol x, \omega)=\int_\Omega \mathbf
G_{\rm N}(\boldsymbol x, \boldsymbol y; \omega)\cdot\boldsymbol f (\boldsymbol
y){\rm d}\boldsymbol y,
\end{equation}
where $\mathbf G_{\rm N}(\boldsymbol x, \boldsymbol y; \omega)\in\mathbb
C^{d\times d}$ is Green's tensor for the Navier equation \eqref{ne} and the dot
is the matrix-vector multiplication. Explicitly, we have
\begin{equation}\label{gtn}
  \mathbf G_{\rm N}(\boldsymbol x, \boldsymbol y;
\omega)=\frac{1}{\mu}g_d(\boldsymbol x, \boldsymbol y; \kappa_{\rm s})\mathbf
I_d +\frac{1}{\omega^2}\nabla_{\boldsymbol x}\nabla^\top_{\boldsymbol
x}(g_d(\boldsymbol x, \boldsymbol y; \kappa_{\rm s})-g_d(\boldsymbol x,
\boldsymbol y; \kappa_{\rm p})),
\end{equation}
where $\mathbf I_d$ is the $d\times d$ identity matrix,
\begin{equation}\label{gn}
 g_2(\boldsymbol x, \boldsymbol y; \kappa)=\frac{\rm
i }{4}H_0^{(1)}(\kappa|\boldsymbol x-\boldsymbol y|)\quad\text{and} \quad
g_3(\boldsymbol x, \boldsymbol y; \kappa)=\frac{1}{4\pi}\frac{e^{{\rm
i}\kappa|\boldsymbol x-\boldsymbol y|}}{|\boldsymbol x-\boldsymbol y|}
\end{equation}
are the fundamental solutions for the two- and three-dimensional Helmholtz
equation, respectively, and $H_0^{(1)}$ is the Hankel function of the first
kind with order zero.

Define a boundary operator
\begin{equation}\label{tbo}
 D\boldsymbol{u}=\mu\partial_{\boldsymbol\nu}\boldsymbol{u}
+(\lambda+\mu)(\nabla\cdot\boldsymbol{u})\boldsymbol\nu\quad\text{on}~ \Gamma_R,
\end{equation}
where $\boldsymbol\nu$ is the unit normal vector on $\Gamma_R$. It is shown in
\cite{BP-JMAA08, LWWZ-IP16} that there exists a Dirichlet-to-Neumann (DtN)
operator $T_{\rm N}$ such that
\begin{equation}\label{tbc}
D\boldsymbol u=T_{\rm N}\boldsymbol u\quad\text{on}~\Gamma_R,
\end{equation}
which is the transparent boundary condition for the scattering problem of the
Navier equation.

\begin{prob}[Continuous frequency data for elastic waves]\label{p1}
Let the external force $\boldsymbol f$ be a complex function with the compact
support $\Omega$. The inverse source problem is to determine $\boldsymbol f$
from the displacement $\boldsymbol u(\boldsymbol x, \omega), \boldsymbol x\in
\Gamma_R, \omega\in (0, K)$, where $K>1$ is a constant.
\end{prob}

\begin{rema}
The boundary data does not have to be measured on the sphere $\Gamma_R$.
In fact, it can be measured on any Lipschitz continuous boundary $\Gamma$ which
encloses the compact support of $\boldsymbol f$, e.g., take
$\Gamma=\partial\Omega$. When $\boldsymbol u$ is available on $\Gamma$, we may
consider the following boundary value problem:
\begin{equation}\label{bvp}
 \begin{cases}
  \mu\Delta\boldsymbol u+(\lambda+\mu)\nabla\nabla\cdot\boldsymbol
u+\omega^2\boldsymbol u=0&\quad\text{in} ~ B_R\setminus \bar{\Omega},\\
\boldsymbol u=\boldsymbol u&\quad\text{on} ~ \Gamma,\\
D\boldsymbol u=T_{\rm N}\boldsymbol u&\quad\text{on} ~ \Gamma_R.
 \end{cases}
\end{equation}
It can be shown that the problem \eqref{bvp} has a unique solution $\boldsymbol
u$ in $B_R\setminus\bar{\Omega}$ \cite{LWWZ-IP16}. Therefore, the Dirichlet data
$\boldsymbol u$ is immediately available on $\Gamma_R$ once the problem
\eqref{bvp} is solved, and then the Neumann data $T_{\rm N}\boldsymbol u$ can be
computed on $\Gamma_R$ by using \eqref{tbc}.
\end{rema}

\subsection{Uniqueness}

This section is concerned with uniqueness of the inverse problem. Introduce two
auxiliary functions:
\begin{equation}\label{npw}
\boldsymbol{u}^{\rm inc}_{\rm p}(\boldsymbol x)=\boldsymbol{p} e^{-{\rm
i}\kappa_{\rm p}\boldsymbol{x}\cdot\boldsymbol{d}}\quad\text{and}\quad
\boldsymbol{u}^{\rm inc}_{\rm s}(\boldsymbol x)=\boldsymbol{q}e^{-{\rm
i}\kappa_{\rm s}\boldsymbol{x}\cdot\boldsymbol{d}},
\end{equation}
where $\boldsymbol d\in\mathbb S^{d-1}$ is the unit propagation direction vector
and $\boldsymbol p, \boldsymbol q\in\mathbb S^{d-1}$ are unit polarization
vectors. These unit vectors may be chosen as follows:

\begin{enumerate}[label=(\roman*)]

\item For $d=2$, $\boldsymbol d(\theta)=(\cos\theta, \sin\theta)^\top$,
$\boldsymbol p(\theta)$ and $\boldsymbol q(\theta)$ satisfy $\boldsymbol
p(\theta)=\boldsymbol d(\theta)$ and $\boldsymbol
q(\theta)\cdot\boldsymbol d(\theta)=0$ for all $\theta\in [0, 2\pi]$.

\item For $d=3$, $\boldsymbol d(\theta, \varphi)=(\sin\theta\cos\varphi,
\sin\theta\sin\varphi, \cos\theta)^\top$, $\boldsymbol p(\theta, \varphi)$ and
$\boldsymbol q(\theta, \varphi)$ satisfy $\boldsymbol p(\theta,
\varphi)=\boldsymbol d(\theta, \varphi)$ and $\boldsymbol q(\theta,
\varphi)\cdot\boldsymbol d(\theta, \varphi)=0$ for all $\theta\in[0, \pi],
\varphi\in[0, 2\pi]$.

\end{enumerate}
In fact, $\boldsymbol u^{\rm inc}_{\rm p}$ and $\boldsymbol u^{\rm inc}_{\rm s}$
are known as the compressional and shear plane waves. It is easy to verify that
they satisfy the homogeneous Navier equation:
\begin{equation}\label{nec}
\mu\Delta\boldsymbol{u} + (\lambda +
\mu)\nabla\nabla\cdot\boldsymbol{u} +
\omega^2\boldsymbol{u} =0\quad\text{in}~\mathbb R^d.
\end{equation}

\begin{theo}
Let $I \subset \mathbb R^+$ be an open interval. Then the external force
$\boldsymbol f$ can be uniquely determined by the multiple-frequency
data $\{\boldsymbol u(\boldsymbol x,\omega): \boldsymbol x\in\Gamma_R,\, \omega
\in I\}$.
\end{theo}

\begin{proof}

We prove the two dimensional case in details, and then
briefly present the proof for the three dimensional case since the steps
are similar. Let $\boldsymbol u(\boldsymbol x,\omega) = 0$ for
$\boldsymbol x\in \Gamma_R$ and $\omega \in I$. It suffices to show that
$\boldsymbol f=0$.

(i) Consider $d=2$. Let $\boldsymbol\xi_{\rm p}= \kappa_{\rm p} \boldsymbol d$.
The compressional plane wave in \eqref{npw} can be written as
$\boldsymbol{u}_{\rm p}^{\rm{inc}}(\boldsymbol x) =
\boldsymbol{p} e^{-{\rm i}\boldsymbol\xi_{\rm{p}}\cdot \boldsymbol
x}$. Multiplying the both sides of \eqref{ne} by $\boldsymbol{u}_{\rm
p}^{\rm inc}(\boldsymbol x)$, using the integration by parts over $B_R$, and
noting \eqref{nec}, we obtain
\[
 \int_{B_R}(\boldsymbol p e^{-{\rm i}\boldsymbol\xi_{\rm{p}}\cdot\boldsymbol
x})\cdot\boldsymbol{f}(\boldsymbol x){\rm d}\boldsymbol x=\int_{\Gamma_R}\left(
\boldsymbol{u}_{\rm{p}}^{\rm{inc}}(\boldsymbol x) \cdot
T_{\rm N}\boldsymbol u(\boldsymbol x,\omega)+ \boldsymbol{u}(\boldsymbol
x,\omega)\cdot
D \boldsymbol{u}_{\rm{p}}^{\rm{inc}}(\boldsymbol x)\right){\rm
d}S(\boldsymbol x),
\]
which means
\[
\boldsymbol p \cdot \hat{\boldsymbol f}(\boldsymbol\xi_{\rm{p}}) =
0,\quad\forall\omega\in I.
\]
Since $ \hat{\boldsymbol f}(\kappa_{\rm p} \boldsymbol
d)= \hat{\boldsymbol f}(c_{\rm p} \omega\boldsymbol d)$ is an analytic function
with respect to $\omega \in \mathbb C$, for each fixed $\boldsymbol d$ and
$\boldsymbol p$, we have
$\boldsymbol p \cdot \hat{\boldsymbol f}(\boldsymbol\xi_{\rm{p}}) = 0$ for all
$\kappa_{\rm{p}}\in (0, +\infty)$.

Let $\boldsymbol\xi_{\rm s} = -\kappa_{\rm s} \boldsymbol{d}$ with
$|\boldsymbol\xi_{\rm s}|=\kappa_{\rm s}\in (0, \infty)$. The shear plane wave
in \eqref{npw} can be written as $\boldsymbol u^{\rm inc}_{\rm
s}(\boldsymbol x)=\boldsymbol q e^{-{\rm i}\boldsymbol \xi_{\rm
s}\cdot\boldsymbol x}$. Multiplying $\boldsymbol u^{\rm inc}_{\rm s}$ on both
sides of \eqref{ne}, using the integration by parts, we
may similarly get $\boldsymbol q \cdot \hat{\boldsymbol
f}(\boldsymbol\xi_{\rm{s}}) = 0$ for all $\kappa_{\rm{s}}\in (0, +\infty)$.
Hence, for each $\kappa>0$, we have both $\boldsymbol p \cdot \hat{\boldsymbol
f}(\kappa \boldsymbol d) = 0$ and $\boldsymbol q \cdot \hat{\boldsymbol
f}(\kappa \boldsymbol d) = 0$. Let $\boldsymbol p(\theta)=(\cos\theta,
\sin\theta)^\top$ and take
$\boldsymbol q(\theta)=(-\sin\theta, \cos\theta)^\top$. Then $\boldsymbol
p(\theta)\cdot\boldsymbol q(\theta)=0$ and they form an orthonormal
basis in $\mathbb R^2$ for any $\theta\in [0, 2\pi]$. Hence we have from the
Pythagorean theorem that
\[
|\hat{\boldsymbol{f}}(\kappa\boldsymbol d)|^2 = |\boldsymbol{p} \cdot
\hat{\boldsymbol{f}}(\kappa\boldsymbol d)|^2 + |\boldsymbol{q}\cdot
\hat{\boldsymbol{f}}(\kappa\boldsymbol d)|^2 = 0
\]
for each $\kappa>0$ and $\boldsymbol d \in \mathbb S$, which means
$\hat{\boldsymbol{f}} = 0$ and then $\boldsymbol f$ = 0.

(ii) Consider $d=3$. Repeating similar steps, we get both $\boldsymbol p \cdot
\hat{\boldsymbol f}(\kappa \boldsymbol d) = 0$ and $\boldsymbol q \cdot
\hat{\boldsymbol f}(\kappa \boldsymbol d) = 0$.
Let $\boldsymbol p(\theta, \varphi)=(\sin\theta\cos\varphi,
\sin\theta\sin\varphi, \cos\theta)^\top$. We choose $\boldsymbol q_1(\theta,
\varphi)=(\cos\theta\cos\varphi, \cos\theta\sin\varphi,
-\sin\theta)^\top$ and $\boldsymbol q_2(\theta, \varphi)=\boldsymbol
p(\theta, \varphi)\times\boldsymbol q_1(\theta, \varphi)=(-\sin\varphi,
\cos\varphi, 0)^\top$ for the shear plane wave. It is easy to verify that
$\boldsymbol{p}, \boldsymbol{q}_1, \boldsymbol{q}_2$ are
mutually orthogonal and thus form an orthonormal basis in ${\mathbb{R}^3}$ for
any $\theta\in[0, \pi], \varphi\in[0, 2\pi]$. Using the Pythagorean
theorem yields
\[
|\hat{\boldsymbol{f}}(\kappa\boldsymbol d)|^2 = |\boldsymbol{p} \cdot
\hat{\boldsymbol{f}}(\kappa\boldsymbol d)|^2 + |\boldsymbol{q}_1\cdot
\hat{\boldsymbol{f}}(\kappa\boldsymbol d)|^2 + |\boldsymbol{q}_2
\cdot \hat{\boldsymbol{f}}(\kappa\boldsymbol d)|^2 = 0
\]
for each $\kappa>0$ and $\boldsymbol d \in \mathbb S^2$, which means
$\hat{\boldsymbol{f}} = 0$ and then $\boldsymbol f$ = 0.

\end{proof}

\subsection{Stability with continuous frequency data}

This section discusses the stability from the data with frequency ranging
over a finite interval. Given the Dirichlet data $\boldsymbol u$ on $\Gamma_R$,
$D\boldsymbol u$ can be viewed as the Neumann data. It follows from \eqref{tbc}
that the Neumann data can be computed via the DtN operator $T_{\rm N}$ once the
Dirichlet data is available on $\Gamma_R$. Hence we may just define a boundary
measurement in terms of the Dirichlet data only:
\[
\|\boldsymbol{u}(\cdot,\omega)\|^2_{\Gamma_R}=\int_{\Gamma_R}\left(
|T_{\rm N} \boldsymbol{u}(\boldsymbol x, \omega)|^2 +\omega^2
|\boldsymbol{u}(\boldsymbol x, \omega)|^2 \right){\rm d}\gamma(\boldsymbol x).
\]

Denote a functional space:
\[
\mathbb{F}_M(B_R) =\{\boldsymbol{f}\in H^{m+1}(B_R)^d:
\|\boldsymbol{f}\|_{H^{m+1}(B_R)^d}\leq M, ~{\rm supp}\boldsymbol f=\Omega\},
\]
where $m\geq d$ is an integer and $M>1$ is a constant. Hereafter, the
notation $``a\lesssim b"$ stands for $a\leq Cb$, where $C>0$ is a
generic constant independent of $m, \omega, K, M$, but may change step by step
in the proofs.

The following stability estimate is the main result for Problem \ref{p1}.

\begin{theo}\label{mrn}
Let $\boldsymbol{u}$ be the solution of the scattering problem
\eqref{ne}--\eqref{rc} corresponding to the source $\boldsymbol{f}\in \mathbb
F_M(B_R)$. Then
\begin{equation}\label{fe}
\| \boldsymbol{f}\|^2_{L^2(B_R)^d}\lesssim
\epsilon_1^2+\frac{M^2}{\left(\frac{K^{\frac{2}{3}}|\ln\epsilon_1|^{\frac{1}{4}}
}{(R+1)(6m-6d+3)^3}\right)^{2m-2d+1}},
\end{equation}
where
\[
\epsilon_1=\left(\int_0^K \omega^{d-1}
\|\boldsymbol{u}(\cdot,\omega)\|^2_{\Gamma_R}{\rm d}\omega\right)^{\frac{1}{2}}.
\]
\end{theo}

\begin{rema}
The stability estimate \eqref{fe} consists of two
parts: the data discrepancy and the high frequency tail. The former is of the
Lipschitz type. The latter decreases as $K$ increases which makes the problem
have an almost Lipschitz stability. The result reveals that the problem becomes
more stable when higher frequency data is used.
\end{rema}

\begin{lemm}\label{nfe}
Let $\boldsymbol u$ be the solution of the scattering problem
\eqref{ne}--\eqref{rc} corresponding to the source $\boldsymbol f\in
L^2(B_R)^d$. Then
\[
\|\boldsymbol{f} \|^2_{L^2(B_R)^d}\lesssim\int_0^{\infty}\omega^{d-1}
\|\boldsymbol u(\cdot, \omega)\|^2_{\Gamma_R} {\rm d}\omega.
\]
\end{lemm}

\begin{proof}
Again, we prove the two dimensional case in details, and then
briefly present the proof for the three dimensional case.

(i) Consider $d=2$. Let $\boldsymbol\xi_{\rm p}=
\kappa_{\rm p} \boldsymbol d$ with $|\boldsymbol\xi_{\rm p}|=\kappa_{\rm p}\in
(0, \infty)$. The compressional plane wave in \eqref{npw} can be written as
$\boldsymbol{u}_{\rm p}^{\rm{inc}}(\boldsymbol x) =
\boldsymbol{p} e^{-{\rm i}\boldsymbol\xi_{\rm{p}}\cdot \boldsymbol
x}$. Multiplying the both sides of \eqref{ne} by $\boldsymbol{u}_{\rm
p}^{\rm inc}(\boldsymbol x)$, using the integration by parts over $B_R$, and
noting \eqref{nec}, we obtain
\[
 \int_{B_R}(\boldsymbol p e^{-{\rm i}\boldsymbol\xi_{\rm{p}}\cdot\boldsymbol
x})\cdot\boldsymbol{f}(\boldsymbol x){\rm d}\boldsymbol x=\int_{\Gamma_R}\left(
\boldsymbol{u}_{\rm{p}}^{\rm{inc}}(\boldsymbol x) \cdot
T_{\rm N}\boldsymbol u(\boldsymbol x,\omega)+ \boldsymbol{u}(\boldsymbol
x,\omega)\cdot D \boldsymbol{u}_{\rm{p}}^{\rm{inc}}(\boldsymbol x)\right){\rm
d}\gamma(\boldsymbol x).
\]
A simple calculation yields that
\[
D\boldsymbol u^{\rm inc}_{\rm p}(\boldsymbol x) = -{\rm i}\kappa_{\rm p}
\left(\mu (\boldsymbol{p} \cdot \boldsymbol\nu)\boldsymbol p + (\lambda +
\mu)\boldsymbol\nu\right) e^{{-\rm i}\boldsymbol \xi_{\rm{p}}\cdot\boldsymbol
x},
\]
which gives
\[
 |D\boldsymbol u^{\rm inc}_{\rm p}(\boldsymbol x)|\lesssim\kappa_{\rm p}.
\]
Noting $\mbox{supp}\boldsymbol{f}\subset B_R$, we get
\[
 \int_{B_R}(\boldsymbol p e^{-{\rm i}\boldsymbol \xi_{\rm p}\cdot\boldsymbol
x})\cdot\boldsymbol{f}(\boldsymbol x){\rm d}\boldsymbol x=\boldsymbol
p\cdot\int_{\mathbb{R}^2}\boldsymbol{f}(\boldsymbol x)e^{-{\rm
i}\boldsymbol\xi_{\rm p}\cdot\boldsymbol x}{\rm d}\boldsymbol x=\boldsymbol
p\cdot\hat{\boldsymbol f}(\boldsymbol\xi_{\rm p}).
\]
Combining the above estimates and using the Cauchy--Schwarz inequality yields
\[
|\boldsymbol p\cdot\hat{\boldsymbol f}(\boldsymbol\xi_{\rm p})|^2
\lesssim\int_{\Gamma_R}\left(|T_{\rm N}\boldsymbol{u}(\boldsymbol
x,\omega)|^2+\kappa^2_{\rm{p}}|\boldsymbol{u}(\boldsymbol
x,\omega)|^2\right){\rm d}\gamma(\boldsymbol x).
\]
Hence
\[
 \int_{\mathbb{R}^2}|\boldsymbol p\cdot\hat{\boldsymbol
f}(\boldsymbol\xi_{\rm p})|^2 {\rm d}\boldsymbol\xi_{\rm p}\lesssim
\int_{\mathbb{R}^2}\int_{\Gamma_R}\left(|T_{\rm N}\boldsymbol{u}(\boldsymbol
x,\omega)|^2+\kappa^2_{\rm p}|\boldsymbol{u}(\boldsymbol x,\omega)|^2\right){\rm
d}\gamma(\boldsymbol x) {\rm d}\boldsymbol\xi_{\rm p}.
\]
Using the polar coordinates, we have
\begin{align}\label{lema1_s1}
 \int_{\mathbb{R}^2}|\boldsymbol p\cdot\hat{\boldsymbol
f}(\boldsymbol\xi_{\rm p})|^2{\rm d}\boldsymbol\xi_{\rm p}
&\lesssim\int_0^{2\pi}{\rm d}\theta\int_0^{\infty}\kappa_{\rm
p}\int_{\Gamma_R}\left(|T_{\rm N}\boldsymbol{u}(\boldsymbol x,
\omega)|^2+\kappa^2_{\rm p}|
\boldsymbol{u}(\boldsymbol x,\omega)|^2\right){\rm d}\gamma(\boldsymbol x) {\rm
d}\kappa_{\rm p}\notag\\
 &\leq 2\pi\int_0^{\infty}\kappa_{\rm p}\int_{\Gamma_R}\left(|T_{\rm
N}\boldsymbol{u}(\boldsymbol x,\omega)|^2+\kappa^2_{\rm p}|
\boldsymbol{u}(\boldsymbol x,\omega)|^2\right){\rm d}\gamma(\boldsymbol x){\rm
d}\kappa_{\rm p}\notag\\
 &\lesssim \int_0^{\infty}\omega\int_{\Gamma_R}\left(|T_{\rm
N}\boldsymbol{u}(\boldsymbol x,\omega)|^2 +
\omega^2|\boldsymbol{u}(\boldsymbol x,\omega)|^2\right){\rm
d}\gamma(\boldsymbol x){\rm d}\omega\notag\\
&=\int_0^{\infty}\omega\|\boldsymbol u(\cdot, \omega)\|^2_{\Gamma_R}{\rm
d}\omega.
\end{align}

Let $\boldsymbol\xi_{\rm s} = \kappa_{\rm s} \boldsymbol{d}$ with
$|\boldsymbol\xi_{\rm s}|=\kappa_{\rm s}\in (0, \infty)$. The shear plane wave
in \eqref{npw} can be written as $\boldsymbol u^{\rm inc}_{\rm
s}(\boldsymbol x)=\boldsymbol q e^{-{\rm i}\boldsymbol \xi_{\rm
s}\cdot\boldsymbol x}$. Multiplying $\boldsymbol u^{\rm inc}_{\rm s}$ on both
sides of \eqref{ne}, using the integration by parts, and noting \eqref{nec}, we
may similarly get
\begin{align}\label{lema1_s2}
 \int_{\mathbb{R}^2}|\boldsymbol{q}\cdot
\hat{\boldsymbol{f}}(\boldsymbol\xi_{\rm s})|^2 {\rm d} \boldsymbol\xi_{\rm s}
 &=\int_{\mathbb{R}^2}\left|\int_{\mathbb{R}^2}(\boldsymbol q e^{-{\rm
i}\boldsymbol\xi_{\rm s}\cdot\boldsymbol x})\cdot\boldsymbol{f}(\boldsymbol
x){\rm d}\boldsymbol x\right|^2{\rm d}\boldsymbol\xi_{\rm s}\notag\\
 &\lesssim \int_0^{\infty}\kappa_{\rm s}\int_{\Gamma_R}\left(|T_{\rm
N}\boldsymbol{u}(\boldsymbol x,\omega)|^2 +
\kappa_{\rm s}^2| \boldsymbol{u}(\boldsymbol x,\omega|^2\right){\rm
d}\gamma(\boldsymbol x) {\rm d}\kappa_{\rm s}\notag\\
&\lesssim \int_0^{\infty}\omega\int_{\Gamma_R}\left(|T_{\rm
N}\boldsymbol{u}(\boldsymbol x,\omega)|^2 +
\omega^2|\boldsymbol{u}(\boldsymbol x,\omega)|^2\right){\rm
d}\gamma(\boldsymbol x) {\rm d}\omega\notag\\
&=\int_0^{\infty}\omega\|\boldsymbol u(\cdot,
\omega)\|^2_{\Gamma_R}{\rm d}\omega.
\end{align}

Using the polar coordinates, we deduce that
\begin{align}\label{lema1_s3}
\int_{\mathbb{R}^2}|\boldsymbol{q} \cdot
\hat{\boldsymbol{f}}(\boldsymbol\xi_{\rm s})|^2 {\rm d} \boldsymbol\xi_{\rm s}
&= \int_0^{2\pi}\rm{d}\theta \int_0^{\infty} \kappa_{\rm s}
|\boldsymbol{q}(\theta) \cdot \hat{\boldsymbol{f}}(\kappa_{\rm s}
\boldsymbol{p})|^2 {\rm d}\kappa_{\rm s}\notag \\
&=\int_0^{2\pi}\rm{d}\theta \int_0^{\infty} \kappa_{\rm p}
|\boldsymbol{q}(\theta) \cdot \hat{\boldsymbol{f}}(\kappa_{\rm p}
\boldsymbol{p})|^2 {\rm d}\kappa_{\rm p}= \int_{\mathbb{R}^2}|\boldsymbol{q}
\cdot\hat{\boldsymbol{f}}(\boldsymbol\xi_{\rm p})|^2 {\rm d}\boldsymbol \xi_{\rm
p}.
\end{align}
Let $\boldsymbol p(\theta)=(\cos\theta, \sin\theta)^\top$ and take
$\boldsymbol q(\theta)=(-\sin\theta, \cos\theta)^\top$. Then $\boldsymbol
p(\theta)\cdot\boldsymbol q(\theta)=0$ and they form an orthonormal
basis in $\mathbb R^2$ for any $\theta\in [0, 2\pi]$. Hence we have from the
Pythagorean theorem that
\begin{equation}\label{lema1_s4}
|\hat{\boldsymbol{f}}(\boldsymbol\xi_{\rm p})|^2 = |\boldsymbol{p} \cdot
\hat{\boldsymbol{f}}(\boldsymbol\xi_{\rm p})|^2 + |\boldsymbol{q}\cdot
\hat{\boldsymbol{f}}(\boldsymbol\xi_{\rm p})|^2.
\end{equation}

Noting ${\rm supp}\boldsymbol f\subset B_R$ again, we obtain from the
Parseval theorem and \eqref{lema1_s1}--\eqref{lema1_s4} that
\begin{align*}
\|\boldsymbol f\|^2_{L^2(B_R)^2}&=\|\boldsymbol f\|^2_{L^2(\mathbb
R^2)^2}=\|\hat{\boldsymbol f}\|^2_{L^2(\mathbb R^2)^2}=
\int_{\mathbb{R}^2}|\hat{\boldsymbol{f}}
(\boldsymbol\xi_ { \rm p})|^2 {\rm d} \boldsymbol\xi_{\rm p}\\
&=\int_{\mathbb{R}^2}|\boldsymbol{p} \cdot
\hat{\boldsymbol{f}}(\boldsymbol\xi_{\rm p})|^2 {\rm d}\boldsymbol \xi_{\rm p} +
\int_{\mathbb{R}^2}|\boldsymbol{q}\cdot \hat{\boldsymbol{f}}(\boldsymbol\xi_{\rm
p})|^2 {\rm d}\boldsymbol \xi_{\rm p}\\
&=\int_{\mathbb{R}^2}|\boldsymbol{p} \cdot
\hat{\boldsymbol{f}}(\boldsymbol\xi_{\rm p})|^2 {\rm d}\boldsymbol \xi_{\rm p} +
\int_{\mathbb{R}^2}|\boldsymbol{q} \cdot
\hat{\boldsymbol{f}}(\boldsymbol\xi_{\rm s})|^2 {\rm d} \boldsymbol\xi_{\rm
s}\\
&\lesssim \int_0^{\infty}\omega\|\boldsymbol u(\cdot, \omega)\|^2_{\Gamma_R}
{\rm d}\omega,
\end{align*}
which proves the lemma for the two-dimensional case.

(ii) Consider $d=3$. Repeating similar steps and using the spherical
coordinates, we get
\begin{align}\label{lema1_s5}
\int_{\mathbb{R}^3}|\boldsymbol{p} \cdot
\hat{\boldsymbol{f}}(\boldsymbol \xi_{\rm p})|^2{\rm d}\boldsymbol \xi_{\rm p}&=
\int_{\mathbb{R}^3}\left|\int_{\mathbb{R}^3}(\boldsymbol p e^{-{\rm
i}\boldsymbol\xi_{\rm p}\cdot\boldsymbol
x})\cdot\boldsymbol{f}(\boldsymbol x){\rm d}\boldsymbol x\right|^2{\rm
d}\boldsymbol\xi_{\rm p}\notag\\
&\lesssim\int_0^{2\pi}{\rm d}\theta\int_0^{\pi}{\rm sin}\varphi
{\rm d}\varphi\int_0^{\infty}\kappa_{\rm p}^2\int_{\Gamma_R}\left(|T_{\rm
N}\boldsymbol{u}(\boldsymbol x,\omega)|^2+\kappa^2_{\rm
p}|\boldsymbol{u}(\boldsymbol x,\omega)|^2\right){\rm d}\gamma(\boldsymbol
x){\rm d}\kappa_{\rm p}\notag\\
 &\leq 2\pi^2\int_0^{\infty}\kappa_{\rm p}^2\int_{\Gamma_R}\left(|T_{\rm
N}\boldsymbol{u}(\boldsymbol x,\omega)|^2+\kappa^2_{\rm
p}| \boldsymbol{u}(\boldsymbol x,\omega)|^2\right){\rm d}\gamma(\boldsymbol
x){\rm d}\kappa_{\rm p}\notag\\
&\lesssim \int_0^{\infty}\omega^2\int_{\Gamma_R}\left(|T_{\rm
N}\boldsymbol{u}(\boldsymbol x,\omega)|^2+\omega^2|
\boldsymbol{u}(\boldsymbol x,\omega)|^2\right){\rm d}\gamma(\boldsymbol x)
{\rm d}\omega\notag\\
&=\int_0^{\infty}\omega^2\|\boldsymbol u(\cdot,
\omega)\|^2_{\Gamma_R}{\rm d}\omega.
\end{align}
Let $\boldsymbol p(\theta, \varphi)=(\sin\theta\cos\varphi,
\sin\theta\sin\varphi, \cos\theta)^\top$. We choose $\boldsymbol q_1(\theta,
\varphi)=(\cos\theta\cos\varphi, \cos\theta\sin\varphi,
-\sin\theta)^\top$ and $\boldsymbol q_2(\theta, \varphi)=\boldsymbol
p(\theta, \varphi)\times\boldsymbol q_1(\theta, \varphi)=(-\sin\varphi,
\cos\varphi, 0)^\top$ for the shear plane wave. It is easy to verify that
$\{\boldsymbol{p}, \boldsymbol{q}_1, \boldsymbol{q}_2\}$ are
mutually orthogonal and thus form an orthonormal basis in ${\mathbb{R}^3}$ for
any $\theta\in[0, \pi], \varphi\in[0, 2\pi]$. Using the Pythagorean
theorem yields
\begin{equation}\label{lema1_s6}
|\hat{\boldsymbol{f}}(\boldsymbol \xi_{\rm p})|^2 = |\boldsymbol{p} \cdot
\hat{\boldsymbol{f}}(\boldsymbol \xi_{\rm p})|^2 + |\boldsymbol{q}_1\cdot
\hat{\boldsymbol{f}}(\boldsymbol \xi_{\rm p})|^2 + |\boldsymbol{q}_2
\cdot \hat{\boldsymbol{f}}(\boldsymbol \xi_{\rm p})|^2.
\end{equation}
Following similar arguments as those in \eqref{lema1_s2}--\eqref{lema1_s4}, we
get from \eqref{lema1_s5}--\eqref{lema1_s6} that
\begin{align*}
\|\boldsymbol{f} \|^2_{L^2(B_R)^3} &= \|\boldsymbol{f}
\|^2_{L^2(\mathbb{R}^3)^3}=\int_{\mathbb{R}^3}|\hat{\boldsymbol{f}}
(\boldsymbol\xi_{\rm p})|^2 {\rm d}\boldsymbol\xi_{\rm p}\\
&\lesssim\int_0^{\infty}\omega^2\|\boldsymbol u(\cdot,
\omega)\|^2_{\Gamma_R}{\rm d}\omega,
\end{align*}
which completes the proof.
\end{proof}

For $d=2$, let
\begin{align}
\label{nd2i1}I_1(s)=&\int_0^s
\omega^{3}\int_{\Gamma_R}\left|\int_\Omega\mathbf{G}_ {\rm N}(\boldsymbol
x, \boldsymbol y; \omega)\cdot\boldsymbol{f}(\boldsymbol y){\rm d}\boldsymbol
y\right|^2{\rm d}\gamma(\boldsymbol x){\rm d}\omega,\\
\label{nd2i2}I_2(s)=&\int_0^s \omega\int_{\Gamma_R}\left|\int_\Omega
D_{\boldsymbol x} (\mathbf{G}_{\rm N}(\boldsymbol x, \boldsymbol y;
\omega)\cdot\boldsymbol{f}(\boldsymbol y)){\rm
d}\boldsymbol y\right|^2{\rm d}\gamma(\boldsymbol x){\rm d}\omega.
\end{align}
For $d=3$, let
\begin{align}
\label{nd3i1}I_1(s)=&\int_0^s \omega^{4}\int_{\Gamma_R}\left|\int_\Omega
\mathbf{G}_{\rm N}(\boldsymbol x, \boldsymbol y;
\omega)\cdot\boldsymbol{f}(\boldsymbol y){\rm
d}\boldsymbol y\right|^2{\rm d}\gamma(\boldsymbol x){\rm d}\omega,\\
\label{nd3i2}I_2(s)=&\int_0^s \omega^{2}\int_{\Gamma_R}\left|\int_\Omega
D_{\boldsymbol x} \left(\mathbf{G}_{\rm N}(\boldsymbol x, \boldsymbol y;
\omega)\cdot\boldsymbol{f}(\boldsymbol y)\right){\rm d}\boldsymbol
y\right|^2{\rm d}\gamma(\boldsymbol x){\rm d}\omega.
\end{align}
Denote a sector
\[
\mathcal{V}=\{z\in\mathbb{C}: -\frac{\pi}{4}<{\rm arg}
z<\frac{\pi}{4}\}.
\]
The integrands in \eqref{nd2i1}--\eqref{nd3i2} are analytic functions of the
angular $\omega$. The integrals with respect to $\omega$ can be taken
over any path joining points $0$ and $s$ in $\mathcal{V}$. Thus $I_1(s)$ and
$I_2(s)$ are analytic functions of $s=s_1+{\rm i}s_2\in\mathcal{V}, s_1,
s_2\in\mathbb{R}$.

\begin{lemm}\label{ni}
 Let $\boldsymbol{f}\in H^3(B_R)^d$. For any $s=s_1+{\rm i}s_2\in\mathcal{V}$,
the following estimates hold:
\begin{enumerate}[label=(\roman*)]

\item when $d=3$,
 \begin{align}
 \label{n31} |I_1(s)|&\lesssim
|s|^{5}e^{4c_{\rm s}R|s|}\|\boldsymbol{f}\|_{H^2(B_R)^3}^2,\\
 \label{n32} |I_2(s)|&\lesssim
|s|^{3}e^{4c_{\rm s}R|s|}\|\boldsymbol{f}\|_{H^3(B_R)^3}^2;
 \end{align}

\item  when $d=2$,
\begin{align}
\label{n21}  |I_1(s)|& \lesssim |s|^3  (|s|^{\frac{3}{2}}+|s|^{\frac{1}{2}}+1)^2
e^{4c_{\rm s}R|s|}\|\boldsymbol{f}\|_{H^2(B_R)^2}^2,\\
 \label{n22} |I_2(s)|& \lesssim |s| (|s|^{\frac{3}{2}}+|s|^{\frac{1}{2}}+1)^2
e^{4c_{\rm s}R|s|}\|\boldsymbol{f}\|_{H^3(B_R)^2}^2.
 \end{align}
\end{enumerate}
\end{lemm}

\begin{proof}
We first prove the three-dimensional case and then the corresponding
two-dimensional case.

(i) Consider $d=3$. Recalling \eqref{gtn}, we split the Green tensor
$\mathbf{G}_{\rm N}(\boldsymbol x, \boldsymbol y)$ into two parts:
\[
\mathbf{G}_{\rm N}(\boldsymbol x, \boldsymbol y; \omega)=
\mathbf{G}_1(\boldsymbol x, \boldsymbol y; \omega) + \mathbf{G}_2(\boldsymbol x,
\boldsymbol y; \omega),
\]
where
\begin{align*}
&\mathbf{G}_1(\boldsymbol x, \boldsymbol y; \omega)=  \frac{1}{4 \pi
\mu}\frac{e^{{\rm i}\kappa_{\rm s}|\boldsymbol x- \boldsymbol y|}}{|\boldsymbol
x-\boldsymbol y|}\mathbf{I}_3,\\
&\mathbf{G}_2(\boldsymbol x, \boldsymbol y; \omega)= \frac{1}{4\pi\omega^2}
\nabla_{\boldsymbol x}\nabla_{\boldsymbol x}^{\top}\bigg(\frac{e^{{\rm
i}\kappa_{\rm s}|\boldsymbol x-\boldsymbol y|}}{|\boldsymbol x-\boldsymbol y|} -
\frac{e^{{\rm i}\kappa_{\rm p}|\boldsymbol x-\boldsymbol
y|}}{|\boldsymbol x-\boldsymbol y|}\bigg).
\end{align*}
Let $\omega = st, t\in(0,1)$. Noting \eqref{cj}, we have from a simple
calculation that
\[
|I_1(s)|\lesssim I_{1,1}(s) + I_{1,2}(s),
\]
where
\begin{align*}
I_{1,1}(s) &= \int_0^1 |s|^5 t^4 \int_{\Gamma_R} \bigg|\int_{\Omega}
\mathbf{G}_1(\boldsymbol x, \boldsymbol y; st)\cdot \boldsymbol{f}(\boldsymbol
y){\rm d}\boldsymbol y\bigg|^2 {\rm d}\gamma(\boldsymbol x) {\rm d} t\\
& \lesssim \int_0^1 |s|^5 t^4 \int_{\Gamma_R} \bigg|\int_{\Omega}\frac{e^{{\rm
i}c_{\rm s}st|\boldsymbol x-\boldsymbol y|}}{|\boldsymbol
x-\boldsymbol y|} \mathbf{I}_3\cdot \boldsymbol{f}(\boldsymbol y){\rm
d}\boldsymbol y\bigg|^2 {\rm d}\gamma(\boldsymbol x) {\rm d} t
\end{align*}
and
\begin{align*}
I_{1,2}(s) &= \int_0^1 |s|^5 t^4 \int_{\Gamma_R} \bigg|\int_{\Omega}
\mathbf{G}_2(\boldsymbol x, \boldsymbol y; st)\cdot \boldsymbol{f}(\boldsymbol
y){\rm d}\boldsymbol y\bigg|^2 {\rm d}\gamma(\boldsymbol x) {\rm d} t\\
&\lesssim \int_0^1 |s|^5 t^4 \int_{\Gamma_R} \bigg|\int_{\Omega}
\frac{1}{(st)^2}\nabla_{\boldsymbol y}\nabla_{\boldsymbol
y}^{\top}\bigg(\frac{e^{{\rm
i}c_{\rm s}st|\boldsymbol x-\boldsymbol y|}}{|\boldsymbol x-\boldsymbol y|} -
\frac{e^{{\rm i}c_{\rm p}st|\boldsymbol x-\boldsymbol
y|}}{|\boldsymbol x-\boldsymbol y|}\bigg)\cdot \boldsymbol{f}(\boldsymbol
y){\rm d}\boldsymbol y\bigg|^2 {\rm d}\gamma(\boldsymbol x) {\rm d} t.
\end{align*}
Here we have used
\[
 \nabla_{\boldsymbol x}\nabla_{\boldsymbol x}^{\top}\bigg(\frac{e^{{\rm
i}\kappa_{\rm s}|\boldsymbol x-\boldsymbol y|}}{|\boldsymbol x-\boldsymbol y|} -
\frac{e^{{\rm i}\kappa_{\rm p}|\boldsymbol x-\boldsymbol
y|}}{|\boldsymbol x-\boldsymbol y|}\bigg)=\nabla_{\boldsymbol
y}\nabla_{\boldsymbol y}^{\top}\bigg(\frac{e^{{\rm
i}\kappa_{\rm s}|\boldsymbol x-\boldsymbol y|}}{|\boldsymbol x-\boldsymbol y|} -
\frac{e^{{\rm i}\kappa_{\rm p}|\boldsymbol x-\boldsymbol
y|}}{|\boldsymbol x-\boldsymbol y|}\bigg).
\]

First we estimate $I_{1, 1}(s)$. Noting that ${\rm supp}\boldsymbol f=
\Omega\subset B_R$ and
\[
|e^{{\rm i}c_{\rm s}st |\boldsymbol x-\boldsymbol y|}|\leq e^{2c_{\rm s}R|s|},
\quad\forall~\boldsymbol x\in \Gamma_R, \boldsymbol y\in \Omega,
\]
we have from the Cauchy--Schwarz inequality that
 \begin{align}\label{lema2_s1}
|I_{1,1}(s)|&\lesssim\int_0^1|s|^{5}t^{4}\int_{\Gamma_R}\bigg|\int_{\Omega}
\frac{e^{2c_{\rm s}R|s|}}{|\boldsymbol x-\boldsymbol y|
}|\boldsymbol{f}(\boldsymbol y)|{\rm d}\boldsymbol y\bigg|^2{\rm
d}\gamma(\boldsymbol x){\rm d}t\notag\\
&\lesssim\int_0^1|s|^{5}t^{4}\int_{\Gamma_R}\int_{B_R}
|\boldsymbol{f}(\boldsymbol y)|^2 {\rm
d}\boldsymbol y\int_{\Omega}\frac{e^{4c_{\rm s}R|s|}}{|\boldsymbol x-\boldsymbol
y|^2}{\rm d}\boldsymbol y {\rm d}\gamma(\boldsymbol x){\rm d}t\notag\\
&\lesssim |s|^{5}e^{4c_{\rm s}R|s|}\|\boldsymbol{f}\|_{L^2(B_R)^3}^2.
\end{align}

Next we estimate $I_{1,2}(s)$. For any $c\in\mathbb R$, considering the
following power series
\begin{align*}
\frac{e^{{\rm i}cst|\boldsymbol x-\boldsymbol y|}}{|\boldsymbol
x-\boldsymbol y|} = \frac{1}{|\boldsymbol x-\boldsymbol y|} + {\rm i}(cst)
-\frac{(cst)^2}{2!}|\boldsymbol x-\boldsymbol y| - \frac{{\rm
i}(cst)^3}{3!}|\boldsymbol x-\boldsymbol y|^2
+ \frac{(cst)^4}{4!}|\boldsymbol x-\boldsymbol y|^3 + \cdots,
\end{align*}
we obtain
\begin{align}\label{lema2_s2}
\frac{1}{(st)^2}\nabla_{\boldsymbol y}\nabla_{\boldsymbol
y}^{\top}&\bigg(\frac{e^{{\rm i}c_{\rm s}st|\boldsymbol x-\boldsymbol
y|}}{|\boldsymbol x-\boldsymbol y|} - \frac{e^{{\rm i}c_{\rm p}st|\boldsymbol
x-\boldsymbol y|}}{|\boldsymbol x-\boldsymbol y|}\bigg)=-\frac{1}{2}(c_{\rm
s}^2 - c_{\rm p}^2)\nabla_{\boldsymbol y}\nabla_{\boldsymbol y}^{\top}|\boldsymbol
x-\boldsymbol y| \notag\\
&- \frac{{\rm i}(st)}{3!}(c_{\rm s}^3-c_{\rm p}^3)\nabla_{\boldsymbol
y}\nabla_{\boldsymbol
y}^{\top}|\boldsymbol x-\boldsymbol y|^2  +\frac{(st)^2}{4!}(c_{\rm s}^4 -
c_{\rm p}^4)\nabla_{\boldsymbol y}\nabla_{\boldsymbol
y}^{\top}|\boldsymbol x-\boldsymbol y|^3+\cdots.
\end{align}
Substituting \eqref{lema2_s2} into $I_{1, 2}(s)$ and using the integration by
parts, we have
\begin{align*}
I_{1,2}(s) &= \int_0^1 |s|^5 t^4 \int_{\Gamma_R} \bigg|\int_{\Omega}
\nabla_{\boldsymbol y}\nabla_{\boldsymbol y}^{\top}\Big(\frac{1}{2}(c_{\rm s}^2
-c_{\rm p}^2)|\boldsymbol x-\boldsymbol y|+\cdots\Big)\cdot
\boldsymbol{f}(\boldsymbol y){\rm d}\boldsymbol y\bigg|^2 {\rm
d}\gamma(\boldsymbol x) {\rm d} t\\
&=\int_0^1 |s|^5 t^4 \int_{\Gamma_R} \bigg|\int_{\Omega}
\Big(\frac{1}{2}(c_{\rm s}^2 - c_{\rm p}^2)|\boldsymbol x-\boldsymbol
y|+\cdots\Big)\nabla_{\boldsymbol y}\nabla_{\boldsymbol y}\cdot
\boldsymbol{f}(\boldsymbol y){\rm d}\boldsymbol y\bigg|^2 {\rm
d}\gamma(\boldsymbol x) {\rm d} t.
\end{align*}
Noting $c_{\rm p}<c_{\rm s}$, we have
\begin{align}\label{lema2_s3}
&\bigg|\frac{1}{2}(c_{\rm s}^2 - c_{\rm p}^2)|\boldsymbol x-\boldsymbol y| +
\frac{{\rm
i}(st)}{3!}(c_{\rm s}^3-c_{\rm p}^3)|\boldsymbol x-\boldsymbol y|^2
-\frac{(st)^2}{4!}(c_{\rm p}^4 -c_{\rm s}^4)|\boldsymbol x-\boldsymbol
y|^3+\cdots\bigg|\notag\\
&\leq \frac{1}{2}c_{\rm s}^2|\boldsymbol x-\boldsymbol y| +
\frac{|s|}{3!}c_{\rm s}^3|\boldsymbol
x-\boldsymbol y|^2+\frac{|s|^2}{4!}c_{\rm s}^4|\boldsymbol x-\boldsymbol
y|^3+\cdots\notag\\
&\leq  c_{\rm s}^2|\boldsymbol x-\boldsymbol y|\Big( \frac{1}{2} +
\frac{c_{\rm s}|s||\boldsymbol x-\boldsymbol y|}{3!}+\frac{c_{\rm
s}^2|s|^2|\boldsymbol
x-\boldsymbol y|^2}{4!}+\cdots \Big)\notag\\
&\leq 2 c_{\rm s}^2 R e^{c_{\rm s}|s||\boldsymbol x-\boldsymbol y|}\lesssim  e^{2
c_{\rm s} R|s|}.
\end{align}
Using \eqref{lema2_s3} and the Cauchy--Schwarz inequality gives
\begin{align}\label{lema2_s4}
|I_{1,2}(s)|&\lesssim \int_0^1|s|^{5}t^{4}\int_{\Gamma_R}\Big(\int_{B_R}
|\nabla_{\boldsymbol y}\nabla_{\boldsymbol y}\cdot
\boldsymbol{f}(\boldsymbol y)|^2 {\rm d}\boldsymbol y\Big)
\Big(\int_{\Omega} e^{4 c_{\rm s} R|s|}{\rm d}\boldsymbol y\Big) {\rm
d}\gamma(\boldsymbol x){\rm d}t\notag\\
&\lesssim |s|^{5}e^{4c_{\rm s}R|s|}\|\boldsymbol{f}\|_{H^2(B_R)^3}^2,
\end{align}
Combining \eqref{lema2_s1} and \eqref{lema2_s4} proves \eqref{n31}.

For $I_{2}(s)$, we have from \eqref{tbo}, \eqref{lema2_s2}, and the integrations
by parts that
\[
|I_2(s)|\lesssim I_{2,1}(s)+I_{2,2}(s),
\]
where
\begin{align*}
 I_{2, 1}(s)&=\int_0^1 |s|^3 t^2\int_{\Gamma_R}\bigg|
\int_{\Omega}\nabla_{\boldsymbol x}\left(\mathbf{G}_{\rm N}(\boldsymbol
x,\boldsymbol y)\cdot\boldsymbol{f}(\boldsymbol
y)\right)\cdot\boldsymbol\nu(\boldsymbol x){\rm d}\boldsymbol y\bigg|^2{\rm
d}\gamma(\boldsymbol x){\rm d}t\\
&=\int_0^1 |s|^3 t^2\int_{\Gamma_R}\bigg|\int_{\Omega}\mathbf{G}_{\rm
N}(\boldsymbol x,\boldsymbol y)\cdot \left(\nabla_{\boldsymbol
y}\boldsymbol{f}(\boldsymbol y)\cdot\boldsymbol\nu(\boldsymbol x)\right){\rm
d}\boldsymbol y\bigg|^2{\rm d}\gamma(\boldsymbol x){\rm d}t
\end{align*}
and
\begin{align*}
 I_{2, 2}(s)=&\int_0^1 |s|^3
t^2\int_{\Gamma_R}\bigg|\int_{\Omega}\nabla_{\boldsymbol
x}\cdot\left(\mathbf{G}_{\rm N}(\boldsymbol
x,\boldsymbol y)\cdot\boldsymbol{f}(\boldsymbol
y)\right)\boldsymbol\nu(\boldsymbol x){\rm d}\boldsymbol y \bigg|^2{\rm
d}\gamma(\boldsymbol x){\rm d}t\\
=& \int_0^1 |s|^3 t^2\int_{\Gamma_R}\bigg|\int_{\Omega} \mathbf{G}_{\rm
N}(\boldsymbol x,\boldsymbol y)\cdot
(\nabla_{\boldsymbol y}\cdot\boldsymbol{f}(\boldsymbol
y)) \boldsymbol\nu(\boldsymbol x){\rm d}\boldsymbol y \bigg|^2{\rm
d}\gamma(\boldsymbol x){\rm d}t\\
\lesssim &\int_0^1 |s|^3 t^2\int_{\Gamma_R}\bigg|\int_{\Omega}
\frac{e^{{\rm i}\kappa_{\rm s}|\boldsymbol x-\boldsymbol y|}}{|\boldsymbol
x-\boldsymbol y|}(\nabla_{\boldsymbol y} \cdot
\boldsymbol{f}(\boldsymbol y))\boldsymbol\nu(\boldsymbol x){\rm d}\boldsymbol
y \bigg|^2 {\rm d}\gamma(\boldsymbol x){\rm d}t\\
&\qquad+ \int_0^1 |s|^3 t^2\int_{\Gamma_R}\bigg| \int_{\Omega}
 \Big(\frac{1}{2}(c_{\rm s}^2 - c_{\rm p}^2)|\boldsymbol x-\boldsymbol y|+
\frac{{\rm i}(st)}{3!}(c_{\rm s}^3-c_{\rm p}^3)|\boldsymbol x-\boldsymbol y|^2
\\
&\qquad-\frac{(st)^2}{4!}(c_{\rm p}^4- c_{\rm s}^4)|\boldsymbol x-\boldsymbol
y|^3+\cdots \Big)\nabla_{\boldsymbol y} \cdot \Big(\nabla_{\boldsymbol
y}\nabla_{\boldsymbol y}\cdot\boldsymbol{f}(\boldsymbol
y)\Big)\boldsymbol\nu(\boldsymbol x){\rm
d}\boldsymbol y\bigg|^2  {\rm d}\gamma(\boldsymbol x){\rm d}t.
\end{align*}
Following the similar steps for $I_{1, 1}(s)$ and $I_{1, 2}(s)$, we may
estimate $I_{2, 1}(s)$ and $I_{2, 2}(s)$, respectively, and prove the inequality
\eqref{n32}.

(ii) Consider $d=2$. Similarly, let $\omega = st, t\in(0,1)$ and
\[
 \mathbf{G}_{\rm N}(\boldsymbol x, \boldsymbol y;
\omega)=\mathbf{G}_1(\boldsymbol x, \boldsymbol y;
\omega)+\mathbf{G}_2(\boldsymbol x, \boldsymbol y; \omega)
\]
where
\begin{align*}
&\mathbf{G}_1(\boldsymbol x,\boldsymbol y; \omega) = \frac{\rm
i}{4\mu}H_{0}^{(1)}(\kappa_{\rm s}|\boldsymbol x - \boldsymbol
y|)\mathbf{I}_2,\\
&\mathbf{G}_2(\boldsymbol x,\boldsymbol y; \omega) = \frac{\rm
i}{4\omega^2}\nabla_{\boldsymbol x}\nabla_{\boldsymbol
x}^{\top}\Big(H_{0}^{(1)}(\kappa_{\rm s}|\boldsymbol x - \boldsymbol y|) -
H_{0}^{(1)}(\kappa_{\rm p}|\boldsymbol x - \boldsymbol y|)\Big).
\end{align*}
Noting \eqref{cj}, we get
\begin{align*}
|I_1(s)|\lesssim I_{1,1}(s) + I_{1,2}(s),
\end{align*}
where
\begin{align*}
I_{1,1}(s) &= \int_0^1 |s|^4 t^3 \int_{\Gamma_R} \bigg|\int_{\Omega}
\mathbf{G}_1(\boldsymbol x,\boldsymbol y; st)\cdot \boldsymbol{f}(\boldsymbol
y){\rm d}\boldsymbol y\bigg|^2 {\rm d}\gamma(\boldsymbol x) {\rm d} t \\
& \lesssim \int_0^1 |s|^4 t^3 \int_{\Gamma_R} \bigg|\int_{\Omega}
H_{0}^{(1)}(c_{\rm s}st|\boldsymbol x - \boldsymbol y|) \mathbf{I}_2\cdot
\boldsymbol{f}(\boldsymbol y){\rm d}\boldsymbol y\bigg|^2 {\rm
d}\gamma(\boldsymbol x) {\rm d} t
\end{align*}
and
\begin{align*}
I_{1,2}(s) &= \int_0^1 |s|^4 t^3 \int_{\Gamma_R} \bigg|\int_{\Omega}
\mathbf{G}_2(\boldsymbol x,\boldsymbol y; st)\cdot \boldsymbol{f}(\boldsymbol
y){\rm d}\boldsymbol y\bigg|^2 {\rm d}\gamma(\boldsymbol x) {\rm d} t \\
& \lesssim\int_0^1 |s|^4 t^3 \int_{\Gamma_R} \bigg|\int_{\Omega}
\frac{1}{(st)^2}\nabla_{\boldsymbol y}\nabla_{\boldsymbol
y}^{\top}\big(H_{0}^{(1)}(c_{\rm s}st|\boldsymbol x - \boldsymbol y|) -
H_{0}^{(1)}(c_{\rm p}st|\boldsymbol x - \boldsymbol y|)\big)\cdot
\boldsymbol{f}(\boldsymbol y){\rm d}\boldsymbol y\bigg|^2 {\rm
d}\gamma(\boldsymbol x) {\rm d} t.
\end{align*}
Here we have used
\[
 \nabla_{\boldsymbol x}\nabla_{\boldsymbol
x}^{\top}\big(H_{0}^{(1)}(\kappa_{\rm s}|\boldsymbol x - \boldsymbol y|) -
H_{0}^{(1)}(\kappa_{\rm p}|\boldsymbol x - \boldsymbol
y|)\big)=\nabla_{\boldsymbol y}\nabla_{\boldsymbol
y}^{\top}\big(H_{0}^{(1)}(\kappa_{\rm s}|\boldsymbol x - \boldsymbol y|) -
H_{0}^{(1)}(\kappa_{\rm p}|\boldsymbol x - \boldsymbol y|)\big).
\]

First we estimate $I_{1, 1}(s)$. Recall $H_0^{(1)}(z)=J_0(z) + {\rm
i}Y_0(z)$ and the expansions of $J_0, Y_0$ in \cite{W-22}:
\[
J_0(z)=\sum_{k=0}^{\infty}
\frac{(-1)^k}{4^k} \frac{z^{2k}}{(k!)^2},\quad
 Y_0(z)=\frac{2}{\pi}\left({\rm ln}(\frac{z}{2}) + c_0\right)J_0(z) +
\frac{2}{\pi}\sum_{k=1}^{\infty}(-1)^{k+1}H_k\frac{z^{2k}}{4^k (k!)^2},
\]
where $c_0=0.5772...$ is the Euler--Mascheroni constant and $H_k$ is a
harmonic number defined by
\[
H_k=1+\frac{1}{2}+\cdots+\frac{1}{k}.
\]
It is easy to verify
\begin{equation}\label{fi}
4^k(k!)^2\geq(2k)!,
\end{equation}
which gives
\begin{align*}
|J_0(c_{\rm s} st|\boldsymbol x - \boldsymbol y|)| &=
\left|c_{\rm s}^2\sum_{k=0}^{\infty}(-1)^k\frac{(c_{\rm s}st)^{2k-2}|\boldsymbol
x-\boldsymbol y|^{2k}}{4^k(k!)^2}\right|\\
&\leq c_{\rm s}^2|\boldsymbol x-\boldsymbol
y|^2\sum_{k=0}^{\infty}\frac{(c_{\rm s}|s||\boldsymbol x-\boldsymbol
y|)^{2k}}{(2k)!} \leq c_{\rm s}^2|\boldsymbol x-\boldsymbol y|^2e^{c_{\rm
s}|s||\boldsymbol
x-\boldsymbol y|}\lesssim e^{2 c_{\rm s} R |s|}.
\end{align*}
On the other hand, it can be shown that
\begin{equation}\label{hfi}
\frac{H_k}{4^{k}(k)!^2}\lesssim\frac{1}{(2k)!},
\end{equation}
which yields
\begin{align*}
|Y_0(c_{\rm s} st|\boldsymbol x - \boldsymbol y|)| \lesssim |J_0(c_{\rm s}
st|\boldsymbol x
- \boldsymbol y|)|+ \sum_{k=1}^{\infty}\frac{(c_{\rm s} |s| t|\boldsymbol x -
\boldsymbol y|)^{2k}}{(2k)!}\lesssim e^{2 c_{\rm s} R |s|}.
\end{align*}
It follows from the Cauchy--Schwarz inequality that
\begin{align}\label{lema2_s5}
|I_{1,1}(s)| &\lesssim \int_0^1 |s|^4 t^3 \int_{\Gamma_R} e^{4R c_{\rm s}|s|}
\|\boldsymbol{f}\|^2_{H^2(B_R)^2} {\rm d}\gamma(\boldsymbol x) {\rm d}
t\notag\\
&\lesssim |s|^4  \|\boldsymbol{f}\|^2_{L^2(B_R)^2} \lesssim |s|^3
(|s|^{\frac{3}{2}}+|s|^{\frac{1}{2}}+1)^2
e^{4c_{\rm s}R|s|}\|\boldsymbol{f}\|_{L^2(B_R)^2}^2.
\end{align}

Next we estimate $I_{1,2}(s)$, which requires to evaluate the integral
\[
\left|\int_\Omega\mathbf{G}_2(\boldsymbol x,\boldsymbol y)\cdot
\boldsymbol{f}(\boldsymbol y){\rm d}\boldsymbol y\right|^2,
\]
where
\begin{align*}
\mathbf{G}_2(\boldsymbol x,\boldsymbol y; \omega) =& \frac{\rm i}{4\omega^2}
\nabla_{\boldsymbol x}\nabla_{\boldsymbol x}^{\top}\left(J_{0}(\kappa_{\rm
s}|\boldsymbol x - \boldsymbol y|) - J_{0}(\kappa_{\rm p}|\boldsymbol x -
\boldsymbol y|)\right)\\
& - \frac{1}{4\omega^2} \nabla_{\boldsymbol x}\nabla_{\boldsymbol
x}^{\top}\left(Y_{0}(\kappa_{\rm s}|\boldsymbol x - \boldsymbol y|) -
Y_{0}(\kappa_{\rm p}|\boldsymbol x - \boldsymbol y|)\right).
\end{align*}
Letting $\omega = st$ and using the expansion of $J_0$, we obtain
\begin{align*}
&\frac{1}{(st)^2}\nabla_{\boldsymbol y}\nabla_{\boldsymbol
y}^{\top}\left(J_0(c_{\rm s}st|\boldsymbol x-\boldsymbol y| -
J_0(c_{\rm p}st|\boldsymbol
x-\boldsymbol y|)\right)\\
&=c_{\rm s}^2\sum_{k=1}^{\infty}(-1)^k\frac{(c_{\rm
s}st)^{2k-2}\nabla_{\boldsymbol
y}\nabla_{\boldsymbol y}^{\top}|\boldsymbol x-\boldsymbol
y|^{2k}}{4^k(k!)^2}-c_{\rm p}^2\sum_{k=1}^{\infty}(-1)^k\frac{(c_{\rm
p}st)^{2k-2}\nabla_{
\boldsymbol y}\nabla_{\boldsymbol y}^{\top}|\boldsymbol x-\boldsymbol
y|^{2k}}{4^k(k!)^2}.
\end{align*}
It follows from the integration by parts that
\begin{align}\label{lema2_s6}
&\int_{\Omega}\frac{1}{(st)^2} \Big(\nabla_{\boldsymbol
y}\nabla_{\boldsymbol y}^{\top}\big(J_0(c_{\rm s}st|\boldsymbol x-\boldsymbol
y| - J_0(c_{\rm p}st|\boldsymbol x-\boldsymbol y|)\big)\Big)\cdot
\boldsymbol{f}(\boldsymbol y){\rm d}\boldsymbol y\notag\\
=&\int_{\Omega}\bigg(c_{\rm s}^2\sum_{k=1}^{\infty}(-1)^k\frac{(c_{\rm
s}st)^{2k-2}|\boldsymbol x-\boldsymbol y|^{2k}}{4^k(k!)^2}\bigg)
\nabla_{\boldsymbol y}\nabla_{\boldsymbol y}\cdot\boldsymbol{f}(\boldsymbol
y){\rm d}\boldsymbol y\notag\\
&\quad-\int_{\Omega}\bigg(c_{\rm p}^2\sum_{k=1}^{\infty}(-1)^k\frac{(c_{\rm
p}st)^{2k-2}|\boldsymbol x-\boldsymbol y|^{2k}}{4^k(k!)^2}\bigg)
\nabla_{\boldsymbol y}\nabla_{\boldsymbol y}\cdot\boldsymbol{f}(\boldsymbol
y){\rm d}\boldsymbol y.
\end{align}
Using the inequality \eqref{fi}, we get for any $c>0$ that
\begin{equation}\label{lema2_s7}
\left|c^2\sum_{k=1}^{\infty}(-1)^k\frac{(cst)^{2k-2}|\boldsymbol x-\boldsymbol
y|^{2k}}{4^k(k!)^2}\right|\leq c^2|\boldsymbol x-\boldsymbol
y|^2\sum_{k=0}^{\infty}\frac{(c|s||\boldsymbol
x-\boldsymbol y|)^{2k}}{(2k)!}\leq c^2|\boldsymbol x-\boldsymbol
y|^2e^{c|s||\boldsymbol x-\boldsymbol y|}.
\end{equation}
Combining \eqref{lema2_s6}--\eqref{lema2_s7} and using the Cauchy--Schwarz
inequality, we obtain
\begin{align}\label{lema2_s8}
&\left|\int_{\Omega} \frac{1}{(st)^2}\left(\nabla_{\boldsymbol
y}\nabla_{\boldsymbol y}^{\top}\left(J_0(c_{\rm s}st|\boldsymbol x-\boldsymbol
y| - J_0(c_{\rm p}st|\boldsymbol x-\boldsymbol y|)\right)\right)\cdot
\boldsymbol{f}(\boldsymbol y){\rm d}\boldsymbol y\right|^2\cr
&\lesssim \left(\int_{B_R}\left|\nabla_{\boldsymbol y}\nabla_{\boldsymbol
y}\cdot\boldsymbol{f}(\boldsymbol y)\right|^2{\rm d}\boldsymbol y\right)\left(
\int_{\Omega} c_{\rm s}^4|\boldsymbol x-\boldsymbol y|^4e^{2c_{\rm
s}|s||\boldsymbol
x-\boldsymbol y|}{\rm d}\boldsymbol y \right)\lesssim e^{4R c_{\rm s}|s|}
\|\boldsymbol{f}\|^2_{H^2(B_R)^2}.
\end{align}

Let
\begin{align*}
\frac{1}{(st)^2}\nabla_{\boldsymbol y}\nabla_{\boldsymbol
y}^{\top}\left(Y_0(c_{\rm s}st|\boldsymbol x-\boldsymbol y| -
Y_0(c_{\rm p}st|\boldsymbol
x-\boldsymbol y|)\right)=\mathbf{A}+\mathbf{B},
\end{align*}
where
\begin{align*}
\mathbf{A}=&\frac{2}{\pi}\frac{1}{(st)^2}\nabla_{\boldsymbol
y}\nabla_{\boldsymbol y}^{\top}\left[ \Big( {\rm
ln}(\frac{1}{2}c_{\rm s}st|\boldsymbol x-\boldsymbol y|)+\gamma
\Big)J_0(c_{\rm s}st|\boldsymbol x-\boldsymbol y|)\right]\\
&-\frac{2}{\pi}\frac{1}{(st)^2}\nabla_{\boldsymbol y}\nabla_{\boldsymbol
y}^{\top}\left[\Big( {\rm ln}(\frac{1}{2}c_{\rm p}st|\boldsymbol x-\boldsymbol
y|)+\gamma \Big)J_0(c_{\rm p}st|\boldsymbol x-\boldsymbol y|)\right],\\
\mathbf{B}=&\frac{2}{\pi}\frac{1}{(st)^2} \nabla_{\boldsymbol
y}\nabla_{\boldsymbol
y}^{\top}\sum_{k=1}^{\infty}(-1)^{k+1}H_k\frac{
(c_{\rm s}st|\boldsymbol x-\boldsymbol
y|)^{2k}}{4^k(k!)^2}\\
&-\frac{2}{\pi}\frac{1}{(st)^2}\nabla_{\boldsymbol y}\nabla_{\boldsymbol
y}^{\top}\sum_{k=1}^{\infty}(-1)^{k+1}H_k\frac{(c_{\rm p}st|\boldsymbol
x-\boldsymbol
y|)^{2k}}{4^k(k!)^2}.
\end{align*}
We consider the matrix $\mathbf{B}$ first. Using the
integration by parts yields
\begin{align*}
\int_{\Omega} \mathbf{B} \cdot \boldsymbol{f}(\boldsymbol y){\rm
d}\boldsymbol y
=& \int_{\Omega} \frac{2}{\pi} c_{\rm s}^2 |\boldsymbol x-\boldsymbol y|^2
\sum_{k=0}^{\infty}(-1)^{k+2}H_{k+1}\frac { (c_{\rm s}st|\boldsymbol
x-\boldsymbol
y|)^{2k}}{4^{k+1}((k+1)!)^2}
\nabla_{\boldsymbol y}\nabla_{\boldsymbol
y}\cdot\boldsymbol{f}(\boldsymbol y){\rm d}\boldsymbol y\\
&- \int_{\Omega} \frac{2}{\pi} c_{\rm p}^2 |\boldsymbol x-\boldsymbol y|^2
\sum_{k=0}^{\infty}(-1)^{k+2}H_{k+1}\frac { (c_{\rm p}st|\boldsymbol
x-\boldsymbol
y|)^{2k}}{4^{k+1}((k+1)!)^2}
\nabla_{\boldsymbol y}\nabla_{\boldsymbol y}\cdot\boldsymbol{f}(\boldsymbol
y){\rm d}\boldsymbol y.
\end{align*}
It is easy to verify from \eqref{hfi} that
\[
\frac{H_{k+1}}{4^{k+1}((k+1)!)^2}\lesssim\frac{1}{
(2k)!},
\]
which gives for any $c>0$ that
\[
\left|\sum_{k=0}^{\infty}(-1)^{k+2}H_{k+1}\frac{(cst|\boldsymbol
x-\boldsymbol y|)^{2k}}{4^{k+1}((k+1)!)^2}\right|
\lesssim\sum_{k=0}^{\infty}\frac{(c|s||\boldsymbol x-\boldsymbol
y|)^{2k}}{(2k)!}\lesssim e^{c|s||\boldsymbol x-\boldsymbol y|}.
\]
Using the Cauchy--Schwarz inequality and noting $c_{\rm p}<c_{\rm s}$, we have
\begin{align}\label{lema2_s9}
\left|\int_{\Omega} \mathbf{B} \cdot \boldsymbol{f}(\boldsymbol y){\rm
d}\boldsymbol y\right|& \lesssim \Big(\int_{B_R}\left|\nabla_{\boldsymbol
y}\nabla_{\boldsymbol y}\cdot\boldsymbol{f}\right|^2{\rm d}\boldsymbol y\Big)
\Big(\int_{\Omega} |\boldsymbol x-\boldsymbol y|^4e^{2c_{\rm s}|s||\boldsymbol
x-\boldsymbol y|}{\rm d}\boldsymbol y\Big)\notag\\
&\lesssim e^{4Rc_{\rm s}|s|}\|\boldsymbol{f}\|^2_{H^2(B_R)^2}.
\end{align}

Now we consider the matrix $\mathbf{A}$. Using the identity for any two
smooth functions $l$ and $h$:
\begin{align*}
\nabla_{\boldsymbol y}\nabla_{\boldsymbol y}^{\top}l(\boldsymbol y)h(\boldsymbol
y)=h(\boldsymbol y)\nabla_{\boldsymbol y}\nabla_{\boldsymbol
y}^{\top}l(\boldsymbol y)+l(\boldsymbol y)\nabla_{\boldsymbol
y}\nabla_{\boldsymbol y}^{\top}h(\boldsymbol y)
+\nabla_yl(\boldsymbol y)^{\top}\nabla_{\boldsymbol y} h(\boldsymbol
y)+\nabla_{\boldsymbol y}l(\boldsymbol y)^{\top}\nabla_{\boldsymbol
y}h(\boldsymbol y)
\end{align*}
and
\[
{\rm ln}\Big(\frac{1}{2}cst|\boldsymbol x-\boldsymbol y|\Big)={\rm
ln}\Big(\frac{1}{2}cst\Big) + {\rm ln}(|\boldsymbol x-\boldsymbol y|),
\]
we split $\mathbf{A}$ into three parts:
\[
\mathbf{A} = \mathbf{A}_1 + \mathbf{A}_2 + \mathbf{A}_3,
\]
where
\begin{align*}
\mathbf{A}_1=&\frac{2}{\pi}\frac{1}{(st)^2}\nabla_{\boldsymbol
y}\nabla_{\boldsymbol y}^{\top} \left( {\rm ln}(|\boldsymbol x-\boldsymbol
y|)\right)J_0(c_{\rm s}st|\boldsymbol x-\boldsymbol y|)\\
&-\frac{2}{\pi}\frac{1}{(st)^2}\nabla_{\boldsymbol y}\nabla_{\boldsymbol
y}^{\top}\left( {\rm ln}(|\boldsymbol x-\boldsymbol y|)
\right)J_0(c_{\rm p}st|\boldsymbol x-\boldsymbol y|),\\
\mathbf{A}_2=&\frac{2}{\pi}\frac{1}{(st)^2}\left(\nabla_{\boldsymbol y}
J_0(c_{\rm s}st|\boldsymbol x-\boldsymbol y|)-\nabla_{\boldsymbol y}
J_0(c_{\rm p}st|\boldsymbol x-\boldsymbol
y|)\right)^{\top}\cdot\nabla_{\boldsymbol y}
\left( {\rm ln}(|\boldsymbol x-\boldsymbol y|)\right)\\
&+\frac{2}{\pi}\frac{1}{(st)^2}\left(\nabla_{\boldsymbol y}
J_0(c_{\rm s}st|\boldsymbol x-\boldsymbol y|)-\nabla_{\boldsymbol y}
J_0(c_{\rm p}st|\boldsymbol x-\boldsymbol
y|)\right)^{\top}\cdot\nabla_{\boldsymbol
y}\left( {\rm ln}(|\boldsymbol x-\boldsymbol y|) \right) \\
\mathbf{A}_3=&\frac{2}{\pi}\frac{1}{(st)^2} \Big(
{\rm ln}(\frac{1}{2}c_{\rm s}st|\boldsymbol x-\boldsymbol y|)+\gamma
\Big)\nabla_{\boldsymbol y}\nabla_{\boldsymbol y}^{\top}J_0(c_{\rm
s}st|\boldsymbol
x-\boldsymbol y|)\\
&- \frac{2}{\pi}\frac{1}{(st)^2}\Big( {\rm ln}(\frac{1}{2}c_{\rm
p}st|\boldsymbol x-\boldsymbol y|)+\gamma \Big)\nabla_{\boldsymbol y}\nabla_{\boldsymbol
y}^{\top}J_0(c_{\rm p}st|\boldsymbol x-\boldsymbol y|).
\end{align*}
For $\mathbf{A}_1$, we have from \eqref{fi} and the expansion of
$J_0$ that
\begin{align}\label{lema2_s10}
&\left|\frac{1}{(st)^2}\left(J_0(c_{\rm s}s|\boldsymbol x-\boldsymbol y|) -
J_0(c_{\rm p}st|\boldsymbol x-\boldsymbol y|)\right)\right| \lesssim
|\boldsymbol
x-\boldsymbol y|^2e^{c_{\rm s}|s||\boldsymbol x-\boldsymbol y|}.
\end{align}
Noting ${\rm ln}|\boldsymbol x-\boldsymbol y|$ is analytic when $\boldsymbol
x\in \Gamma_R, \boldsymbol y\in \Omega$, we get from the Cauchy--Schwarz
inequality that
\begin{align}\label{lema2_s11}
\left|\int_{\Omega}\mathbf{A}_1 \cdot\boldsymbol{f}(\boldsymbol
y){\rm d}\boldsymbol y\right|^2 \lesssim e^{4R
c_{\rm s}|s|}\int_{\Omega}\left|\nabla_{\boldsymbol y}\nabla_{\boldsymbol
y}^{\top}
{\rm ln}(|\boldsymbol x-\boldsymbol y|) \cdot \boldsymbol{f}(\boldsymbol
y)\right|^2{\rm d}\boldsymbol y \lesssim e^{4R c_{\rm s}|s|}
\|\boldsymbol{f}\|^2_{L^{2}(B_R)^2}.
\end{align}
For $\mathbf A_2$, using the analyticity of ${\rm ln}|\boldsymbol x-\boldsymbol
y|$ for $\boldsymbol x\in \Gamma_R, \boldsymbol y\in\Omega$, the integration by
parts, and the estimate of \eqref{lema2_s10}, we have
\begin{align}\label{lema2_s12}
\left|\int_{\Omega}\mathbf{A}_2 \cdot \boldsymbol{f}(\boldsymbol
y){\rm d}\boldsymbol y\right|^2 \lesssim e^{4R c_{\rm s}|s|}
\|\boldsymbol{f}\|^2_{H^{1}(B_R)^2}.
\end{align}
Now we consider $\mathbf{A}_3.$ Noting
\[
{\rm ln}\Big(\frac{1}{2}cst|\boldsymbol
x-\boldsymbol y|\Big)+c_0 = {\rm ln}\Big(\frac{1}{2}c|\boldsymbol
x-\boldsymbol y|\Big)+c_0+{\rm ln}(st)
\]
and using the expansion of $J_0$:
\[
\nabla_{\boldsymbol y}\nabla_{\boldsymbol y}^{\top}J_0(cst|\boldsymbol
x-\boldsymbol y|) = \nabla_{\boldsymbol y}\nabla_{\boldsymbol
y}^{\top}\sum_{k=1}^{\infty} \frac{(-1)^k}{4^k} \frac{(cst|\boldsymbol
x-\boldsymbol y|)^{2k}}{(k!)^2},
\]
we have from the integration by parts that
\[
 \int_{\Omega}\mathbf{A}_3\cdot \boldsymbol{f}(\boldsymbol y){\rm d}\boldsymbol
y=\boldsymbol A_{3, 1}+\boldsymbol A_{3, 2},
\]
where
\begin{align*}
 \boldsymbol A_{3, 1}=&\int_{\Omega}\frac{2}{\pi}\frac{1}{(st)^2}
\sum_{k=1}^{\infty}
\frac{(-1)^k}{4^k} \frac{(c_{\rm s}st|\boldsymbol x-\boldsymbol
y|)^{2k}}{(k!)^2}
\nabla_{\boldsymbol y}\nabla_{\boldsymbol y}\cdot\left[ \Big({\rm
ln}(\frac{1}{2}c_{\rm s}|\boldsymbol x-\boldsymbol y|) +
c_0\Big)\boldsymbol{f}(\boldsymbol y)\right]{\rm d}\boldsymbol y\\
&-\int_{\Omega}\frac{2}{\pi}\frac{1}{(st)^2}  \sum_{k=1}^{\infty}
\frac{(-1)^k}{4^k} \frac{(c_{\rm p}st|\boldsymbol x-\boldsymbol
y|)^{2k}}{(k!)^2}\nabla_{\boldsymbol y}\nabla_{\boldsymbol y}\cdot
\left[\Big({\rm ln}(\frac{1}{2}c_{\rm p}|\boldsymbol x-\boldsymbol y|) +
c_0\Big)\boldsymbol{f}(\boldsymbol y)\right]{\rm d}\boldsymbol y
\end{align*}
and
\begin{align*}
\boldsymbol A_{3,2}=&\int_{\Omega}\frac{2}{\pi}\frac{1}{(st)^2} {\rm ln}(st)
J_0(c_{\rm s}st|\boldsymbol x-\boldsymbol y|) \nabla_{\boldsymbol
y}\nabla_{\boldsymbol y}\cdot \boldsymbol{f}(\boldsymbol y){\rm d}\boldsymbol
y\\
&-\int_{\Omega}\frac{2}{\pi}\frac{1}{(st)^2} {\rm ln}(st)
J_0(c_{\rm p}st|\boldsymbol x-\boldsymbol y|) \nabla_{\boldsymbol
y}\nabla_{\boldsymbol y}\cdot \boldsymbol{f}(\boldsymbol y){\rm d}\boldsymbol y.
\end{align*}
Since the function ${\rm ln}(\frac{1}{2}c|\boldsymbol x-\boldsymbol y|) +c_0$ is
analytic for $\boldsymbol y\in\Omega$, $\boldsymbol x\in\Gamma_R$,
we have from \eqref{lema2_s7} and the Cauchy--Schwarz
inequality that
\begin{equation}\label{lema2_s13}
|\boldsymbol A_{3,1}|^2\lesssim e^{4R
c_{\rm s}|s|} \|\boldsymbol{f}\|^2_{H^{2}(B_R)^2}.
\end{equation}
It is easy to verify that
\[
\left|(|s|t)^{\frac{1}{2}}~{\rm ln}(st)\right| \lesssim
(|s|t)^{\frac{1}{2}}~(|s|t+\frac{1}{(|s|t)^{\frac{1}{2}}}+\pi)\lesssim
|s|^{\frac{3}{2}}+\pi|s|^{\frac{1}{2}}+1.
\]
Hence
\begin{align*}
&\left|(|s|t)^{\frac{1}{2}}  {\rm ln}(st) \frac{1}{(st)^2}
\left(J_0(c_{\rm s}st|\boldsymbol x-\boldsymbol y|)-J_0(c_{\rm p}st|\boldsymbol
x-\boldsymbol y|)\right)\right|\\
&\lesssim(|s|^{\frac{3}{2}}+\pi|s|^{\frac{1}{2}}+1)\Bigg|  c_{\rm
s}^2|\boldsymbol
x-\boldsymbol y|^2\sum_{k=0}^{\infty} \frac{(-1)^{k+1}}{4^{k+1}}
\frac{(c_{\rm s}st|\boldsymbol x-\boldsymbol y|)^{2k}}{((k+1)!)^2} \\
&\qquad - c_{\rm p}^2|\boldsymbol
x-\boldsymbol y|^2\sum_{k=0}^{\infty} \frac{(-1)^{k+1}}{4^{k+1}}
\frac{(c_{\rm p}st|\boldsymbol x-\boldsymbol y|)^{2k}}{((k+1)!)^2}\Bigg|\\
&\lesssim c_{\rm s}^2|\boldsymbol x-\boldsymbol y|^2
(|s|^{\frac{3}{2}}+\pi|s|^{\frac{1}{2}}+1) \bigg( \sum_{k=0}^{\infty}
\frac{(c_{\rm s}|s||\boldsymbol x-\boldsymbol y|)^{2k}}{(2k)!} +
\sum_{k=0}^{\infty} \frac{(c_{\rm p}|s||\boldsymbol x-\boldsymbol y|)^{2k}}{(2k)!}\bigg)\\
&\lesssim c_{\rm s}^2|\boldsymbol x-\boldsymbol y|^2
(|s|^{\frac{3}{2}}+\pi|s|^{\frac{1}{2}}+1) e^{c_{\rm s}|s||\boldsymbol
x-\boldsymbol y|}.
\end{align*}
Multiplying $\boldsymbol A_{3,2}$ by $(|s|t)^{\frac{1}{2}}$ and using the
Cauchy--Schwarz inequality, we obtain
\begin{align}
|(|s|t)^{\frac{1}{2}}\boldsymbol A_{3,2}|^2
&\lesssim\left( \int_{B_R}|\nabla_{\boldsymbol
y}\nabla_{\boldsymbol y}\cdot\boldsymbol{f}(\boldsymbol y)|^2{\rm d}\boldsymbol
y\right)\left(\int_{\Omega} c_{\rm s}^4|\boldsymbol x-\boldsymbol y|^4
(|s|^{\frac{3}{2}}+\pi|s|^{\frac{1}{2}}+1)^2 e^{2c_{\rm s}|s||\boldsymbol
x-\boldsymbol y|}{\rm d}\boldsymbol y\right)\cr
\label{lema2_s14}&\lesssim (|s|^{\frac{3}{2}}+\pi|s|^{\frac{1}{2}}+1)^2
e^{2c_{\rm s}|s||\boldsymbol x-\boldsymbol y|}\int_{B_R}|\nabla_{\boldsymbol
y}\nabla_{\boldsymbol y}\cdot\boldsymbol{f}(\boldsymbol y)|^2{\rm d}\boldsymbol
y\cr
&\lesssim (|s|^{\frac{3}{2}}+\pi|s|^{\frac{1}{2}}+1)^2 e^{4R
c_{\rm s}|s|}\|\boldsymbol{f}\|^2_{H^2(B_R)^3}.
\end{align}
Combining \eqref{lema2_s8}--\eqref{lema2_s14}, we obtain
\[
\left|(|s|t)^{\frac{1}{2}} \int_{\Omega} \mathbf{G}_2(\boldsymbol x,\boldsymbol
y; st)\cdot \boldsymbol{f}(\boldsymbol y){\rm d}\boldsymbol y\right|^2 \lesssim
(|s|^{\frac{3}{2}}+|s|^{\frac{1}{2}}+1)^2
e^{4Rc_{\rm s}|s_2|}\|\boldsymbol{f}\|_{H^2(B_R)^2}^2,
\]
which implies
\begin{equation}\label{lema2_s15}
|I_{1,2}(s)| \lesssim |s|^3 (|s|^{\frac{3}{2}}+|s|^{\frac{1}{2}}+1)^2
e^{4Rc_{\rm s}|s_2|}\|\boldsymbol{f}\|_{H^2(B_R)^2}^2.
\end{equation}
Combining \eqref{lema2_s5} and \eqref{lema2_s15} completes the proof of the
inequality \eqref{n21}.

For $I_{2}(s)$, we have from \eqref{tbo} and the integration by parts that
\begin{align*}
\int_{\Omega}& D_{\boldsymbol x} \left(\mathbf{G}_{\rm N}(\boldsymbol
x,\boldsymbol y; \omega)\cdot\boldsymbol{f}(\boldsymbol y)\right){\rm
d}\boldsymbol y= \mu \int_{\Omega} \mathbf{G}_{\rm N}(\boldsymbol x,\boldsymbol
y)\cdot \left(\nabla_{\boldsymbol y}\boldsymbol{f}(\boldsymbol y)\cdot
\boldsymbol\nu(\boldsymbol x)\right){\rm d}\boldsymbol y\\
& + (\lambda + \mu)\int_{\Omega} \frac{\rm i}{4  \mu }H_{0}^{(1)}(\kappa_{\rm
s}|\boldsymbol x - \boldsymbol y|)(\nabla_{\boldsymbol y} \cdot
\boldsymbol{f}(\boldsymbol y))\boldsymbol\nu(\boldsymbol x)~{\rm d}\boldsymbol
y\\
&+(\lambda + \mu)\int_{\Omega} \frac{\rm
i}{4\omega^2}\left(H_0^{(1)}(\kappa_{\rm s}|\boldsymbol
x-\boldsymbol y|)-H_0^{(1)}(\kappa_{\rm p}|\boldsymbol x-\boldsymbol y|)
\right)\nabla_{\boldsymbol y}\cdot (\nabla_{\boldsymbol y}\nabla_{\boldsymbol
y}\cdot\boldsymbol f(\boldsymbol y))\boldsymbol\nu(\boldsymbol x){\rm
d}\boldsymbol y.
\end{align*}
Using the estimates for the integrals involving $\mathbf{G}_1(\boldsymbol
x,\boldsymbol y)$ and $\mathbf{G}_2(\boldsymbol x,\boldsymbol y)$, which we have
obtained for $I_1(s)$, and the Cauchy--Schwarz inequality, we may similarly get
\[
|I_{2}(s)| \lesssim |s| (|s|^{\frac{3}{2}}+|s|^{\frac{1}{2}}+1)^2
e^{4Rc_{\rm s}|s_2|}\|\boldsymbol{f}\|_{H^3(B_R)^2}^2,
\]
which shows \eqref{n22}. The proof is now complete.
\end{proof}

\begin{lemm}\label{nhfe}
Let $\boldsymbol{f}\in\mathbb{F}_M(B_R)$. For any $s\ge 1$, the following
estimate holds:
\[
\int_s^{\infty}\omega^{d-1}\|\boldsymbol
u(\cdot, \omega)\|^2_{\Gamma_R}{\rm d}\omega \lesssim s^{-(2m-2d+1)}\|
\boldsymbol{f}\|^2_{H^{m+1}(B_R)^d}.
\]
\end{lemm}

\begin{proof}
Let
\begin{align*}
&\int_s^{\infty}\omega^{d-1} \|\boldsymbol
u(\cdot, \omega)\|^2_{\Gamma_R}{\rm d}\omega\\
= &\int_s^{\infty} \omega^{d-1} \int_{\Gamma_R}\bigl(|T_{\rm N}
\boldsymbol{u}(\boldsymbol x,\omega)|^2 +\omega^2|\boldsymbol{u}(\boldsymbol x,
\omega)|^2 \bigr){\rm d}\gamma(\boldsymbol x){\rm d}\omega=L_1 + L_2,
\end{align*}
where
\begin{align*}
 L_1=&\int_s^{\infty}\omega^{d+1} \int_{\Gamma_R} |\boldsymbol{u}(\boldsymbol
x, \omega)|^2 {\rm d}\gamma(\boldsymbol x){\rm d}\omega,\\
L_2= &\int_s^{\infty} \omega^{d-1} \int_{\Gamma_R}
|T_{\rm N} \boldsymbol{u}(x,\omega)|^2{\rm d}\gamma(\boldsymbol x){\rm
d}\omega.
\end{align*}

(i) Consider $d=3$. Using \eqref{us} and noting $s\geq 1$,
we have
\[
L_1=\int_s^{\infty}\omega^{4} \int_{\Gamma_R} |\boldsymbol{u}(\boldsymbol x,
\omega)|^2 {\rm d}\gamma(\boldsymbol x){\rm d}\omega\lesssim L_{1, 1}+L_{1, 2},
\]
where
\begin{align*}
L_{1, 1}&=\int_s^{\infty}\omega^{4} \int_{\Gamma_R}
\bigg|\int_{\Omega}\frac{e^{{\rm i}\kappa_{\rm s}|\boldsymbol
x-\boldsymbol y|}}{|\boldsymbol x-\boldsymbol
y|}\mathbf{I}_3\cdot\boldsymbol{f}(\boldsymbol y){\rm d}\boldsymbol y\bigg|^2
{\rm d}\gamma(\boldsymbol x){\rm d}\omega,\\
L_{1, 2}&=\int_s^{\infty} \int_{\Gamma_R}
\bigg|\int_{\Omega}\nabla_{\boldsymbol y}\nabla_{\boldsymbol
y}^{\top}\Big(\frac{e^{{\rm i}\kappa_{\rm s}|\boldsymbol x-\boldsymbol
y|}}{|\boldsymbol x-\boldsymbol y|} -
\frac{e^{{\rm i}\kappa_{\rm p}|\boldsymbol x-\boldsymbol y|}}{|\boldsymbol
x-\boldsymbol y|}\Big)\cdot\boldsymbol{f}(\boldsymbol y){\rm d}\boldsymbol
y\bigg|^2 {\rm d}\gamma(\boldsymbol x){\rm d}\omega.
\end{align*}
Noting $\Omega\subset B_{\hat R}\subset B_R$, using the polar
coordinates $\rho=|\boldsymbol y-\boldsymbol x|$ originated at
$\boldsymbol x$ with respect to $\boldsymbol y$ and the integration by parts,
we obtain
\begin{align*}
L_{1,1}&=\int_s^{\infty} \omega^{4}\int_{\Gamma_R} \bigg|\int_0^{2\pi}{\rm
d}\theta\int_0^{\pi}\sin\varphi{\rm d}\varphi\int_{R-\hat R}^{R+\hat R}e^{{\rm
i}\kappa_s\rho}\mathbf{I}_3\cdot(\boldsymbol{f}\rho) {\rm d}\rho\bigg|^2 {\rm
d}\gamma(\boldsymbol x){\rm d}\omega\\
&=\int_s^{\infty} \omega^{4} \int_{\Gamma_R}\bigg|\int_0^{2\pi}{\rm
d}\theta\int_0^{\pi}\sin\varphi{\rm d}\varphi\int_{R-\hat
R}^{R+\hat R}\frac{e^{{\rm i}\kappa_s\rho}}{({\rm
i}\kappa_s)^m}\mathbf{I}_3\cdot\frac{\partial^m(\boldsymbol{f}
\rho)}{\partial\rho^m} {\rm d}\rho\bigg|^2 {\rm d}\gamma(\boldsymbol x){\rm
d}\omega,
\end{align*}
which gives after using $\kappa_{\rm s} = c_{\rm s}\omega$ that
\begin{align*}
L_{1,1}\lesssim & \int_s^{\infty}\omega^{4} \int_{\Gamma_R}
\bigg|\int_0^{2\pi}{\rm d}\theta\int_0^{\pi}\sin\varphi{\rm
d}\varphi\int_{R-\hat R}^{R+\hat R}\omega^{-m}\\
&\qquad\bigg(\bigg|\sum_{|\boldsymbol \alpha|=m}\partial_{\boldsymbol
y}^{\boldsymbol \alpha}
\boldsymbol{f}\bigg|\rho +m\bigg|\sum_{ |\boldsymbol
\alpha|=m-1}\partial_{\boldsymbol y}^{\boldsymbol \alpha}\boldsymbol{f}\bigg|\bigg){\rm d}\rho\bigg|^2 {\rm
d}\gamma(\boldsymbol x){\rm d}\omega\cr
=& \int_s^{\infty}  \omega^{4}\int_{\Gamma_R}\bigg|\int_0^{2\pi}{\rm
d}\theta\int_0^{\pi}\sin\varphi{\rm d}\varphi\int_{R-\hat R}^{R+\hat R}
\omega^{-m}\\
&\qquad\bigg(\bigg|\sum_{|\boldsymbol \alpha|=m}\partial_{\boldsymbol
y}^{\boldsymbol \alpha}\boldsymbol{f}\bigg|\frac{1}{\rho}
+\bigg|\sum\limits_{|\boldsymbol
\alpha|=m-1}\partial_{\boldsymbol y}^{\boldsymbol \alpha}
\boldsymbol{f}\bigg|\frac{m}{\rho^2}\bigg)\rho^2{\rm d}\rho\bigg|^2 {\rm
d}\gamma(\boldsymbol x){\rm d}\omega\\
\leq& \int_s^{\infty}\omega^{4} \int_{\Gamma_R}
\bigg|\int_0^{2\pi}{\rm d}\theta\int_0^{\pi}\sin\varphi{\rm
d}\varphi\int_{R-\hat R}^{R+\hat R}\omega^{-m}\\
&\qquad\bigg(\bigg|\sum_{|\boldsymbol \alpha|=m}\partial_{\boldsymbol
y}^{\boldsymbol \alpha}\boldsymbol{f}\bigg|\frac{1}{(R-\hat R)}
+\bigg|\sum_{|\boldsymbol
\alpha|=m-1}\partial_{\boldsymbol y}^{\boldsymbol \alpha}
\boldsymbol{f}\bigg|\frac{m}{(R-\hat R)^2}\bigg)\rho^2{\rm d}\rho\bigg|^2
{\rm d}\gamma(\boldsymbol x){\rm d}\omega\\
\leq&\int_s^{\infty}\omega^{4} \int_{\Gamma_R}
\bigg|\int_0^{2\pi}{\rm d}\theta\int_0^{\pi}\sin\varphi{\rm
d}\varphi\int_{0}^{\infty}\omega^{-m}\\
&\qquad\bigg(\bigg|\sum_{|\boldsymbol \alpha|=m}\partial_{\boldsymbol
y}^{\boldsymbol \alpha}\boldsymbol{f}\bigg|\frac{1}{(R-\hat R)}
+\bigg|\sum_{|\boldsymbol
\alpha|=m-1}\partial_{\boldsymbol y}^{\boldsymbol \alpha}
\boldsymbol{f}\bigg|\frac{m}{(R-\hat R)^2}\bigg)\rho^2{\rm d}\rho\bigg|^2
{\rm d}\gamma(\boldsymbol x){\rm d}\omega.
\end{align*}
Changing back to the Cartesian coordinates with respect to $\boldsymbol y$, we
have
\begin{align}\label{lema3_s1}
L_{1,1}\leq& \int_s^{\infty} \int_{\Gamma_R}
\omega^{4}\bigg|\int_{\Omega}\omega^{-m}\notag\\
&\qquad\bigg(\bigg|\sum_{|\boldsymbol \alpha|=m}\partial_{\boldsymbol
y}^{\boldsymbol \alpha}\boldsymbol{f}\bigg|\frac{1}{(R-\hat R)}
+\bigg|\sum_{|\boldsymbol
\alpha|=m-1}\partial_{\boldsymbol y}^{\boldsymbol \alpha}
\boldsymbol{f}\bigg|\frac{m}{(R-\hat R)^2}\bigg){\rm d} \boldsymbol y\bigg|^2
{\rm d}\gamma(\boldsymbol x){\rm d}\omega\notag\\
\lesssim& m\|\boldsymbol{f}\|^2_{H^m(B_R)^3}\int_s^{\infty}\omega^{4-2m}{\rm
d}\omega\notag\\
\lesssim&\left(\frac{m}{2m-5}\right)s^{-(2m-5)}
\|\boldsymbol{f}\|^2_{H^m(B_R)^3}
\lesssim s^{-(2m-5)}\|\boldsymbol{f}\|^2_{H^m(B_R)^3},
\end{align}
where we have used the fact that $m\geq d=3$.

For $L_{1,2}$, it follows from the integration by parts that
\begin{align*}
L_{1,2} &= \int_s^{\infty}\int_{\Gamma_R}
\bigg|\int_{\Omega}\bigg(\frac{e^{{\rm i}\kappa_{\rm
s}|\boldsymbol x-\boldsymbol y|}}{|\boldsymbol x-\boldsymbol y|} - \frac{e^{{\rm
i}\kappa_{\rm p}|\boldsymbol x-\boldsymbol y|}}{|\boldsymbol x-\boldsymbol
y|}\bigg)\nabla_{\boldsymbol y}\nabla_{\boldsymbol y} \cdot
\boldsymbol{f}(\boldsymbol y){\rm d}\boldsymbol y\bigg|^2 {\rm
d}\gamma(\boldsymbol x){\rm d}\omega\\
&\lesssim  \int_s^{\infty}\int_{\Gamma_R}
\bigg|\int_{\Omega}\frac{e^{{\rm i}\kappa_{\rm
s}|\boldsymbol x-\boldsymbol y|}}{|\boldsymbol x-\boldsymbol
y|}\nabla_{\boldsymbol y}\nabla_{\boldsymbol y} \cdot
\boldsymbol{f}(\boldsymbol y){\rm d}\boldsymbol y\bigg|^2 {\rm
d}\gamma(\boldsymbol x){\rm d}\omega\\
&\qquad+ \int_s^{\infty}\int_{\Gamma_R}
\bigg|\int_{\Omega}\frac{e^{{\rm
i}\kappa_{\rm p}|\boldsymbol x-\boldsymbol y|}}{|\boldsymbol x-\boldsymbol
y|}\nabla_{\boldsymbol y}\nabla_{\boldsymbol y} \cdot
\boldsymbol{f}(\boldsymbol y){\rm d}\boldsymbol y\bigg|^2 {\rm
d}\gamma(\boldsymbol x){\rm d}\omega.
\end{align*}
We may follow the same steps as those for \eqref{lema3_s1} to show
\begin{align}\label{lema3_s2}
L_{1, 2} \leq& \int_s^{\infty}
\int_{\Gamma_R}\bigg|\int_{\Omega}\omega^{-(m-2)}\bigg(\bigg|\sum_{|\boldsymbol
\alpha|=m-2}\partial_{\boldsymbol
y}^{\boldsymbol \alpha}(\nabla_{\boldsymbol y}\nabla_{\boldsymbol y} \cdot
\boldsymbol{f})\bigg|\frac{1}{(R-\hat R)}\notag\\
&\qquad+\bigg|\sum_{|\boldsymbol \alpha|=m-3}\partial_{\boldsymbol
y}^{\boldsymbol \alpha} (\nabla_{\boldsymbol y}\nabla_{\boldsymbol y} \cdot
\boldsymbol{f})\bigg|\frac{(m-2)}{(R-\hat R)^2}\bigg){\rm d} \boldsymbol
y\bigg|^2 {\rm d}\gamma(\boldsymbol x){\rm d}\omega\cr
\lesssim& (m-2)\|\boldsymbol{f}\|^2_{H^m(B_R)^3}\int_s^{\infty}\omega^{4-2m}{\rm
d}\omega\cr
\lesssim&\left(\frac{m-2}{2m-5}\right)s^{-(2m-5)}
\|\boldsymbol{f}\|^2_{H^m(B_R)^3} \lesssim
s^{-(2m-5)}\|\boldsymbol{f}\|^2_{H^m(B_R)^3},
\end{align}

Noting $s\geq 1$, using \eqref{tbo} and the integration by parts, we get
\begin{align*}
L_2(s)&=\int_{s}^{\infty}\omega^2\int_{\Gamma_R}\bigg|\int_{\Omega}
D_{\boldsymbol x} \big(\mathbf{G}_{\rm N}(\boldsymbol x,\boldsymbol
y; \omega)\cdot\boldsymbol{f}(\boldsymbol y)\big){\rm d}\boldsymbol y\bigg|^2
{\rm d}\gamma(\boldsymbol x){\rm d}\omega\cr
&\lesssim \int_{s}^{\infty}\omega^2\int_{\Gamma_R}\bigg|\int_{\Omega}
\mathbf{G}_{\rm N}(\boldsymbol x,\boldsymbol y;\omega)\cdot
\big(\nabla_{\boldsymbol
y}\boldsymbol{f}(\boldsymbol y)\cdot\boldsymbol \nu(\boldsymbol x)\big){\rm
d}\boldsymbol y\bigg|^2{\rm d}\gamma(\boldsymbol x){\rm d}\omega\cr
&\quad+\int_{s}^{\infty}\omega^2\int_{\Gamma_R}\bigg|\int_{\Omega}\frac{e^{{\rm
i}\kappa_{\rm s}|\boldsymbol x-\boldsymbol y|}}{|\boldsymbol x-\boldsymbol y|}
\big(\nabla_{\boldsymbol y} \cdot \boldsymbol{f}(\boldsymbol y)\big)
\boldsymbol\nu(\boldsymbol x){\rm d}\boldsymbol y\bigg|^2{\rm
d}\gamma(\boldsymbol x){\rm d}\omega\cr
&\quad+\int_{s}^{\infty}\frac{1}{\omega^2}\int_{\Gamma_R}\bigg|\int_{\Omega}
\bigg(\frac{e^{{\rm i}\kappa_{\rm s}|\boldsymbol x-\boldsymbol y|}}{|\boldsymbol
x-\boldsymbol y|} - \frac{e^{{\rm i}\kappa_{\rm p}|\boldsymbol x-\boldsymbol
y|}}{|\boldsymbol x-\boldsymbol y|}\bigg)
\nabla_{\boldsymbol y} \cdot \big(\nabla_{\boldsymbol
y}\nabla_{\boldsymbol y}\cdot\boldsymbol{f}\big)\boldsymbol\nu(\boldsymbol
x)~{\rm d}\boldsymbol y\bigg|^2{\rm d}\gamma(\boldsymbol x){\rm d}\omega.
\end{align*}
Again, we may follow similar arguments as those for \eqref{lema3_s1} and
\eqref{lema3_s2} to get
\begin{align}\label{lema3_s3}
L_2\lesssim s^{-(2m-5)}\|\boldsymbol{f}\|^2_{H^m(B_R)^3}.
\end{align}
Combining \eqref{lema3_s1}--\eqref{lema3_s3} completes the proof for the three
dimensional case.

(ii) Consider $d=2$. Noting $s\geq1$, we have
\[
 L_1=\int_s^{\infty}  \omega^{3}\int_{\Gamma_R}|\boldsymbol{u}(\boldsymbol x,
\omega)|^2 {\rm d}\gamma(\boldsymbol x){\rm d}\omega\lesssim L_{1,1}+L_{1,2},
\]
where
\begin{align*}
L_{1,1}&=\int_s^{\infty} \int_{\Gamma_R}
\omega^{3}\left|\int_{\Omega}H_{0}^{(1)}(\kappa_{\rm s}|\boldsymbol x -
\boldsymbol y|)\mathbf{I}_2\cdot\boldsymbol{f}(\boldsymbol y){\rm d}\boldsymbol
y\right|^2 {\rm d}\gamma(\boldsymbol x){\rm d}\omega,\\
L_{1,2}&=\int_s^{\infty} \int_{\Gamma_R} \frac{1}{\omega}\left|\int_{\Omega}
\nabla_{\boldsymbol x}\nabla_{\boldsymbol x}^{\top}\left(H_{0}^{(1)}(\kappa_{\rm
s}|\boldsymbol x - \boldsymbol y|) - H_{0}^{(1)}(\kappa_{\rm p}|\boldsymbol x -
\boldsymbol y|)\right)\cdot\boldsymbol{f}(\boldsymbol y){\rm d}\boldsymbol
y\right|^2 {\rm d}\gamma(\boldsymbol x){\rm d}\omega.
\end{align*}
The Hankel function can be expressed by the following integral when $t>0$
(e.g., \cite{W-22}, Chapter VI):
\begin{align*}
H_0^1(t)=\frac{2}{{\rm i}\pi}\int_1^{\infty}e^{{\rm
i}t\tau}(\tau^2-1)^{-\frac{1}{2}}{\rm d}\tau.
\end{align*}
Using the polar coordinates $\rho=|\boldsymbol y-\boldsymbol x|$ originated at
$\boldsymbol x$ with respect to $\boldsymbol y$ and noting $\Omega\subset \hat
R\subset R$, we have
\begin{align*}
L_{1,1}=\int_s^{\infty}\omega^{3} \int_{\Gamma_R} \bigg|\int_0^{2\pi}{\rm
d}\theta\int_{R-\hat R}^{R+\hat R}H_0^{(1)}(\kappa_s\rho)\mathbf{I}
_2\cdot(\boldsymbol {f} \rho){\rm d}\rho\bigg|^2 {\rm d}\gamma(\boldsymbol
x){\rm d}\omega.
\end{align*}
Let
\begin{align}\label{lema3_s4}
W_k(t)=\frac{2}{{\rm i}\pi}\int_1^{\infty}\frac{e^{{\rm
i}t\tau}}{({\rm i}\tau)^k(\tau^2-1)^{\frac{1}{2}}}{\rm d}\tau, \quad
k=1,2,\cdots.
\end{align}
It is easy to verify that
\[
 W_0(t)=H_0^{(1)}(t)\quad\text{and}\quad \frac{{\rm d}
W_k(t)}{{\rm d}t}=W_{k-1}(t), ~t>0, ~k\in\mathbb{N}.
\]
Using the integration by parts yields
\begin{align*}
L_{1,1}=&\int_s^{\infty} \omega^{3}\int_{\Gamma_R} \bigg|\int_0^{2\pi}{\rm
d}\theta\int_{R-\hat R}^{R+\hat
R}\frac{W_1(\kappa_s\rho)}{\kappa_s}\mathbf I_2\cdot\frac{\partial
(\boldsymbol{f}\rho)}{\partial\rho}{\rm d}\rho\bigg|^2 {\rm
d}\gamma(\boldsymbol x){\rm d}\omega\cr
=&\int_s^{\infty} \omega^{3}\int_{\Gamma_R} \bigg|\int_0^{2\pi}{\rm
d}\theta\int_{R-\hat R}^{R+\hat
R}\frac{W_{m+1}(\kappa_s\rho)}{\kappa_s^{m+1}}\mathbf
I_2\cdot\frac{\partial^{m+1}(\boldsymbol{f} \rho) }{\partial \rho^{m+1}}{\rm
d}\rho\bigg|^2 {\rm d}\gamma(\boldsymbol x){\rm d}\omega.
\end{align*}
Consequently,
\begin{align*}
L_{1,1}\lesssim&\int_s^{\infty}\omega^{3} \int_{\Gamma_R}
\bigg|\int_0^{2\pi}{\rm d}\theta\int_{R-\hat R}^{R+\hat R}\bigg|\frac{H_{m+1}
(\kappa_s\rho)}{\omega^{m+1}}\bigg|\bigg|\frac{\partial^{m+1}
(\boldsymbol{f}\rho)}{\partial \rho^{m+1}}\bigg|{\rm
d}\rho\bigg|^2 {\rm d}\gamma(\boldsymbol x){\rm d}\omega\cr
\lesssim&\int_s^{\infty}\omega^{3} \int_{\Gamma_R}
\bigg|\int_0^{2\pi}{\rm d}\theta\int_{R-\hat R}^{R+\hat R}\bigg|\frac{H_{m+1}
(\kappa_{\rm s}\rho)}{\omega^{m+1}}\bigg|\\
&\qquad\bigg(\bigg|\sum_{|\boldsymbol \alpha|=m+1}\partial_{\boldsymbol
y}^{\boldsymbol \alpha}\boldsymbol{f}\bigg|+\bigg|\sum_{|\boldsymbol
\alpha|=m}\partial_{\boldsymbol y}^{\boldsymbol \alpha
}\boldsymbol{f}\bigg|\frac{(m+1)}{\rho}\bigg)\rho{\rm d}\rho\bigg|^2 {\rm
d}\gamma(\boldsymbol x){\rm d}\omega\\
\lesssim &\int_s^{\infty}\omega^{3} \int_{\Gamma_R} \bigg|\int_0^{2\pi}{\rm
d}\theta\int_{R-\hat R}^{R+\hat R}\bigg|\frac{H_{m+1}(\kappa_s\rho)
}{\omega^{m+1}}\bigg|\\
&\qquad\bigg(\bigg|\sum_{|\boldsymbol\alpha|=m+1}\partial_{\boldsymbol
y}^{\boldsymbol \alpha}\boldsymbol{f}
\bigg|+\bigg|\sum_{|\boldsymbol
\alpha|=m}\partial_{\boldsymbol y}^{\boldsymbol \alpha
}\boldsymbol{f}\bigg|\frac{(m+1)}{(R-\hat R)}\bigg)\rho{\rm d}\rho\bigg|^2
{\rm d}\gamma(\boldsymbol x){\rm d}\omega.
\end{align*}
It is easy to note from \eqref{lema3_s4} that there exists a constant $C>0$
such that $|H_n(\kappa\rho)|\leq C$ for $m\ge 1$. Hence,
\begin{align*}
L_{1,1}\lesssim&\int_s^{\infty} \omega^{3}
\int_{\Gamma_R}\bigg|\int_0^{2\pi}{\rm
d}\theta\int_{R-\hat R}^{R+\hat R}\omega^{-(m+1)}\\
&\qquad\bigg(\bigg|\sum_{|\boldsymbol \alpha|=m+1}\partial_{\boldsymbol
y}^{\boldsymbol \alpha}\boldsymbol{f}\bigg|
+\bigg|\sum_{|\boldsymbol \alpha|=m}\partial_{\boldsymbol
y}^{\boldsymbol \alpha}\boldsymbol{f}\bigg|\frac{(m+1)}{(R-\hat
R)}\bigg)\rho{\rm d}\rho\bigg|^2 {\rm d}\gamma(\boldsymbol x){\rm d}\omega.
\end{align*}
Changing back to the Cartesian coordinates with respect to $\boldsymbol y$, we
have
\begin{align}\label{lema3_s5}
L_{1,1}\lesssim&\int_s^{\infty} \omega^{3}\int_{\Gamma_R}
\bigg|\int_{\Omega}\omega^{-(m+1)}\bigg(\bigg|\sum_{|\alpha|=m+1}\partial_{\boldsymbol
y}^{\boldsymbol \alpha}\boldsymbol{f}\bigg|
+\bigg|\sum\limits_{|\boldsymbol
\alpha|=m}\partial_{\boldsymbol y}^{\boldsymbol
\alpha}\boldsymbol{f}\bigg|\frac{(m+1)}{(R-\hat R)}
\bigg){\rm d}\boldsymbol y\bigg|^2 {\rm d}\gamma(\boldsymbol x){\rm
d}\omega\notag\\
&\lesssim (m+1)\|\boldsymbol{f}\|^2_{H^{m+1}(B_R)^2}\int_s^{\infty}\omega^{1-2m}{ \rm
d}\omega\notag\\
&=\left(\frac{m+1}{2m-2}\right)s^{-(2m-2)}\|\boldsymbol{f}
\|_{H^{m+1}(B_R)^2}\lesssim s^{-(2m-2)}\|\boldsymbol{f}\|^2_
{H^{m+1}(B_R)^2}.
\end{align}

Using the integration by parts yields
\begin{align*}
 L_{1,2}&=\int_s^{\infty} \frac{1}{\omega} \int_{\Gamma_R}\left|\int_{\Omega}
\left(H_{0}^{(1)}(\kappa_{\rm
s}|\boldsymbol x - \boldsymbol y|) - H_{0}^{(1)}(\kappa_{\rm p}|\boldsymbol x -
\boldsymbol y|)\right)\nabla_{\boldsymbol y}\nabla_{\boldsymbol
y}\cdot\boldsymbol{f}(\boldsymbol y){\rm d}\boldsymbol
y\right|^2 {\rm d}\gamma(\boldsymbol x){\rm d}\omega\\
&\lesssim\int_s^{\infty} \frac{1}{\omega} \int_{\Gamma_R}\left|\int_{\Omega}
H_{0}^{(1)}(\kappa_{\rm s}|\boldsymbol x - \boldsymbol y|) \nabla_{\boldsymbol
y}\nabla_{\boldsymbol y}\cdot\boldsymbol{f}(\boldsymbol y){\rm d}\boldsymbol
y\right|^2 {\rm d}\gamma(\boldsymbol x){\rm d}\omega\\
&\qquad+ \int_s^{\infty} \frac{1}{\omega} \int_{\Gamma_R}\left|\int_{\Omega}
H_{0}^{(1)}(\kappa_{\rm p}|\boldsymbol x -
\boldsymbol y|)\nabla_{\boldsymbol y}\nabla_{\boldsymbol
y}\cdot\boldsymbol{f}(\boldsymbol y){\rm d}\boldsymbol
y\right|^2 {\rm d}\gamma(\boldsymbol x){\rm d}\omega
\end{align*}
We may follow a similar proof for \eqref{lema3_s5} to show that
\begin{align}\label{lema3_s6}
 L_{1,2}\lesssim&\int_s^{\infty} \frac{1}{\omega}\int_{\Gamma_R}
\bigg|\int_{\Omega}\omega^{-(m-1)}
\bigg(\bigg|\sum_{|\alpha|=m-1}\partial_{\boldsymbol
y}^{\boldsymbol \alpha}( \nabla_{\boldsymbol
y}\nabla_{\boldsymbol y}\cdot\boldsymbol{f})\bigg|\notag\\
&\qquad+\bigg|\sum_{|\boldsymbol
\alpha|=m-2}\partial_{\boldsymbol y}^{\boldsymbol \alpha}( \nabla_{\boldsymbol
y}\nabla_{\boldsymbol y}\cdot\boldsymbol{f})\bigg|\frac{(m-1)}{(R-\hat R)}
\bigg){\rm d}\boldsymbol y\bigg|^2 {\rm d}\gamma(\boldsymbol x){\rm
d}\omega\cr
&\lesssim
(m-1)\|\boldsymbol{f}\|^2_{H^{m+1}(B_R)^2}\int_s^{\infty}\omega^{1-2m}{\rm
d}\omega\notag\\
&=\left(\frac{m-1}{2m-2}\right)s^{-(2m-2)}\|\boldsymbol{f}
\|_{H^{m+1}(B_R)^2}\lesssim s^{-(2m-2)}\|\boldsymbol{f}\|^2_
{H^{m+1}(B_R)^2}.
\end{align}

Next is to consider $L_2.$ Again, we use \eqref{tbo} and the integration by
parts to get
\begin{align*}
L_2(s)&=\int_{s}^{\infty}\omega\int_{\Gamma_R}\left|\int_{\Omega}
D_{\boldsymbol x}\left(\mathbf{G}_{\rm N}(\boldsymbol x,\boldsymbol
y)\cdot\boldsymbol{f}(\boldsymbol y)\right){\rm d}\boldsymbol y\right|^2 {\rm
d}\gamma(\boldsymbol x){\rm d}\omega\cr
&\lesssim \int_{s}^{\infty}\omega\int_{\Gamma_R}\left|\int_{\Omega}
\mathbf{G}_{\rm N}(\boldsymbol x,\boldsymbol y)\cdot \left(\nabla_{\boldsymbol
y}\boldsymbol{f}\cdot \boldsymbol\nu(\boldsymbol x)\right){\rm d}\boldsymbol
y\right|^2{\rm d}\gamma(\boldsymbol x){\rm d}\omega\cr
&\quad+ \int_{s}^{\infty}\omega\int_{\Gamma_R}\left|\int_{\Omega}
H_{0}^{(1)}(\kappa_{\rm s}|\boldsymbol x -\boldsymbol y|)  (\nabla_{\boldsymbol
y} \cdot \boldsymbol{f}(\boldsymbol y))\boldsymbol \nu(\boldsymbol x){\rm
d}\boldsymbol y\right|^2{\rm d}\gamma(\boldsymbol x){\rm d}\omega\cr
&\quad+\int_{s}^{\infty}\frac{1}{\omega^3}\int_{\Gamma_R}\left|\int_{\Omega}
\left(H_{0}^{(1)}(\kappa_{\rm s}|\boldsymbol x - \boldsymbol y|) -
H_{0}^{(1)}(\kappa_{\rm p}|\boldsymbol x -\boldsymbol y|)\right)\right.
\left.\nabla_{\boldsymbol y} \cdot \left(\nabla_{\boldsymbol
y}\nabla_{\boldsymbol y}\cdot\boldsymbol{f}\right)\boldsymbol\nu(\boldsymbol
x)~{\rm d}\boldsymbol y\right|^2 {\rm d} S(\boldsymbol x){\rm d}\omega.
\end{align*}
Following similar arguments as those for \eqref{lema3_s5} and \eqref{lema3_s6},
we have
\begin{align}\label{lema3_s7}
L_2 \lesssim s^{-(2m-2)}\|\boldsymbol{f}\|^2_{H^{m+1} (B_R)^2}.
\end{align}
Combing \eqref{lema3_s5}--\eqref{lema3_s7} completes the proof for the two
dimensional case.
\end{proof}

\begin{lemm}\label{ni12}
 Let $\boldsymbol{f}\in\mathbb F_M(B_R) $. Then there exists a function
$\beta(s)$ satisfying
\begin{equation} \label{beta}
\begin{cases}
  \beta(s)\geq\frac{1}{2}, \quad & s\in(K, ~ 2^{\frac{1}{4}}K),\\
  \beta(s)\geq \frac{1}{\pi}((\frac{s}{K})^4-1)^{-\frac{1}{2}}, \quad & s\in
(2^{\frac{1}{4}}K, ~\infty),
 \end{cases}
\end{equation}
such that
\[
 |I_1(s)+I_2(s)|\lesssim M^2 e^{(4R+1)c_{\rm
s}s}\epsilon_1^{2\beta(s)},\quad\forall
s\in (K, ~\infty).
\]
\end{lemm}

\begin{proof}
 It follows from Lemma \ref{ni} that
\[
 |\left(I_1(s)+I_2(s)\right)e^{-(4R+1)c_{\rm s}|s|}|\lesssim M^2,\quad\forall
s\in\mathcal{V}.
\]
Recalling \eqref{nd2i1}--\eqref{nd3i2}, we have
\[
 |\left(I_1(s)+I_2(s)\right)e^{-(4R+1)c_{\rm s}s}|\leq\epsilon_1^2,\quad s\in
[0, ~K].
\]
An direct application of Lemma \ref{caf} shows that there exists a function
$\beta(s)$ satisfying \eqref{beta} such that
\[
 |\left(I_1(s)+I_2(s)\right)e^{-(4R+1)c_{\rm s}s}|\lesssim
M^2\epsilon_1^{2\beta},\quad\forall
s\in (K, ~\infty),
\]
which completes the proof.
\end{proof}

Next we prove Theorem \ref{mrn}.

\begin{proof}
We can assume that $\epsilon_1<e^{-1}$, otherwise the estimate is obvious.
Let
\[
s=\begin{cases}
\frac{1}{((4R+3)c_{\rm
s}\pi)^{\frac{1}{3}}}K^{\frac{2}{3}}|\ln\epsilon_1|^{\frac{1}{4}}
,& 2^{\frac{1}{4}}
((4R+3)c_{\rm
s}\pi)^{\frac{1}{3}}K^{\frac{1}{3}}<|\ln\epsilon_1|^{\frac{1}{4}},\\
K, &|\ln\epsilon_1|^{\frac{1}{4}}\leq
2^{\frac{1}{4}}((4R+3)c_{\rm s}\pi)^{\frac{1}{3}}K^{\frac{1}{3}}.
 \end{cases}
\]
If
$2^{\frac{1}{4}}((4R+3)c_{\rm s}\pi)^{\frac{1}{3}}K^{\frac{1}{3}}
<|\ln\epsilon_1|^{\frac{1}{4}}$, then we have from Lemma \ref{ni12} that
\begin{align*}
 |I_1(s)+I_2(s)|&\lesssim M^2 e^{(4R+3)c_{\rm s}s}
e^{-\frac{2|\ln\epsilon_1|}{\pi}((\frac{s}{K})^4-1)^{-\frac{1}{2}}}\cr
&\lesssim M^2 e^{\frac{(4R+3)c_{\rm s}}{((4R+3)c_{\rm
s}\pi)^{\frac{1}{3}}}K^{\frac{2}{3}}|\ln\epsilon_1|^{\frac{1}{4}}-\frac{
2|\ln\epsilon_1|}{\pi}(\frac{K}{s})^2}\\
&=M^2 e^{-2\left(\frac{c_{\rm
s}^2(4R+3)^2}{\pi}\right)^{\frac{1}{3}}K^{\frac{2}{3}}
|\ln\epsilon_1|^{\frac{1}{2}}\left(1-\frac{1}{2}
|\ln\epsilon_1|^{-\frac{1}{4}}\right)}.
\end{align*}
Noting
\[
\frac{1}{2} |\ln\epsilon_1|^{-\frac{1}{4}}<\frac{1}{2},\quad
\left(\frac{(4R+3)^2}{\pi}\right)^{\frac{1}{3}}>1,
\]
we have
\[
 |I_1(s)+I_2(s)| \lesssim M^2 e^{-(c_{\rm
s}K)^{\frac{2}{3}}|\ln\epsilon_1|^{\frac{1}{2}}}.
\]
Using the elementary inequality
\begin{equation}\label{ei}
 e^{-x}\leq \frac{(6m-6d+3)!}{x^{3(2m-2d+1)}}, \quad x>0,
\end{equation}
we get
\begin{equation}\label{theo1_s1}
|I_1(s)+I_2(s)|\lesssim\frac{M^2}{\left(\frac{K^2|\ln\epsilon_1|^{\frac{3}{2}
}}{(6m-6d+3)^3}\right)^{2m-2d+1}}.
\end{equation}
If $|\ln\epsilon_1|^{\frac{1}{4}}\leq
2^{\frac{1}{4}}(((4R+3)\pi)^{\frac{1}{3}}K^{\frac{1}{3}}$, then $s=K$. We have
from \eqref{nd2i1}--\eqref{nd3i2}  that
\[
 |I_1(s)+I_2(s)|\leq \epsilon_1^2.
\]
Here we have used the fact that
\[
I_1(s)+I_2(s)=\int_0^s \omega^{d-1}\|\boldsymbol u(\cdot,
\omega)\|^2_{\Gamma_R}{\rm d}\omega,\quad s>0.
\]
Hence we obtain from Lemma \ref{nhfe} and \eqref{theo1_s1} that
\begin{align*}
 &\int_0^\infty \omega^{d-1}\|\boldsymbol
u(\cdot, \omega)\|^2_{\Gamma_R}{\rm d}\omega\\
&\leq I_1(s)+I_2(s)+\int_s^\infty \omega^{d-1}
\|\boldsymbol u(\cdot, \omega)\|^2_{\Gamma_R}{\rm d}\omega\\
&\lesssim\epsilon_1^2+\frac{M^2}{\left(\frac{
K^2|\ln\epsilon_1|^{\frac{3}{2}}}{(6m-6d+3)^3}
\right)^{2m-2d+1}}+\frac{M^2}{\left(2^{
-\frac{1}{4}}((4R+3)\pi)^{-\frac{1}{3}}K^{\frac{2}{3}}
|\ln\epsilon_1|^{\frac{1}{4}}\right)^{2m-2d+1}}.
\end{align*}
By Lemma \ref{nfe}, we have
\[
 \|\boldsymbol{f}\|^2_{L^2(B_R)^d}\lesssim \epsilon_1^2
+\frac{M^2}{\left(\frac{K^2|\ln\epsilon_1|^{\frac{3}{2}}}{(6m-6d+3)^3}
\right)^{2m-2d+1}}+\frac{M^2}{\left(\frac{K^{\frac{2}
{3}}|\ln\epsilon_1|^{\frac{1}{4}}}{(R+1)(6m-6d+3)^3}\right)^{2m-2d+1}}.
\]
Since $K^{\frac{2}{3}}|\ln\epsilon_1|^{\frac{1}{4}}\leq K^2
|\ln\epsilon_1|^{\frac{3}{2}}$ when $K>1$ and $|\ln\epsilon_1|>1$, we finish the
proof and obtain the stability estimate \eqref{fe}.
\end{proof}

\subsection{Stability with discrete frequency data}

In this section, we discuss the stability at a discrete set of frequencies. Let
us first specify the discrete frequency data. For $\boldsymbol n\in\mathbb
Z^d\setminus\{0\}$, let $n=|\boldsymbol n|$ and define two angular frequencies
\[
\omega_{{\rm p}, n} = \frac{n\pi}{c_{\rm p}R}, \quad \omega_{{\rm
s}, n} = \frac{n\pi}{c_{\rm s}R}.
\]
The corresponding wavenumbers are
\begin{equation}\label{dwn}
\kappa_{{\rm p}, n}=c_{\rm p}\omega_{{\rm p}, n}=\frac{n\pi}{R},\quad
\kappa_{{\rm s}, n}=c_{\rm s}\omega_{{\rm s},
n}=\frac{n\pi}{R}.
\end{equation}
Recall the boundary measurement at continuous frequencies:
\[
\|\boldsymbol{u}(\cdot,\omega)\|^2_{\Gamma_R}=\int_{\Gamma_R}\left(
|T_{\rm N} \boldsymbol{u}(\boldsymbol x, \omega)|^2
+\omega^2 |\boldsymbol{u}(\boldsymbol x, \omega)|^2 \right){\rm
d}\gamma(\boldsymbol x).
\]
Now we define the boundary measurements at discrete frequencies:
\begin{align*}
\|\boldsymbol{u}(\cdot,\omega_{{\rm p},
n})\|^2_{\Gamma_R}&=\int_{\Gamma_R}\left(
|T_{\rm N} \boldsymbol{u}(\boldsymbol x, \omega_{{\rm p}, n})|^2
+n^2 |\boldsymbol{u}(\boldsymbol x, \omega_{{\rm p}, n})|^2 \right){\rm
d}\gamma(\boldsymbol x),\\
\|\boldsymbol{u}(\cdot,\omega_{{\rm s},
n})\|^2_{\Gamma_R}&=\int_{\Gamma_R}\left(
|T_{\rm N} \boldsymbol{u}(\boldsymbol x, \omega_{{\rm s}, n})|^2
+n^2 |\boldsymbol{u}(\boldsymbol x, \omega_{{\rm s}, n})|^2 \right){\rm
d}\gamma(\boldsymbol x).
\end{align*}

Since the discrete frequency data cannot recover the Fourier coefficient of
$\boldsymbol f$ at $\boldsymbol n=0$, i.e., $\hat{\boldsymbol
f}_0=\frac{1}{(2R)^d}\int_{U_R} \boldsymbol f(\boldsymbol x){\rm d}\boldsymbol
x$ is missing, we assume that $\hat{\boldsymbol f}_0=0$. Otherwise we may
replace $\boldsymbol f(\boldsymbol x)$ by $\tilde{\boldsymbol f}(\boldsymbol
x)=\boldsymbol f(\boldsymbol x)-\left(\int_\Omega \boldsymbol f(\boldsymbol
x){\rm d}\boldsymbol x\right)\chi_\Omega(\boldsymbol x)$, where $\chi$ is the
characteristic function, such that $\tilde{\boldsymbol f}$ has a compact support
$\Omega$ and $\int_\Omega \tilde{\boldsymbol f}(\boldsymbol x){\rm
d}\boldsymbol x=0$. In fact, when $\omega=0$, the Navier equation \eqref{ne}
reduces to
\begin{equation}\label{sne}
\mu\Delta\boldsymbol{u}+ (\lambda + \mu)\nabla\nabla\cdot\boldsymbol{u}= \boldsymbol{f}.
\end{equation}
Integrating \eqref{sne} on both sides on $B_R$ and using the
integration by parts, we have
\[
\int_{\Gamma_R}T_{\rm N}\boldsymbol u (\boldsymbol x){\rm d}\gamma
=\int_{B_R}\boldsymbol f(\boldsymbol x){\rm d}\boldsymbol x,
\]
which implies that $\hat{\boldsymbol f}_0$ can be indeed recovered by the data
corresponding to the static Navier equation. Hence we define
\[
\tilde{\mathbb F}_M(B_R)=\{\boldsymbol f\in\mathbb F_M(B_R): \int_\Omega
\boldsymbol f(\boldsymbol x){\rm d}\boldsymbol x=0\}.
\]

\begin{prob}[Discrete frequency data for elastic waves]\label{p2}
Let $\boldsymbol f\in \tilde{\mathbb F}_M(B_R)$. The inverse source problem is
to determine $\boldsymbol f$ from the displacement $\boldsymbol u(\boldsymbol x,
\omega),  \boldsymbol x\in \Gamma _R, \omega\in
(0, \frac{\pi}{c_{\rm p}R}]\cup\cup_{n=1}^N \{\omega_{{\rm p}, n}, \omega_{{\rm
s}, n}\}$, where $1<N\in\mathbb N$.
\end{prob}

The following stability estimate is the main result for Problem \ref{p2}.

\begin{theo}\label{mrnd}
Let $\boldsymbol{u}$ be the solution of the scattering problem
\eqref{ne}--\eqref{rc} corresponding to the source $\boldsymbol{f}\in \tilde{\mathbb
F}_M(B_R)$. Then
\begin{equation}\label{dfe}
\|\boldsymbol{f}\|^2_{L^2(B_R)^d}\lesssim\epsilon_2^2+\frac{M^2}{\left(\frac{N^{
\frac{5}{8}}|\ln\epsilon_3|^{\frac{1}{9}}}{(6m-3d+3)^3}\right)^{2m-d+1}},
\end{equation}
where
\begin{align*}
\epsilon_2&=\left(\sum_{n=1}^N \|\boldsymbol{u}(\cdot,\omega_{{\rm
p}, n})\|^2_{\Gamma_R} + \|\boldsymbol{u}(\cdot,\omega_{{\rm
s}, n})\|^2_{\Gamma_R}\right)^{\frac{1}{2}},\\
\epsilon_3&= \sup_{\omega \in
(0,\frac{\pi}{c_{\rm p}R}]}\left\|\boldsymbol{u}\left(\cdot,\omega
\right)\right\|_{\Gamma_R}.
\end{align*}
\end{theo}

\begin{rema}
The stability estimate \eqref{dfe} for the discrete frequency data is analogous
to the estimate \eqref{fe} for the continuous frequency data. It also consists of the
Lipschitz type data discrepancy and the high frequency tail of the source
function. The stability increases as $N$ increases, i.e., the inverse problem
is more stable when higher frequency data is used.
\end{rema}

The rest of this section is to prove Theorem \ref{mrnd}. Similarly, we consider the
auxiliary functions of compressional and shear plane waves:
\begin{equation}\label{dcsw}
 \boldsymbol u^{\rm inc}_{{\rm p}, n}=\boldsymbol p_{\boldsymbol n} e^{-{\rm
i}\kappa_{{\rm p}, n}\boldsymbol x\cdot\hat{\boldsymbol n}}\quad\text{and}\quad
\boldsymbol u^{\rm inc}_{{\rm s}, n}=\boldsymbol q_{\boldsymbol n} e^{-{\rm
i}\kappa_{{\rm s}, n}\boldsymbol x\cdot\hat{\boldsymbol n}},
\end{equation}
where $\hat{\boldsymbol n}=\boldsymbol n/n$ represents the unit propagation
direction vector and $\boldsymbol p_{\boldsymbol n}, \boldsymbol
q_{\boldsymbol n}$ are unit polarization vectors satisfying $\boldsymbol
p_{\boldsymbol n}=\hat{\boldsymbol n}$ and $\boldsymbol
q_{\boldsymbol n}\cdot\hat{\boldsymbol n}=0$. Substituting \eqref{dwn} into
\eqref{dcsw} yields
\[
 \boldsymbol u^{\rm inc}_{{\rm p}, n}=\boldsymbol p_{\boldsymbol n} e^{-{\rm
i}(\frac{\pi}{R})\boldsymbol x\cdot\boldsymbol n}\quad\text{and}\quad
\boldsymbol u^{\rm inc}_{{\rm s}, n}=\boldsymbol q_{\boldsymbol n} e^{-{\rm
i}(\frac{\pi}{R})\boldsymbol x\cdot\boldsymbol n}.
\]
It is easy to verify that $\boldsymbol u^{\rm inc}_{{\rm p}, n}$ and
$\boldsymbol u^{\rm inc}_{{\rm s}, n}$ satisfy the homogeneous
Navier equation in $\mathbb R^d$:
\begin{equation}\label{necd}
\mu\Delta\boldsymbol{u}^{\rm inc}_{{\rm p}, n} + (\lambda +
\mu)\nabla\nabla\cdot\boldsymbol{u}^{\rm inc}_{{\rm p}, n} +
\omega_{{\rm p}, n}^2\boldsymbol{u}^{\rm inc}_{{\rm p}, n}
=0
\end{equation}
and
\begin{equation}\label{nesd}
\mu\Delta\boldsymbol{u}^{\rm inc}_{{\rm s}, n} + (\lambda +
\mu)\nabla\nabla\cdot\boldsymbol{u}^{\rm inc}_{{\rm s}, n} +
\omega_{{\rm s}, n}^2\boldsymbol{u}^{\rm inc}_{{\rm s}, n}
=0.
\end{equation}

\begin{lemm}\label{nfed}
Let $\boldsymbol u$ be the solution of the scattering problem
\eqref{ne}--\eqref{rc} corresponding to the source $\boldsymbol f\in
L^2(B_R)^d$. For all $\boldsymbol n\in\mathbb Z^d\setminus\{0\}$, the Fourier
coefficients of $\boldsymbol f$ satisfy
\[
|\hat{\boldsymbol f}_{\boldsymbol
n} |^2\lesssim\|\boldsymbol{u}(\cdot,\omega_{{\rm p}, n})\|^2_{\Gamma_R}
+ \|\boldsymbol{u}(\cdot,\omega_{{\rm s}, n})\|^2_{\Gamma_R}.
\]
\end{lemm}

\begin{proof}
(i) First consider $d = 2$. Multiplying the both sides of \eqref{ne} by
$\boldsymbol{u}_{{\rm p}, n}^{\rm inc}(\boldsymbol x)$, using the
integration by parts over $B_R$, and noting \eqref{necd}, we obtain
\[
 \int_{B_R}(\boldsymbol p_{\boldsymbol{n}}  e^{-{\rm
i}(\frac{\pi}{R})\boldsymbol{x}\cdot\boldsymbol n}
)\cdot\boldsymbol{f}(\boldsymbol x){\rm d}\boldsymbol x=\int_{\Gamma_R}\left(
\boldsymbol{u}_{{\rm p}, n}^{\rm{inc}}(\boldsymbol x) \cdot
T_{\rm N}\boldsymbol u(\boldsymbol x,\omega_{{\rm p}, n})+
\boldsymbol{u}(\boldsymbol x,\omega_{{\rm p}, n})\cdot
D \boldsymbol{u}_{{\rm p}, n}^{\rm inc}(\boldsymbol x)\right){\rm
d}\gamma(\boldsymbol x).
\]
A simple calculation yields that
\[
D\boldsymbol u^{\rm inc}_{{\rm p}, n}(\boldsymbol x) =
-{\rm i}n\Big(\frac{\pi}{R}\Big) \left(\mu (\boldsymbol{p}_{\boldsymbol{n}}
\cdot \boldsymbol\nu)\boldsymbol p_{\boldsymbol{n}} + (\lambda +
\mu)\boldsymbol\nu\right)e^{-{\rm
i}(\frac{\pi}{R})\boldsymbol{x}\cdot\boldsymbol n},
\]
which gives
\[
 |D\boldsymbol u^{\rm inc}_{\rm p}(\boldsymbol x)|\lesssim n.
\]
Noting ${\rm supp}\boldsymbol{f}\subset B_R\subset U_R$, we get from Lemma
\ref{fst} that
\begin{align*}
\frac{1}{(2R)^d}\int_{B_R}(\boldsymbol p_{\boldsymbol{n}}  e^{-{\rm
i}(\frac{\pi}{R})\boldsymbol{x}\cdot\boldsymbol
n})\cdot\boldsymbol{f}(\boldsymbol x){\rm d}\boldsymbol x =\boldsymbol
p_{\boldsymbol{n}} \cdot \frac{1}{(2R)^d}
\int_{U_R}\boldsymbol{f}(\boldsymbol x)e^{-{\rm i} (\frac{\pi}{R})\boldsymbol
n\cdot\boldsymbol x}{\rm d}\boldsymbol x = \boldsymbol p_{\boldsymbol{n}} \cdot
\hat{\boldsymbol f}_{\boldsymbol{n}}.
\end{align*}
Combining the above estimates and using the Cauchy--Schwarz inequality yields
\[
|\boldsymbol p_{\boldsymbol{n}} \cdot \hat{\boldsymbol f}_{\boldsymbol{n}}|^2
\lesssim\int_{\Gamma_R}\left(|T_{\rm N}\boldsymbol{u}(\boldsymbol
x,\omega_{{\rm p}, n})|^2+ n^2|\boldsymbol{u}(\boldsymbol
x,\omega_{{\rm p}, n})|^2\right){\rm d}\gamma(\boldsymbol x).
\]
Using $\boldsymbol u^{\rm inc}_{{\rm s}, n}$ and \eqref{nesd}, we may repeat
the above steps and obtain similarly
\[
|\boldsymbol q_{\boldsymbol{n}} \cdot \hat{\boldsymbol f}_{\boldsymbol{n}}|^2
\lesssim\int_{\Gamma_R}\left(|T_{\rm N}\boldsymbol{u}(\boldsymbol
x,\omega_{{\rm s}, n})|^2+n^2|\boldsymbol{u}(\boldsymbol
x,\omega_{{\rm s}, n})|^2\right){\rm d}\gamma(\boldsymbol x).
\]
It follows from the Pythagorean theorem and the above estimates that we get
\begin{align*}
|\hat{\boldsymbol f}_{\boldsymbol n}|^2&
= |\boldsymbol{p}_{\boldsymbol{n}}\cdot\hat{\boldsymbol f}_{\boldsymbol{n}}|^2
+ |\boldsymbol{q}_{\boldsymbol{n}}\cdot\hat{\boldsymbol f}_{\boldsymbol{n}}|^2\\
&\lesssim \int_{\Gamma_R}\left(|T_{\rm N}\boldsymbol{u}(\boldsymbol
x,\omega_{{\rm p}, n})|^2+n^2|\boldsymbol{u}(\boldsymbol
x,\omega_{{\rm p}, n})|^2\right){\rm d}\gamma(\boldsymbol x)\\
&\qquad+ \int_{\Gamma_R}\left(|T_{\rm N}\boldsymbol{u}(\boldsymbol
x,\omega_{{\rm s}, n})|^2+n^2|\boldsymbol{u}(\boldsymbol
x,\omega_{{\rm s}, n})|^2\right){\rm d}\gamma(\boldsymbol x)\\
&= \|\boldsymbol{u}(\cdot,\omega_{{\rm p}, n})\|^2_{\Gamma_R} +
\|\boldsymbol{u}(\cdot,\omega_{{\rm s}, n})\|^2_{\Gamma_R}.
\end{align*}

(ii) Next is to consider $d = 3$. Let $\boldsymbol p_{\boldsymbol
n}=\hat{\boldsymbol n}$. We pick two unit vectors $\boldsymbol q_{1,
\boldsymbol n}$ and $\boldsymbol q_{2, \boldsymbol n}$ such
that $\{\boldsymbol{p}_{\boldsymbol n}, \boldsymbol{q}_{1, {\boldsymbol n}},
\boldsymbol{q}_{2, {\boldsymbol n}}\}$ are mutually orthogonal and form an
orthonormal basis in ${\mathbb{R}^3}$. Thus
\begin{align*}
|\hat{\boldsymbol f}_{\boldsymbol{n}}|^2
= |\boldsymbol{p}_{\boldsymbol{n}}\cdot\hat{\boldsymbol f}_{\boldsymbol{n}}|^2
+ |\boldsymbol{q}_{1, \boldsymbol{n}}\cdot\hat{\boldsymbol
f}_{\boldsymbol{n}}|^2 + |\boldsymbol{q}_{2,
\boldsymbol{n}}\cdot\hat{\boldsymbol f}_{\boldsymbol{n}}|^2.
\end{align*}
Using $\boldsymbol p_{\boldsymbol n}$ as the polarization vector for
$\boldsymbol u^{\rm inc}_{{\rm p}, n}$ and $\boldsymbol q_{1, \boldsymbol n},
\boldsymbol q_{2, \boldsymbol n}$ as the polarization vectors for $\boldsymbol
u^{\rm inc}_{{\rm s}, n}$ in \eqref{dcsw}, we may follow similar arguments for
$d = 2$ and obtain
\begin{align*}
|\hat{\boldsymbol f}_{\boldsymbol n}|^2\lesssim
\|\boldsymbol{u}(\cdot,\omega_{{\rm p}, n})\|^2_{\Gamma_R} +
\|\boldsymbol{u}(\cdot,\omega_{{\rm s}, n})\|^2_{\Gamma_R},
\end{align*}
which completes the proof.
\end{proof}

\begin{lemm}\label{TFE}
Let $\boldsymbol f \in H^{m+1}(B_R)^{d}.$ For any $N_0\in\mathbb N$, the
following estimate holds:
\[
\sum_{n =N_0}^{\infty} |\hat{\boldsymbol f}_{\boldsymbol n}|^2
\lesssim N_0^{-(2m-d+1)}\|\boldsymbol f\|^2_{H^{m+1}(B_R)^{d}}.
\]
\end{lemm}

\begin{proof}
Let $\boldsymbol n=(n_1, \dots, n_d)^\top$ and choose $n_j = {\rm max}\{n_1, \dots,
n_d\}$. Then we have $n^2 \leq d n_j^2$, which implies that $n_j^{-(m+1)}
\leq d^{\frac{m+1}{2}} n^{-(m+1)}$. Let $\boldsymbol f=(f_1, \dots, f_d)^\top$.
Noting ${\rm supp}\boldsymbol f \subset B_R\subset U_R$ and using integrating by
parts, we obtain
\begin{align*}
\left|\int_{B_R}f_1({\boldsymbol x}) e^{-{\rm i} (\frac{\pi}{R})\boldsymbol n\cdot\boldsymbol x}
 {\rm d} \boldsymbol x\right|^2
\lesssim \left|\int_{B_R}n_j^{-(m+1)} e^{-{\rm i}(\frac{\pi}{R})
\boldsymbol n\cdot\boldsymbol x} \partial^{m+1}_{x_j} f_1({\boldsymbol x}) {\rm
d} \boldsymbol x\right|^2\lesssim n^{-2(m+1)}\|\boldsymbol
f\|^2_{H^{m+1}(B_R)^d}.
\end{align*}
Hence
\[
|\hat{\boldsymbol f}_{\boldsymbol n}|^2 \lesssim \left|\int_{B_R} \boldsymbol
f({\boldsymbol x}) e^{-{\rm i} \boldsymbol(\frac{\pi}{R}) n\cdot\boldsymbol x}
{\rm d} \boldsymbol x\right|^2\lesssim
n^{-2(m+1)}\|\boldsymbol f\|^2_{H^{m+1}(B_R)^d}.
\]
It is easy to note that there are at most $O(n^d)$ elements in \{$\boldsymbol n
\in \mathbb Z^d$, $|\boldsymbol n| = n$\}. Combining the above estimates
yields
\begin{align*}
\sum_{n = N_0}^{\infty} |\hat{\boldsymbol f}_{\boldsymbol n}|^2
&\lesssim \left(\sum_{n=N_0}^\infty
n^{d-2(m+1)}\right)\|\boldsymbol
f\|^2_{H^{m+1}(B_R)^d}\\
&\lesssim \left(\int_0^\infty (N_0+t)^{d-2(m+1)}{\rm d}t
\right)\|\boldsymbol f\|^2_{H^{m+1}(B_R)^d}\\
&=\frac{N_0^{-(2m-d+1)}}{(2m-d+1)}\|\boldsymbol
f\|^2_{H^{m+1}(B_R)^d} \lesssim N_0^{-(2m-d+1)}\|\boldsymbol
f\|^2_{H^{m+1}(B_R)^{d}}.
\end{align*}
which completes the proof.
\end{proof}

\begin{lemm}\label{FF}
Let $\boldsymbol u$ be the solution of the scattering problem
\eqref{ne}--\eqref{rc} corresponding to the source $\boldsymbol f\in
L^2(B_R)^d$. For any $\kappa \in (0,\frac{\pi}{R}]$ and $\boldsymbol d\in\mathbb
S^{d-1}$, the following estimate holds:
\begin{align*}
\left|\int_{B_R}\boldsymbol f(\boldsymbol x) e^{-{\rm i} \kappa \boldsymbol x
\cdot \boldsymbol d}\rm d \boldsymbol x\right|^2\lesssim\epsilon_3^2.
\end{align*}
\end{lemm}

\begin{proof}
Taking the compressional plane wave $\boldsymbol u^{\rm inc}_{\rm
p}(\boldsymbol x)=\boldsymbol d e^{-{\rm i} c_{\rm p} \left(\frac{\kappa}{c_{\rm
p}}\right) \boldsymbol x \cdot \boldsymbol d}$ and using similar arguments as
those in Lemma \ref{nfed}, we obtain
\begin{align*}
\left|\boldsymbol d \cdot \int_{B_R}\boldsymbol f(\boldsymbol x) e^{-{\rm i}
\kappa \boldsymbol x \cdot \boldsymbol d}\rm d \boldsymbol x\right|^2 &=
\left|\boldsymbol d \cdot \int_{B_R} \boldsymbol f(\boldsymbol x) e^{-{\rm i} c_{\rm p}
\left(\frac{\kappa}{c_{\rm p}}\right) \boldsymbol x \cdot \boldsymbol
d}{\rm d} \boldsymbol x\right|^2\\
&\lesssim \int_{\Gamma_R}\left(\left|T_{\rm N}\boldsymbol{u}\left(\boldsymbol
x,\frac{\kappa}{c_{\rm p}}\right)\right|^2+\left(\frac{\kappa}{c_{\rm p}}
\right)^2\left|\boldsymbol{u}\left(\boldsymbol
x,\frac{\kappa}{c_{\rm p}}\right)\right|^2\right){\rm d}\gamma(\boldsymbol x).
\end{align*}
Let the shear plane wave be $\boldsymbol u^{\rm inc}_{\rm
s}(\boldsymbol x)=\boldsymbol p e^{-{\rm i} c_{\rm s} \left(\frac{\kappa}{c_{\rm
s}}\right) \boldsymbol x \cdot \boldsymbol d}$, where $\boldsymbol p$ is a unit
vector such that $\boldsymbol d\perp \boldsymbol p.$ We may similarly get
\begin{align*}
\left|\boldsymbol p \cdot \int_{B_R}\boldsymbol f(\boldsymbol x) e^{-{\rm i}
\kappa \boldsymbol x \cdot \boldsymbol d}{\rm d} \boldsymbol x\right|^2 &=
\left|\boldsymbol p \cdot \int_{B_R}\boldsymbol f(\boldsymbol x) e^{-{\rm i} c_{\rm s}
\left(\frac{\kappa}{c_{\rm s}}\right) \boldsymbol x \cdot \boldsymbol
d}{\rm d} \boldsymbol x\right|^2\\
&\lesssim \int_{\Gamma_R}\left(\left|T_{\rm N}\boldsymbol{u}\left(\boldsymbol
x,\frac{\kappa}{c_{\rm s}}\right)\right|^2+\left(\frac{\kappa}{c_{\rm s}}
\right)^2\left|\boldsymbol{u}\left(\boldsymbol
x,\frac{\kappa}{c_{\rm s}}\right)\right|^2\right){\rm d}\gamma(\boldsymbol x).
\end{align*}
Noting $c_{\rm p}<c_{\rm s}$, we have from the Pythagorean theorem that
\begin{align*}
\left|\int_{B_R}\boldsymbol
f(\boldsymbol x) e^{-{\rm i} \kappa \boldsymbol x \cdot \boldsymbol d}{\rm d}
\boldsymbol x\right|^2 &= \left|\boldsymbol d \cdot \int_{B_R} f(\boldsymbol x)
e^{-{\rm i} \kappa \boldsymbol x \cdot \boldsymbol d}\boldsymbol
{\rm d} \boldsymbol x\right|^2 + \left|\boldsymbol p \cdot
\int_{B_R}\boldsymbol f(\boldsymbol x) e^{-{\rm i} \kappa \boldsymbol x \cdot
\boldsymbol d}{\rm d} \boldsymbol x\right|^2\lesssim \epsilon_3^2.
\end{align*}
The proof is the same for the three-dimensional case when we take two
orthonormal polarization vectors $\boldsymbol p_1$ and $\boldsymbol p_2$ such
that $\{\boldsymbol d, \boldsymbol p_1, \boldsymbol p_2\}$ form an orthonormal
basis in $\mathbb R^3$. The details is omitted for brevity.
\end{proof}

\begin{lemm}\label{ni12E}
 Let $\boldsymbol{f}\in\tilde{\mathbb F}_M(B_R) $. Then there exists a function
$\beta(s)$ satisfying
\begin{equation} \label{betad}
\begin{cases}
  \beta(s)\geq\frac{1}{2}, \quad & s\in(\frac{\pi}{R}, ~ 2^{\frac{1}{4}}\frac{\pi}{R}),\\
  \beta(s)\geq \frac{1}{\pi}((\frac{Rs}{\pi})^4-1)^{-\frac{1}{2}}, \quad & s\in
(2^{\frac{1}{4}}\frac{\pi}{R}, ~\infty),
 \end{cases}
\end{equation}
such that
\[
 \left|\int_{B_R}\boldsymbol f(\boldsymbol x)  e^{-{\rm i}
(\frac{\pi}{R})\boldsymbol n \cdot \boldsymbol x}
{\rm d} \boldsymbol x\right|^2\lesssim M^2
e^{2n R}\epsilon_3^{2n\beta(\frac{n\pi}{R})}, \quad\forall\boldsymbol n \in
\mathbb Z^d,\,n>1.
\]
\end{lemm}

\begin{proof}
We fix a propagation direction vector $\boldsymbol d\in\mathbb S^{d-1}$ and
consider those $\boldsymbol n \in \mathbb Z^d$ which are parallel to
$\boldsymbol d$. Define
\[
I(s) = \left|\int_{B_R}\boldsymbol f(\boldsymbol x) e^{-{\rm i} s \boldsymbol d
\cdot \boldsymbol x} {\rm d} \boldsymbol x\right|^2.
\]
It follows from the Cauchy--Schwarz inequality that there exists a positive
constant $C$ depending on $R, d$ such that
\[
I(s)\leq C(R,d) e^{2|s|R} M^2, \quad\forall
s\in\mathcal{V},
\]
which gives
\[
e^{-2|s|R}I(s) \lesssim M^2, \quad \forall s\in\mathcal{V}.
\]
Noting $\int_\Omega \boldsymbol f(\boldsymbol x){\rm d}\boldsymbol x=0$ and
using Lemma \ref{FF}, we have
\[
e^{-2|s|R}\left|\int_{B_R}\boldsymbol f(\boldsymbol x) e^{-{\rm i} s \boldsymbol
d \cdot \boldsymbol x} {\rm d} \boldsymbol x\right|^2 \lesssim
\epsilon_3^2, \quad \forall s\in [0, \frac{\pi}{R}].
\]
Applying Lemma \ref{caf} shows that there exists a function
$\beta(s)$ satisfying \eqref{betad} such that
\[
 |I(s)e^{-2sR}|\lesssim M^2\epsilon_3^{2\beta},\quad\forall
s\in (\frac{\pi}{R}, ~\infty),
\]
which yields that
\[
 |I(s)|\lesssim M^2e^{2sR}\epsilon_3^{2\beta},\quad\forall
s\in (\frac{\pi}{R}, ~\infty).
\]
Noting that the constant $C(R,d)$ does not depend on $\boldsymbol
d$, we have obtained for all $\boldsymbol
n\in\mathbb Z^d, n>1$ that
\begin{align*}
\left|\int_{B_R}\boldsymbol f(\boldsymbol x) e^{-{\rm i}
(\frac{\pi}{R})\boldsymbol n \cdot \boldsymbol x} {\rm d} \boldsymbol x\right|^2
= \left|\int_{B_R}\boldsymbol f(\boldsymbol x) e^{-{\rm i} (\frac{n\pi}{R})
\hat{\boldsymbol n} \cdot \boldsymbol x} {\rm d} \boldsymbol x\right|^2\lesssim
M^2 e^{2 n R}\epsilon_3^{2n\beta(\frac{n\pi}{R}) },
\end{align*}
which completes the proof.
\end{proof}

Next we prove Theorem \ref{mrnd}.

\begin{proof}
Applying Lemma \ref{fst} and Lemma \ref{nfed}, we have
\begin{align*}
\int_{B_R} |\boldsymbol f|^2{\rm d}\boldsymbol x
\lesssim\sum_{n=0}^{N_0} |\hat{\boldsymbol f}_{\boldsymbol n} |^2
+\sum_{n=N_0+1}^\infty |\hat{\boldsymbol f}_{\boldsymbol n} |^2.
\end{align*}
Let
\[
N_0=
 \begin{cases}
  [N^{\frac{3}{4}}|\ln\epsilon_3|^{\frac{1}{9}}], &
N^{\frac{3}{8}}<\frac{1}{2^{\frac{5}{6}}
\pi^{\frac{2}{3}}}|\ln\epsilon_3|^{\frac{1}{9}},\\
N,&N^{\frac{3}{8}}\geq \frac{1}{2^{\frac{5}{6}}
\pi^{\frac{2}{3}}}|\ln\epsilon_3|^{\frac{1}{9}}
 \end{cases}.
\]
Using Lemma \ref{ni12E} leads to
\begin{align*}
 \Bigl|\int_{B_R}\boldsymbol f(\boldsymbol x) e^{-{\rm
i}(\frac{\pi}{R})\boldsymbol n
\cdot \boldsymbol x} {\rm d}\boldsymbol x\Bigr|^2 &\lesssim M^2
e^{2nR}\epsilon_3^{2n\beta}\lesssim M^2
e^{2n R}e^{2n\beta|\ln\epsilon_3|}\\
&\lesssim M^2 e^{2 nR}e^{-\frac{2}{\pi}(n^4-1)^{-\frac{1}{2}}|\ln\epsilon_3|}
\lesssim M^2 e^{2n R-\frac{2}{\pi}n^{-2}|\ln\epsilon_3|}\\
&\lesssim M^2 e^{-\frac{2}{\pi}n^{-2}|\ln\epsilon_3|(1-2\pi n^3|\ln\epsilon_2|^{
-1})},\quad \forall n\in (2^{\frac{1}{4}},~\infty).
\end{align*}
Hence
\begin{equation}\label{c2}
 \Bigl|\int_{B_R}\boldsymbol f(\boldsymbol x) e^{-{\rm i}(\frac{\pi}{R}
)\boldsymbol n \cdot \boldsymbol x} {\rm d}\boldsymbol x\Bigr|^2\lesssim M^2
e^{-\frac{2}{\pi^3}N^{-2}|\ln\epsilon_3|(1-2\pi^4 N^3|\ln\epsilon_3|^{
-1})},\quad\forall n\in (2^{\frac{1}{4}},~N_0\pi].
\end{equation}

If $N^{\frac{3}{8}}<\frac{1}{2^{\frac{5}{6}}
\pi^{\frac{2}{3}}}|\ln\epsilon_3|^{\frac{1}{9}}$, then $2\pi^4
N_0^3|\ln\epsilon_3|^{-1}<\frac{1}{2}$ and
\begin{equation}\label{c3}
 e^{-\frac{2}{\pi^3}\frac{|\ln\epsilon_3|}{N_0^2}}\leq
e^{-\frac{2}{\pi^3}\frac{|\ln\epsilon_3|}{N^{\frac{3}{2}}|\ln\epsilon_3|^{\frac{
2}{9}}}}\leq
e^{-\frac{2}{\pi^3}\frac{|\ln\epsilon_3|^{\frac{7}{9}}}{N^\frac{3}{2 }}}\leq
e^{-\frac{2}{\pi^3}\frac{2^5\pi^4
|\ln\epsilon_3|^{\frac{1}{9}} N^{\frac{9}{4}}}{N^{\frac{3}{2}}}}= e^{-64\pi
|\ln\epsilon_3|^{\frac{1}{9}}N^{\frac{3}{4}}}.
\end{equation}
Combining \eqref{c2} and \eqref{c3}, we obtain
\begin{align*}
  \Bigl|\int_{B_R}\boldsymbol f(\boldsymbol x) e^{-{\rm
i}(\frac{\pi}{R})\boldsymbol n
\cdot \boldsymbol x} {\rm d}\boldsymbol x\Bigr|^2&\lesssim M^2
e^{-\frac{2}{\pi^3}N_0^{-2}|\ln\epsilon_3|(1-2\pi^4 N_0^3|\ln\epsilon_2|^{
-1})}\\
&\lesssim M^2 e^{-\frac{1}{\pi^3}N_0^{-2}|\ln\epsilon_2|}\lesssim M^2 e^{-32\pi
|\ln\epsilon_3|^{\frac{1}{9}} N^{\frac{3}{4}}},\quad\forall n\in
(2^{\frac{1}{4}},~N_0\pi].
\end{align*}
Using \eqref{ei} yields
\[
  \Bigl|\int_{B_R}\boldsymbol f(\boldsymbol x) e^{-{\rm
i}(\frac{\pi}{R})\boldsymbol n \cdot \boldsymbol x} {\rm d}\boldsymbol x\Bigr|^2\lesssim
\frac{M^2}{\left(\frac{|\ln\epsilon_3|^{\frac{1}{3}}N^{\frac{9}{4}}}{
(6m-3d+3)^3}\right)^{2m-d+1}},
\quad n=1, \dots, N_0.
\]
Consequently,
\begin{align*}
& \sum_{n=1}^{N_0}  \Bigl|\int_{B_R}\boldsymbol f(\boldsymbol x)
e^{-{\rm i}(\frac{\pi}{R})\boldsymbol n \cdot \boldsymbol x} {\rm d}\boldsymbol
x\Bigr|^2\lesssim \frac{M^2 N_0}{\left(\frac{|\ln\epsilon_3|^{\frac{1}{3}}N^{\frac{9}{4 }}}{
(6m-3d+3)^3}\right)^{2m-d+1}}\\
&\qquad\lesssim \frac{M^2 N^{\frac{3}{4}}|\ln\epsilon_3|^{\frac{1}{9}}}{\left(\frac{
|\ln\epsilon_3|^{\frac{1}{3}}N^{\frac{9}{4}}}{(6m-3d+3)^3}\right)^{2m-d+1}}\lesssim
\frac{M^2}{\left(\frac{|\ln\epsilon_3|^{\frac{2}{9}}N^{\frac{3}{2}}}{
(6m-3d+3)^3}\right)^{2m-d+1}}\lesssim
\frac{M^2}{\left(\frac{|\ln\epsilon_3|^{\frac{1}{9}}N^{\frac{3}{2}}}{
(6m-3d+3)^3 }\right)^{2m-d+1}}.
\end{align*}
Here we have used that $|\ln\epsilon_3|>1$ when
$N^{\frac{3}{8}}<\frac{1}{2^{\frac{5}{6}}
\pi^{\frac{2}{3}}}|\ln\epsilon_3|^{\frac{1}{9}}$.
If $N^{\frac{3}{8}}<\frac{1}{2^{\frac{5}{6}}
\pi^{\frac{2}{3}}}|\ln\epsilon_3|^{\frac{1}{9}}$, we have
\[
\left(\bigl[|\ln\epsilon_3|^{\frac{1}{9}}N^{\frac{3}{4}}\bigr]
+1\right)^{2m-d+1}\geq \left(|\ln\epsilon_3|^{\frac{1}{9}}N^{\frac{3}{4
}}\right)^{2m-d+1}.
\]

If $N^{\frac{3}{8}}\geq \frac{1}{2^{\frac{5}{6}}
\pi^{\frac{2}{3}}}|\ln\epsilon_3|^{\frac{1}{9}}$, then $N_0=N$. It follows from
Lemma \ref{nfed} that
\[
 \sum_{n=1}^{N_0}  \Bigl|\int_{B_R}\boldsymbol f(\boldsymbol x) e^{-{\rm
i}(\frac{\pi}{R})\boldsymbol n \cdot \boldsymbol x} {\rm d}\boldsymbol
x\Bigr|^2\lesssim\epsilon_2^2.
\]
Combining the above estimates and Lemma \ref{TFE}, we obtain
\begin{align*}
& \sum_{n=1}^{\infty}  \Bigl|\int_{B_R}\boldsymbol f(\boldsymbol
x)e^{-{\rm i}(\frac{\pi}{R})\boldsymbol n \cdot \boldsymbol x} {\rm
d}\boldsymbol x\Bigr|^2\lesssim
\epsilon_2^2+\frac{M^2}{\left(\frac{|\ln\epsilon_3|^{\frac{1}{9}}N^{\frac{3}{
2}}}{(6m-3d+3)^3}\right)^{2m-d+1}}\\
&\qquad+ \frac{M^2}{\bigl(|\ln\epsilon_3|^{\frac{1}{9}}N^{\frac{3}{4}}
\bigr)^{2m-d+1}}+\frac{M^2(2^{\frac{5}{6}}\pi^{\frac{2}{3}})^{2m-d+1}}{
\left(|\ln\epsilon_3|^{\frac{1}{9}}N^{\frac{5}{8}}\right)^{2m-d+1}}.
\end{align*}
Noting that $N^{\frac{5}{8}}\le N^{\frac{3}{4}} \le N^{\frac{3}{2}}$ and
$2^{\frac{5}{6}}\pi^{\frac{2}{3}}\le (6m-3d+3)^3$, $\forall m\geq d$, we
complete the proof after combining the above estimates.
\end{proof}

\section{Electromagnetic waves}

This section concerns the inverse source problem for electromagnetic waves.
Following the general theme for the elasticity case presented in Section 2, we
discuss the uniqueness of the problem and then show that the increasing
stability can be achieved to reconstruct the radiating electric current density
from the tangential trace of the electric field at multiple frequencies.
The technical details differ from elastic waves due to different model
equations and Green's tensors.

\subsection{Problem formulation}

We consider the time-harmonic Maxwell equations in a homogeneous medium:
\begin{equation}\label{me}
\nabla \times \boldsymbol{E} - {\rm i}\kappa \boldsymbol{H} = 0,\quad
\nabla \times \boldsymbol{H} + {\rm i}\kappa \boldsymbol{E} =
\boldsymbol{J}\quad {\rm in} ~\mathbb R^3,
\end{equation}
where $\kappa>0$ is the wavenumber, $\boldsymbol E\in\mathbb C^3$ and
$\boldsymbol H\in\mathbb C^3$ are the electric field and the magnetic field,
respectively, $\boldsymbol J\in\mathbb C^3$ is the electric current density and
is assumed to have a compact support $\Omega$. The problem geometry is the same
as that for elastic waves and is shown in Figure \ref{pg}. The
Silver--M\"{u}ller radiation condition is required to make the direct problem
well-posed:
\begin{equation}\label{smrc}
\lim_{r\to \infty}((\nabla\times\boldsymbol E) \times
\boldsymbol x - {\rm i}\kappa r\boldsymbol E) = 0, \quad r=|\boldsymbol x|.
\end{equation}

Eliminating the magnetic field $\boldsymbol H$ from \eqref{me}, we obtain
the decoupled Maxwell system for the electric field $\boldsymbol E$:
\begin{equation}\label{ef}
 \nabla\times(\nabla\times\boldsymbol E)-\kappa^2\boldsymbol E={\rm
i}\kappa\boldsymbol J\quad\text{in}~ \mathbb R^3.
\end{equation}
Given $\boldsymbol J\in L^2(\Omega)^3$, it is known that
the scattering problem \eqref{smrc}--\eqref{ef} has a unique
solution (cf. \cite{M-03}):
\begin{equation}\label{es}
 \boldsymbol E(\boldsymbol x, \kappa)=\int_\Omega \mathbf G_{\rm M}(\boldsymbol
x, \boldsymbol y; \kappa)\cdot\boldsymbol J(\boldsymbol y){\rm d}\boldsymbol y,
\end{equation}
where $\mathbf G_{\rm M}(\boldsymbol x, \boldsymbol y; \kappa)$ is Green's
tensor for the Maxwell system \eqref{ef}. Explicitly, we have
\begin{equation}\label{gtm}
 \mathbf G_{\rm M}(\boldsymbol x, \boldsymbol y; \kappa)={\rm i}\kappa
g_3(\boldsymbol x, \boldsymbol y; \kappa)\mathbf I_3+\frac{\rm
i}{\kappa}\nabla_{\boldsymbol x}\nabla_{\boldsymbol x}^\top g_3(\boldsymbol x,
\boldsymbol y; \kappa),
\end{equation}
where $g_3$ is the fundamental solution of the three-dimensional Helmholtz
equation and is given in \eqref{gn}.

Let $\boldsymbol E\times\boldsymbol\nu$ and $\boldsymbol H\times\boldsymbol\nu$
be the tangential trace of the electric field and the magnetic field,
respectively. It is shown in \cite{AN-SIMA00} that there exists a capacity
operator $T_{\rm M}$ such that
\begin{equation}\label{mtbc1}
 \boldsymbol H\times\boldsymbol\nu=T_{\rm M}(\boldsymbol
E\times\boldsymbol\nu)\quad\text{on}~\Gamma_R,
\end{equation}
which implies that $\boldsymbol H\times\boldsymbol\nu$ can be computed once
$\boldsymbol E\times\boldsymbol\nu$ is available on $\Gamma_R$. The transparent
boundary condition \eqref{mtbc1} can be equivalently written as
\begin{equation}\label{mtbc2}
(\nabla\times\boldsymbol E)\times\boldsymbol\nu={\rm i}\kappa T_{\rm M}
(\boldsymbol E\times\boldsymbol\nu)\quad\text{on}~\Gamma_R.
\end{equation}
It follows from \eqref{mtbc2} that we define the following boundary measurement
in terms of the tangential trace of the electric field only:
\[
\|\boldsymbol E(\cdot,
\kappa)\times\boldsymbol\nu\|^2_{\Gamma_R}=\int_{\Gamma_R}\left(|T_{\rm
M}(\boldsymbol E(\boldsymbol x, \kappa)\times\boldsymbol\nu)|^2+|\boldsymbol
E(\boldsymbol x, \kappa)\times\boldsymbol\nu|^2\right){\rm d}\gamma(\boldsymbol
x).
\]

\begin{prob}\label{p3}
Let $\boldsymbol J$ be the electric current density with the compact support
$\Omega$. The inverse source problem of electromagnetic waves is to determine
$\boldsymbol J$ from the tangential trace of the electric field $\boldsymbol
E(\boldsymbol x, \kappa)\times\boldsymbol\nu$ for $\boldsymbol x\in\Gamma_R.$
\end{prob}

\subsection{Uniqueness}

In this section, we discuss the uniqueness and non-uniqueness of Problem
\ref{p3}. We begin with a simple uniqueness result.

\begin{theo}
Let $I \subset \mathbb R^+$ be an open interval. If $\nabla\cdot\boldsymbol
J=0$, then the multiple-frequency
data $\{\boldsymbol E(\boldsymbol x,\kappa)\times \boldsymbol\nu: \boldsymbol
x\in\Gamma_R,\, \omega \in I\}$ can uniquely determine $\boldsymbol J$.
\end{theo}

\begin{proof}
We assume $\boldsymbol E(\boldsymbol x, \kappa)\times\boldsymbol\nu = 0$ on
$\Gamma_R$ for all $\kappa \in I$. Let $\boldsymbol{E}^{\rm inc}$ and
$\boldsymbol{H}^{\rm inc}$ be the electric
and magnetic plane waves. Explicitly, we
have
\begin{equation}\label{ehpw}
\boldsymbol{E}^{\rm inc} = \boldsymbol{p}e^{{\rm
i}\kappa\boldsymbol{x} \cdot \boldsymbol{d}}\quad\text{and}\quad
\boldsymbol{H}^{\rm inc} = \boldsymbol{q}e^{{\rm
i}\kappa\boldsymbol{x} \cdot \boldsymbol{d}},
\end{equation}
where $\boldsymbol{d}(\theta, \varphi) = (\sin\theta\cos\varphi,
\sin\theta\sin\varphi, \cos\theta)^\top$ is the unit propagation vector, and
$\boldsymbol p, \boldsymbol q$ are two unit polarization vectors and satisfy $
\boldsymbol p(\theta, \varphi)\cdot\boldsymbol d(\theta,
\varphi)=0, \boldsymbol q (\theta, \varphi)= \boldsymbol p (\theta,
\varphi)\times \boldsymbol d (\theta, \varphi)$ for all $\theta\in[0, \pi],
\varphi\in[0, 2\pi]$. It is easy to verify that $\boldsymbol{E}^{\rm inc}$
and $\boldsymbol{H}^{\rm inc}$ satisfy the homogeneous Maxwell equations
in $\mathbb R^3$:
\begin{equation}\label{eiw}
\nabla \times (\nabla \times \boldsymbol{E}^{\rm inc}) -
\kappa^2 \boldsymbol{E}^{\rm inc} = 0
\end{equation}
and
\begin{equation}\label{hiw}
\nabla \times (\nabla \times \boldsymbol{H}^{\rm inc}) -
\kappa^2 \boldsymbol{H}^{\rm inc} = 0.
\end{equation}

Let $\boldsymbol\xi = \kappa \boldsymbol d$ with $|\boldsymbol\xi|=\kappa\in
I$. We have from \eqref{ehpw} that
$\boldsymbol{E}^{\rm inc} =\boldsymbol{p} e^{-{\rm i}\boldsymbol\xi\cdot
\boldsymbol x}$ and $\boldsymbol{H}^{\rm inc} = \boldsymbol{q} e^{-{\rm
i}\boldsymbol\xi\cdot \boldsymbol x}$. Multiplying the both sides of \eqref{ef}
by $\boldsymbol{E}^{\rm inc}$, using the integration by parts
over $B_R$ and \eqref{eiw}, we obtain
\begin{align*}
{\rm i}\kappa\int_{B_R} \boldsymbol{p}e^{- {\rm i}\boldsymbol{\xi} \cdot
\boldsymbol{x}} \cdot \boldsymbol{J}(\boldsymbol x){\rm d}\boldsymbol x
= -\int_{\Gamma_R} \left({\rm i}\kappa
T_{\rm M}(\boldsymbol{E}(\boldsymbol{x},\kappa) \times \boldsymbol\nu) \cdot
\boldsymbol{E}^{\rm inc}+ (\boldsymbol{E}(\boldsymbol{x},\kappa) \times
\boldsymbol\nu)\cdot(\nabla \times \boldsymbol{E}^{\rm
inc})\right){\rm d}S,
\end{align*}
which means $\boldsymbol{p} \cdot \hat{\boldsymbol J}(\kappa \boldsymbol d) = 0$
for all $\omega \in I$. Similarly, we have $\boldsymbol{q} \cdot
\hat{\boldsymbol J}(\kappa \boldsymbol d) = 0$ for all $\kappa \in I$. On the
other hand, since $\hat{\boldsymbol J}(\kappa \boldsymbol d)$ is an analytic
function with respect to $\kappa \in \mathbb C$, where $\mathbb C$ denotes the
complex number domain, we have both $\boldsymbol{p} \cdot \hat{\boldsymbol
J}(\kappa \boldsymbol d) = 0$ and $\boldsymbol{q} \cdot \hat{\boldsymbol
J}(\kappa \boldsymbol d) = 0$ for all $\kappa>0$. Since $\boldsymbol p,
\boldsymbol q, \boldsymbol d$ are orthonormal vectors,
they form an orthonormal basis in $\mathbb R^3$. We have from
the Pythagorean theorem that
\[
|\hat{\boldsymbol{J}}(\boldsymbol\xi)|^2 = |\boldsymbol{p} \cdot
\hat{\boldsymbol{J}}(\boldsymbol\xi)|^2 + |\boldsymbol{q}\cdot
\hat{\boldsymbol{J}}(\boldsymbol\xi)|^2 + |\boldsymbol{d} \cdot
\hat{\boldsymbol{J}}(\boldsymbol\xi)|^2.
\]

On the other hand, since $\nabla \cdot \boldsymbol J = 0,$ we have for each
$\kappa \in I$ that
\[
{\rm i}\kappa \boldsymbol d \cdot \hat{\boldsymbol J}(\boldsymbol \xi)  = {\rm
i}\kappa\int_{B_R}\boldsymbol{d}e^{{\rm
i}\kappa\boldsymbol{x} \cdot \boldsymbol{d}} \cdot \boldsymbol J(\boldsymbol
x){\rm d}\boldsymbol x = \int_{B_R} \nabla e^{{\rm
i}\kappa\boldsymbol{x} \cdot \boldsymbol{d}}\cdot \boldsymbol J{\rm
d}\boldsymbol x = \int_{B_R} e^{{\rm
i}\kappa\boldsymbol{x} \cdot \boldsymbol{d}} \nabla \cdot \boldsymbol J {\rm
d}\boldsymbol x = 0,
\]
which yields that
\[
|\hat{\boldsymbol{J}}(\boldsymbol\xi)|^2 = |\boldsymbol{p} \cdot
\hat{\boldsymbol{J}}(\boldsymbol\xi)|^2 + |\boldsymbol{q}\cdot
\hat{\boldsymbol{J}}(\boldsymbol\xi)|^2 = 0,
\]
which means $\hat{\boldsymbol{J}}(\boldsymbol\xi) = 0$ for all $\boldsymbol \xi
\in \mathbb R^3,$ and then $\boldsymbol{J}(\boldsymbol x) = 0.$
\end{proof}

Next we discuss the uniqueness result much further. The goal is to distinguish
the radiating and non-radiating current densities. We study a variational
equation relating the unknown current density $\boldsymbol{J}$ to the data
$\boldsymbol E\times\boldsymbol\nu$ on $\Gamma_R$.

Multiplying \eqref{ef} by the complex conjugate of a test function
$\boldsymbol{\xi}$ on both sides, integrating over $B_R$, and using the
integration by parts, we obtain
\begin{align}\label{VE}
\int_{ B_R}\left(\nabla \times \boldsymbol{E} \cdot \nabla \times
\bar{\boldsymbol \xi} - \kappa^2 \boldsymbol{E} \cdot
\bar{\boldsymbol \xi}\right){\rm d}\boldsymbol{x} - \int_{\Gamma_R}
[(\nabla\times\boldsymbol E)\times\boldsymbol\nu ]\cdot
\bar{\boldsymbol \xi} {\rm d}\gamma
= {\rm i}\kappa\int_{B_R}  \boldsymbol{J}\cdot\bar{\boldsymbol \xi} {\rm
d}\boldsymbol{x}.
\end{align}
Substituting \eqref{mtbc2} into \eqref{VE}, we obtain the variational problem:
To find $\boldsymbol{E}\in H({\rm curl}, B_R)$ such that
\begin{align}\label{VE1}
\int_{ B_R}\left(\nabla\times \boldsymbol{E} \cdot\nabla \times
\bar{\boldsymbol \xi}- \kappa^2 \boldsymbol{E} \cdot \bar{\boldsymbol
\xi}\right) {\rm d}\boldsymbol{x} -{\rm i}\kappa\int_{\Gamma_R}
T_{\rm M}(\boldsymbol{E}\times\boldsymbol\nu) \cdot\bar{\boldsymbol
\xi}_{\Gamma_R} {\rm d}\gamma \notag\\
= {\rm i}\kappa \int_{ B_R}\boldsymbol{J}\cdot\bar{\boldsymbol \xi}
{\rm d}\boldsymbol{x},\quad\forall~ \boldsymbol{\xi}\in H({\rm curl}, B_R).
\end{align}
Given $\boldsymbol{J}\in L^2(B_R)^3$, the variational problem
\eqref{VE1} can be shown to have a unique weak solution $\boldsymbol E\in H({\rm
curl}, B_R)$ (cf. \cite{M-03, N-00}).

Assuming that $\boldsymbol{\xi}$ is a smooth function, we take the integration
by parts one more time of \eqref{VE1} and get the identity:
\begin{align}\label{VE2}
{\rm i}\kappa \int_{B_R} \boldsymbol{J}\cdot \bar{\boldsymbol{\xi}} {\rm
d}\boldsymbol{x} =& \int_{B_R} \boldsymbol{E}\cdot (\nabla \times
(\nabla \times \bar{\boldsymbol \xi}) - \kappa^2\bar{\boldsymbol \xi}){\rm
d}\boldsymbol{x}\notag\\
& - \int_{\Gamma_R}\left(  (\boldsymbol{E}\times\boldsymbol\nu) \cdot( \nabla
\times \bar{\boldsymbol \xi})+ {\rm i}\kappa T_{\rm
M}(\boldsymbol{E}\times\boldsymbol\nu)  \cdot
\bar{\boldsymbol \xi}_{\Gamma_R}\right){\rm d}\gamma.
\end{align}
Now we choose $\boldsymbol{\xi} \in H({\rm curl}, B_R)$ to satisfy
\begin{align}\label{VE3}
\int_{B_R} \left(\nabla \times \boldsymbol{\xi} \cdot \nabla \times
\boldsymbol{\psi}- \kappa^2 \boldsymbol{\xi} \cdot \boldsymbol{\psi}
\right){\rm d}\boldsymbol{x} = 0,\quad\forall\boldsymbol\psi\in
C_0^{\infty}(B_R)^3,
\end{align}
which implies that $\boldsymbol{\xi}$ is a weak solution of the Maxwell system:
\begin{align*}
\nabla \times (\nabla \times \boldsymbol{\xi}) - \kappa^2 \boldsymbol{\xi} = 0
\quad {\rm in}~ B_R.
\end{align*}
Using this choice of $\boldsymbol{\xi}$, we can see that \eqref{VE2} becomes
\begin{align}\label{VE4}
 {\rm i}\kappa\int_{ B_R} \boldsymbol{J}\cdot \bar{\boldsymbol{\xi}} {\rm
d}\boldsymbol{x}=- \int_{\Gamma_R}\left(  (\boldsymbol{E}\times\boldsymbol\nu)
\cdot( \nabla \times \bar{\boldsymbol \xi})+ {\rm i}\kappa T_{\rm
M}(\boldsymbol{E}\times\boldsymbol\nu)  \cdot
\bar{\boldsymbol \xi}_{\Gamma_R}\right){\rm d}\gamma
\end{align}
for all $\boldsymbol{\xi} \in H({\rm curl}, B_R)$ satisfying \eqref{VE3}.

Denote by $\mathbb X(B_R)$ be the closure of the set $\{\boldsymbol{E} \in
H({\rm curl}, B_R): \boldsymbol{E} ~\text{ satisfies}~ \eqref{VE3}\}$ in the
$L^2(B_R)^3$ norm. We have the following orthogonal decomposition of
$L^2(B_R)^3$:
\begin{align*}
L^2(B_R)^3 = \mathbb X(B_R)\oplus \mathbb Y(B_R).
\end{align*}
It is shown in \cite{AM-IP06} that $\mathbb Y(B_R)$ is an infinitely dimensional
subspace of $L^2(B_R)^3$, which is stated in the following lemma.

\begin{lemm}\label{phi}
Let $\boldsymbol{\psi} \in C_0^{\infty}(B_R)^3$. If
$ \boldsymbol{\phi}= \nabla \times (\nabla \times \boldsymbol{\psi}) - \kappa^2
\boldsymbol{\psi}$, then $\boldsymbol{\phi} \in \mathbb Y(B_R)$.
\end{lemm}

It follows from  Lemma \ref{phi} that $\mathbb X(B_R)$ is a proper subspace of
$L^2(B_R)^3$. Given $\boldsymbol{J} \in L^2(B_R)^3$, only the component of
$\boldsymbol{J}$ in $\mathbb X(B_R)$ can be determined from the data
$\boldsymbol E\times\boldsymbol\nu$ on $\Gamma_R$. Moreover, it is impossible
that some other equation could be derived to determine the component of
$\boldsymbol{J}$ in $\mathbb Y(B_R)$ from the data $\boldsymbol
E\times\boldsymbol\nu$ on $\Gamma_R$.

\begin{theo}\label{nuniq}
Suppose $\boldsymbol{J}\in\mathbb Y(B_R)$. Then $\boldsymbol{J}$ does not
produce any tangential trace of electric fields on $\Gamma_R$ and thus cannot be
identified.
\end{theo}

\begin{proof}
Since $\boldsymbol J\in\mathbb Y(B_R)$, we have
from \eqref{VE4} that
\begin{align*}
 \int_{\Gamma_R}\left(  (\boldsymbol{E}\times\boldsymbol\nu)
\cdot( \nabla \times \bar{\boldsymbol \xi})+ {\rm i}\kappa T_{\rm
M}(\boldsymbol{E}\times\boldsymbol\nu)  \cdot
\bar{\boldsymbol \xi}_{\Gamma_R}\right){\rm d}\gamma=0,\quad\forall~
\boldsymbol{\xi} \in \mathbb X(B_R),
\end{align*}
which yields
\begin{equation}\label{theo2_s1}
\int_{\Gamma_R}  (\boldsymbol{E}\times\boldsymbol\nu) \cdot (\overline{\nabla
\times \boldsymbol \xi - {\rm i}\kappa T^{\ast}_{\rm M}({\boldsymbol
\xi}_{\Gamma_R})}){\rm d}\gamma =0.
\end{equation}
Here $T^*_{\rm M}$ is the adjoint operator of $T_{\rm M}$. Let
\begin{align*}
\nabla \times \boldsymbol{\xi} - {\rm i}\kappa
T^{\ast}_{\rm M}(\boldsymbol{\xi}_{\partial B_R}) = \bar{\boldsymbol \eta} \quad
{\rm on}~ \Gamma_R.
\end{align*}
More precisely, $\boldsymbol{\xi} \in H({\rm curl}, B_R)$ satisfies the
variational problem
\begin{align*}
\int_{B_R} ( \nabla \times \boldsymbol{\xi} \cdot \nabla \times
\boldsymbol{\phi}- \kappa^2 \boldsymbol{\xi}\cdot\boldsymbol{\phi} ){\rm
d}\boldsymbol{x}+ \int_{\Gamma_R} (\bar{\boldsymbol{\eta}} - {\rm i}\kappa
T^{\ast}_{\rm M}(\boldsymbol{\xi}_{\Gamma_R}))\cdot (\boldsymbol{\phi}\times
\boldsymbol\nu){\rm d}\gamma =0,\quad\forall ~\boldsymbol{\phi} \in H({\rm
curl},
B_R).
\end{align*}
It is shown in \cite{AM-IP06} that there exists a unique solution
$\boldsymbol{\xi}\in \mathbb X(B_R)$ to the above boundary value problem for
any $\boldsymbol{\eta}\in H^{-1/2}({\rm curl}, \Gamma_R)$, where
$\nabla_{\Gamma_R}$ is the surface gradient. Hence we have from
\eqref{theo2_s1} that
\begin{align*}
\int_{\Gamma_R}  (\boldsymbol{E}\times\boldsymbol\nu) \cdot
\boldsymbol{\eta} {\rm d}\gamma = 0, \quad\forall ~
\boldsymbol{\eta} \in H^{-1/2}({\rm curl}, \Gamma_R),
\end{align*}
which yields that $\boldsymbol E\times\boldsymbol\nu = 0$ on $\Gamma_R$ and
completes the proof.
\end{proof}

\begin{rema}
The electric current densities in $\mathbb Y(B_R)$ are called
non-radiating sources. It corresponds to find a minimum norm solution
when computing the component of the source in $\mathbb X(B_R)$.
\end{rema}

It is shown in Theorem \ref{nuniq} that $\boldsymbol J$ cannot be determined
from the tangential trace of the electric field $\boldsymbol
E\times\boldsymbol\nu$ on $\Gamma_R$ if $\boldsymbol J\in \mathbb Y(B_R)$. We
show in the following theorem that it is also impossible to determine
$\boldsymbol J$ from the normal component of the electric field $\boldsymbol
E\cdot\boldsymbol\nu$ on $\Gamma_R$ if $\boldsymbol J\in \mathbb Y(B_R)$.

\begin{theo}
 Suppose $\boldsymbol{J}\in\mathbb Y(B_R)$. Then $\boldsymbol{J}$ does not
produce any normal component of electric fields on $\Gamma_R$.
\end{theo}

\begin{proof}
 Let $\phi\in C^\infty(B_R)$. Multiplying both sides of \eqref{ef} by
$\nabla\phi$ and integrating on $B_R$, we have
\[
\int_{B_R} (\nabla\times(\nabla\times\boldsymbol E)-\kappa^2\boldsymbol E)\cdot
\nabla \phi  {\rm d} \boldsymbol{x} = {\rm i} \kappa \int_{B_R} \boldsymbol{J}
\cdot \nabla \phi {\rm d} \boldsymbol{x},
\]
It follows from the integration by parts that
\[
\int_{B_R} (\nabla\times(\nabla\times\boldsymbol E)) \cdot \nabla \phi {\rm d}
\boldsymbol{x} = \int_{B_R} (\nabla\times\boldsymbol E) \cdot (\nabla\times
\nabla \phi) {\rm d} \boldsymbol{x}- \int_{\Gamma_R} (\boldsymbol\nu\times
(\nabla\times\boldsymbol E)) \cdot \nabla \phi{\rm d}\gamma.
\]
Noting $\nabla\times \nabla \phi = 0$ and \eqref{mtbc2}, and using Theorem
\ref{nuniq}, we obtain
\[
\int_{\Gamma_R} (\boldsymbol\nu\times (\nabla\times\boldsymbol E)) \cdot
\nabla \phi{\rm d}\gamma=0.
\]
Combining the above equations gives
\[
 -\kappa^2 \int_{B_R} \boldsymbol E \cdot \nabla \phi  {\rm d} \boldsymbol{x} =
{\rm i} \kappa \int_{B_R} \boldsymbol{J} \cdot \nabla \phi {\rm d}
\boldsymbol{x},
\]
which implies
\begin{equation}\label{nq-s1}
{\rm i}\kappa \int_{B_R} \boldsymbol E \cdot \nabla \phi {\rm d} \boldsymbol{x}
= \int_{B_R} \boldsymbol{J} \cdot \nabla \phi {\rm d} \boldsymbol{x}.
\end{equation}
We have from the integration by parts that
\begin{equation}\label{nq-s2}
{\rm i}\kappa \int_{B_R} \boldsymbol E \cdot \nabla \phi {\rm d} \boldsymbol{x}
=-{\rm i}\kappa \int_{B_R} \nabla \cdot \boldsymbol E  \phi{\rm d}
\boldsymbol{x}+{\rm i}\kappa \int_{\Gamma_R}(\boldsymbol E \cdot
\boldsymbol\nu)\phi{\rm d}\gamma.
\end{equation}

On the other hand, since
\[
\nabla \times \boldsymbol H + {\rm i}\kappa \boldsymbol E = \boldsymbol J,
\]
then by taking the divergence on both sides, we have
\[
{\rm i}\kappa\nabla \cdot \boldsymbol E =\nabla\cdot\boldsymbol J.
\]
Hence
\begin{equation}\label{nq-s3}
{\rm i}\kappa \int_{B_R} \nabla \cdot \boldsymbol E \phi{\rm d} \boldsymbol{x}
= \int_{B_R} \nabla \cdot \boldsymbol J  \phi \boldsymbol{x} = - \int_{B_R}
\boldsymbol J \cdot \nabla  \phi {\rm d} \boldsymbol{x}.
\end{equation}
Combing \eqref{nq-s1}--\eqref{nq-s3}, we get
\begin{align*}
\int_{\Gamma_R}(\boldsymbol E \cdot \boldsymbol\nu) \phi  {\rm
d} \boldsymbol{x} = 0,\quad\forall\phi\in C^\infty(B_R),
\end{align*}
which implies that $\boldsymbol E \cdot \boldsymbol\nu$ on $\Gamma_R$ and
completes the proof.
\end{proof}

The following theorem concerns the uniqueness result of Problem \ref{p3}.

\begin{theo}
Suppose $\boldsymbol{J} \in \mathbb X(B_R)$, then $\boldsymbol{J}$ can be
uniquely determined by the data $\boldsymbol{E} \times\boldsymbol \nu$ on
$\Gamma_R$.
\end{theo}

\begin{proof}
It suffices to show that $\boldsymbol J=0$ if $\boldsymbol
E\times\boldsymbol\nu=0$ on $\Gamma_R$. It follows from \eqref{VE4} that we
have
\begin{align*}
\int_{ B_R}  \boldsymbol{J}  \cdot\bar{\boldsymbol \xi} {\rm
d}\boldsymbol{x}=0,\quad\forall~\boldsymbol\xi\in\mathbb X(B_R).
\end{align*}
Taking $\boldsymbol \xi=\boldsymbol J$ yields that
\begin{align*}
\int_{B_R}  |\boldsymbol{J}|^2{\rm d}\boldsymbol x = 0,
\end{align*}
which completes the proof.
\end{proof}

Taking account of the uniqueness result, we revise the inverse source problem
for electromagnetic waves, which is to determine $\boldsymbol J$ in the smaller
space $\mathbb X(B_R)$.

\begin{prob}[Continuous frequency data for electromagnetic waves]\label{p4}
Let $\boldsymbol J\in \mathbb X(B_R)$. The
inverse source problem of electromagnetic waves is to
determine $\boldsymbol J$ from the tangential trace of the electric field
$\boldsymbol E(\boldsymbol x, \kappa)\times\boldsymbol\nu$ for $\boldsymbol
x\in\Gamma_R, \kappa\in (0, K)$, where $K>1$ is a constant.
\end{prob}

\subsection{Stability with continuous frequency data}

Define a functional space
\[
 \mathbb J_M(B_R)=\{\boldsymbol J\in\mathbb X(B_R)\cap H^m(B_R)^3:
\|\boldsymbol J\|_{H^m(B_R)^3}\leq M\},
\]
where $m\geq d$ is an integer and $M>1$ is a constant. The following is our main
result regarding the stability for Problem \ref{p4}.

\begin{theo}\label{mrm}
Let $\boldsymbol{E}$ be the solution of the scattering problem
\eqref{smrc}--\eqref{ef} corresponding to
$\boldsymbol{J}\in \mathbb J_M(B_R)$. Then
\begin{equation}\label{je}
 \| \boldsymbol{J}\|^2_{L^2(B_R)^3}\lesssim
\epsilon_4^2+\frac{M^2}{\left(\frac{K^{\frac{2}{3}}|\ln\epsilon_4|^{\frac{1}{4
}}}{(R+1)(6m-15)^3}\right)^{2m-5}},
\end{equation}
where
\[
 \epsilon_4=\left(\int_0^K \kappa^2\|\boldsymbol E(\cdot,
\kappa)\times\boldsymbol\nu\|^2_{\Gamma_R}{\rm d}\kappa \right)^{\frac{1}{2}}
\]

\end{theo}

\begin{rema}
The stability estimate \eqref{je} is consistent with that for elastic waves in
\eqref{fe}. It also has two parts: the data discrepancy and the high frequency
tail. The ill-posedness of the inverse problem decreases as $K$ increases.
\end{rema}

We begin with several useful lemmas.

\begin{lemm}\label{mje}
Let $\boldsymbol E$ be the solution of \eqref{smrc}--\eqref{ef} corresponding
to the source $\boldsymbol J\in \mathbb X(B_R)$. Then
\[
\|\boldsymbol{J} \|^2_{L^2(B_R)^3}\lesssim\int_0^{\infty}\kappa^{2}
\|\boldsymbol E(\cdot, \kappa)\times\boldsymbol\nu\|^2_{\Gamma_R} {\rm
d}\kappa.
\]
\end{lemm}

\begin{proof}
Let $\boldsymbol{E}^{\rm inc}$ and $\boldsymbol{H}^{\rm inc}$ be the electric
and magnetic plane waves:
\begin{equation}\label{mpw}
\boldsymbol{E}^{\rm inc}(\boldsymbol x) = \boldsymbol{p}e^{-{\rm
i}\kappa\boldsymbol{x} \cdot \boldsymbol{d}}\quad\text{and}\quad
\boldsymbol{H}^{\rm inc}(\boldsymbol x) = \boldsymbol{q}e^{-{\rm
i}\kappa\boldsymbol{x} \cdot \boldsymbol{d}},
\end{equation}
where $\boldsymbol{d}(\theta, \varphi) = (\sin\theta\cos\varphi,
\sin\theta\sin\varphi, \cos\theta)^\top$ is the unit propagation vector, and
$\boldsymbol p, \boldsymbol q$ are two unit polarization vectors and satisfy $
\boldsymbol p(\theta, \varphi)\cdot\boldsymbol d(\theta,
\varphi)=0, \boldsymbol q (\theta, \varphi)= \boldsymbol p (\theta,
\varphi)\times \boldsymbol d (\theta, \varphi)$ for all $\theta\in[0, \pi],
\varphi\in[0, 2\pi]$. The electric and magnetic plane waves satisfy
\begin{equation}\label{epw}
\nabla \times (\nabla \times \boldsymbol{E}^{\rm inc}) -
\kappa^2 \boldsymbol{E}^{\rm inc} = 0
\end{equation}
and
\begin{equation}\label{hpw}
\nabla \times (\nabla \times \boldsymbol{H}^{\rm inc}) -
\kappa^2 \boldsymbol{H}^{\rm inc} = 0.
\end{equation}

Let $\boldsymbol\xi = \kappa \boldsymbol d$ with $|\boldsymbol\xi|=\kappa\in
(0, \infty)$. We have from \eqref{mpw} that
$\boldsymbol{E}^{\rm inc} =\boldsymbol{p} e^{-{\rm i}\boldsymbol\xi\cdot
\boldsymbol x}$ and $\boldsymbol{H}^{\rm inc} = \boldsymbol{q} e^{-{\rm
i}\boldsymbol\xi\cdot \boldsymbol x}$. Multiplying the both sides of \eqref{ef}
by $\boldsymbol{E}^{\rm inc}$, using the integration by parts
over $B_R$ and \eqref{epw}, we obtain
\begin{align*}
{\rm i}\kappa\int_{B_R} (\boldsymbol{p}e^{- {\rm i}\boldsymbol{\xi} \cdot
\boldsymbol{x}} )\cdot \boldsymbol{J}(\boldsymbol x){\rm d}\boldsymbol x
= -\int_{\Gamma_R} \left({\rm i}\kappa
T_{\rm M}(\boldsymbol{E}(\boldsymbol{x},\kappa) \times \boldsymbol\nu) \cdot
\boldsymbol{E}^{\rm inc}+ (\boldsymbol{E}(\boldsymbol{x},\kappa) \times
\boldsymbol\nu)\cdot(\nabla \times \boldsymbol{E}^{\rm
inc})\right){\rm d}\gamma.
\end{align*}
A simple calculation yields that
\begin{align*}
\nabla \times \boldsymbol{E}^{\rm inc} = -{\rm i}\kappa
\boldsymbol{d} \times \boldsymbol{p}e^{-{\rm i}\kappa\boldsymbol
x\cdot\boldsymbol d},
\end{align*}
which gives
\begin{align*}
|\nabla \times \boldsymbol{E}^{\rm inc}| = \kappa.
\end{align*}
Since ${\rm supp}\boldsymbol{f}=\Omega\subset B_R$, we have
\[
 \int_{B_R}(\boldsymbol p e^{-{\rm i}\boldsymbol \xi\cdot\boldsymbol
x})\cdot\boldsymbol{J}(\boldsymbol x){\rm d}\boldsymbol
x=\boldsymbol p\cdot\int_{\mathbb{R}^3}\boldsymbol{J}(\boldsymbol x)e^{-{\rm
i}\boldsymbol\xi\cdot\boldsymbol
x}{\rm d}\boldsymbol x=\boldsymbol p\cdot\hat{\boldsymbol J}(\boldsymbol\xi).
\]
Combining the above estimates yields
\[
 |\boldsymbol p\cdot\hat{\boldsymbol J}(\boldsymbol\xi)|^2
\lesssim \int_{\Gamma_R}(|T_{\rm M}(\boldsymbol{E}(\boldsymbol
x,\kappa)\times\boldsymbol \nu)|^2+|\boldsymbol{E}(\boldsymbol x,\kappa)\times
\boldsymbol\nu|^2){\rm d}\gamma(\boldsymbol x)=\|\boldsymbol E(\cdot,
\kappa)\|^2_{\Gamma_R}.
\]
Hence
\[
 \int_{\mathbb R^3}|\boldsymbol p\cdot\hat{\boldsymbol
J}(\boldsymbol\xi)|^2{\rm d}\boldsymbol\xi
\lesssim\int_{\mathbb R^3}\|\boldsymbol E(\cdot,
\kappa)\|^2_{\Gamma_R}{\rm d}\boldsymbol\xi.
\]
Using the spherical coordinates, we get
\[
  \int_{\mathbb R^3}|\boldsymbol p\cdot\hat{\boldsymbol
J}(\boldsymbol\xi)|^2{\rm d}\boldsymbol\xi
\lesssim\int_0^{2\pi}{\rm d}\theta\int_0^{\pi}{\rm sin}\varphi
{\rm d}\varphi\int_0^{\infty}\kappa^2\|\boldsymbol E(\cdot,
\kappa)\|^2_{\Gamma_R}{\rm d}\kappa
\lesssim \int_0^{\infty}\kappa^2\|\boldsymbol E(\cdot,
\kappa)\|^2_{\Gamma_R}
{\rm d}\kappa.
\]
Similarly we may show from \eqref{hpw} and the integration by parts that
\begin{align*}
\int_{\mathbb{R}^3}\left|\int_{\mathbb{R}^3}(\boldsymbol q e^{-{\rm
i}\boldsymbol\xi\cdot\boldsymbol
x})\cdot\boldsymbol{J}(\boldsymbol x){\rm d}\boldsymbol x\right|^2{\rm
d}\boldsymbol\xi= \int_{\mathbb R^3}|\boldsymbol q\cdot\hat{\boldsymbol
J}(\boldsymbol\xi)|^2{\rm d}\boldsymbol\xi
\lesssim \int_0^{\infty}\kappa^2\|\boldsymbol E(\cdot,
\kappa)\|^2_{\Gamma_R}{\rm d}\kappa.
\end{align*}
Since $\boldsymbol p, \boldsymbol q, \boldsymbol d$ are orthonormal vectors,
they form an orthonormal basis in $\mathbb R^3$. We have from
the Pythagorean theorem that
\[
|\hat{\boldsymbol{J}}(\boldsymbol\xi)|^2 = |\boldsymbol{p} \cdot
\hat{\boldsymbol{J}}(\boldsymbol\xi)|^2 + |\boldsymbol{q}\cdot
\hat{\boldsymbol{J}}(\boldsymbol\xi)|^2 + |\boldsymbol{d} \cdot
\hat{\boldsymbol{J}}(\boldsymbol\xi)|^2.
\]

On the other hand, since $\boldsymbol J$ has a compact support $\Omega$
contained in $B_R$ and $\boldsymbol J \in \mathbb X(B_R)$, we obtain that
$\boldsymbol{J}$ is a weak solution of the Maxwell system:
\begin{align*}
\nabla \times( \nabla \times \boldsymbol{J}) - \kappa^2 \boldsymbol{J} = 0 \quad
{\rm in} ~B_R.
\end{align*}
Multiplying the above equation by ${\boldsymbol d} e^{{\rm i}\kappa \boldsymbol
d \cdot \boldsymbol x} $ and using integration by parts, we get
\begin{align*}
\int_{B_R}(\nabla \times \boldsymbol{J}) \cdot (\nabla \times ({\boldsymbol d}
e^{{\rm i}\kappa  \boldsymbol x\cdot\boldsymbol d})){\rm d}\boldsymbol x
= \kappa^2  {\boldsymbol d}\cdot\int_{B_R} \boldsymbol{J}(\boldsymbol x)e^{{\rm
i}\kappa \boldsymbol x\cdot\boldsymbol d}{\rm d}\boldsymbol x
= \kappa^2\boldsymbol d \cdot \hat{\boldsymbol J}(\boldsymbol\xi).
\end{align*}
Noting $\nabla \times ({\boldsymbol d} e^{{\rm i}\kappa\boldsymbol x\cdot
\boldsymbol d})={\rm i}\kappa\boldsymbol d\times\boldsymbol d e^{{\rm
i}\kappa\boldsymbol x\cdot\boldsymbol d}=0$, we get $\boldsymbol d \cdot
\hat{\boldsymbol J}(\boldsymbol\xi) = 0,$ which yields that
\[
|\hat{\boldsymbol{J}}(\boldsymbol\xi)|^2 = |\boldsymbol{p} \cdot
\hat{\boldsymbol{J}}(\boldsymbol\xi)|^2 + |\boldsymbol{q}\cdot
\hat{\boldsymbol{J}}(\boldsymbol\xi)|^2.
\]
Hence, we obtain from the Parseval theorem that
\begin{align*}
\|\boldsymbol J\|^2_{L^2(B_R)^3}&=\|\boldsymbol J\|^2_{L^2(\mathbb
R^3)^3}=\|\hat{\boldsymbol J}\|^2_{L^2(\mathbb R^3)^3}=
\int_{\mathbb{R}^3}|\hat{\boldsymbol{J}}
(\boldsymbol\xi)|^2 {\rm d} \boldsymbol\xi\\
&=\int_{\mathbb{R}^3}|\boldsymbol{p} \cdot
\hat{\boldsymbol{J}}(\boldsymbol\xi)|^2 {\rm d}\boldsymbol \xi +
\int_{\mathbb{R}^3}|\boldsymbol{q}\cdot \hat{\boldsymbol{J}}(\boldsymbol\xi)|^2
{\rm d}\boldsymbol \xi\lesssim  \int_0^{\infty}\kappa^2\|\boldsymbol E(\cdot,
\kappa)\|^2_{\Gamma_R} {\rm d}\kappa,
\end{align*}
which completes the proof.
\end{proof}

Eliminating $\boldsymbol{E}$ from \eqref{me}, we obtain
\begin{equation}\label{he}
\nabla \times (\nabla \times \boldsymbol{H}) - \kappa^2 \boldsymbol{H} = \nabla
\times \boldsymbol{J}\quad {\rm in}~\mathbb R^3.
\end{equation}
It is easy to verify from \eqref{me} that $\nabla\cdot\boldsymbol H=0$. Using
the identity
\[
\nabla \times( \nabla \times \boldsymbol{H} )= -\Delta \boldsymbol{H} +
\nabla \nabla \cdot \boldsymbol{H}= -\Delta \boldsymbol{H},
\]
we get from \eqref{he} that $\boldsymbol H$ satisfies the inhomogeneous
Helmholtz equation:
\begin{equation}\label{ihe}
\Delta\boldsymbol{H} + \kappa^2 \boldsymbol{H} = -\nabla \times
\boldsymbol{J}\quad {\rm in}~ \mathbb R^3.
\end{equation}
It is known that \eqref{ihe} has a unique solution:
\begin{equation}\label{hs}
\boldsymbol H(\boldsymbol x, \kappa) = -\int_{\Omega} g_3(\boldsymbol x,
\boldsymbol y)\mathbf I_3 \cdot \nabla \times \boldsymbol J(\boldsymbol y) {\rm
d} \boldsymbol y.
\end{equation}

Let
\begin{align}
\label{mi1}I_1(s)=&\int_0^s
\kappa^{2}\int_{\Gamma_R}\left|\int_\Omega\mathbf{G}_{\rm M}(\boldsymbol
x, \boldsymbol y; \kappa)\cdot\boldsymbol{J}(\boldsymbol x){\rm d}\boldsymbol y
\times \boldsymbol\nu(\boldsymbol x)\right|^2{\rm
d}\gamma(\boldsymbol x){\rm d}\kappa,\\
\label{mi2}I_2(s)=&\int_0^s \kappa^{2}\int_{\Gamma_R}\left|\int_\Omega
g_3(\boldsymbol x, \boldsymbol y; \kappa)\mathbf I_3\cdot \nabla \times
\boldsymbol{J}(\boldsymbol y){\rm
d}\boldsymbol y \times \boldsymbol\nu(\boldsymbol x)\right|^2{\rm
d}\gamma(\boldsymbol x){\rm d}\kappa.
\end{align}
Again, the integrands in \eqref{mi1}--\eqref{mi2} are analytic functions of
$\kappa$. The integrals with respect to $\kappa$ can be taken over any
path joining points $0$ and $s$ in $\mathcal{V}$. Thus $I_1(s)$ and $I_2(s)$ are
analytic functions of $s=s_1+{\rm i}s_2\in \mathcal{V}, s_1, s_2\in\mathbb{R}$.

\begin{lemm}\label{mi}
Let $\boldsymbol{J}\in H^2(B_R)^3$. For any $s=s_1+{\rm
i}s_2\in\mathcal{V}$, the following estimates hold:
\begin{align}
\label{m1}|I_1(s)|&\lesssim
(|s|^{5}+|s|)e^{4R|s|}\|\boldsymbol{J}\|_{H^2(B_R)^3}^2,\\
\label{m2}|I_2(s)|&\lesssim |s|^{3}e^{4R|s|}\|\boldsymbol{J}\|_{H^1(B_R)^3}^2.
\end{align}
\end{lemm}

\begin{proof}
Let $\kappa = st, t\in (0,1)$. Noting \eqref{gtm}, we have from direct calculation that
\begin{align*}
|I_1(s)|\leq I_{1,1}(s)+I_{1,2}(s),
\end{align*}
where
\begin{align*}
&I_{1,1}(s) = \int_0^1 |s|^3 t^2 \int_{\Gamma_R}\left|\int_\Omega s t
g_3(\boldsymbol x, \boldsymbol y; \kappa)\mathbf I_3 \cdot  \boldsymbol
J(\boldsymbol y){\rm d}\boldsymbol y\right|^2{\rm d}\gamma(\boldsymbol x) {\rm
d}t,\\
&I_{1,2}(s) = \int_0^1 |s|^3 t^2 \int_{\Gamma_R}\left|\int_\Omega
\frac{1}{st}\nabla_{\boldsymbol y}\nabla_{\boldsymbol y}^{\top}g(\boldsymbol x,
\boldsymbol y) \cdot  \boldsymbol J(\boldsymbol y){\rm d}\boldsymbol
y\right|^2{\rm d}\gamma(\boldsymbol x) {\rm d}t.
\end{align*}
Since ${\rm supp}\boldsymbol J=\Omega\subset B_R$ and
\[
|e^{{\rm i}st |\boldsymbol x-\boldsymbol y|}|\leq e^{2R|s|},
\quad\forall \boldsymbol x\in \Gamma_R, \boldsymbol y\in\Omega,
\]
we get from the Cauchy--Schwarz inequality that
 \begin{align}\label{lema4_s1}
I_{1,1}(s)&\lesssim\int_0^1|s|^{5}t^{4}\int_{\Gamma_R}\bigg|\int_\Omega\frac{
e^{2R|s|}}{|\boldsymbol x-\boldsymbol y|
}|\boldsymbol{J}(\boldsymbol y)|{\rm d}\boldsymbol y\bigg|^2{\rm
d}\gamma(\boldsymbol x){\rm d}t\notag\\
&\lesssim\int_0^1|s|^{5}t^{4}\int_{\Gamma_R}\left(\int_{B_R}
|\boldsymbol{J}(\boldsymbol y)|^2 {\rm
d}\boldsymbol y\right)\left(\int_\Omega\frac{e^{4R|s|}}{|\boldsymbol
x-\boldsymbol y|^2}{\rm d}\boldsymbol y\right) {\rm d}\gamma(\boldsymbol x){\rm
d}t\notag\\
&\lesssim |s|^{5}e^{4R|s|}\|\boldsymbol{J}\|_{L^2(B_R)^3}^2.
\end{align}

On the other hand, we obtain from the integration
by parts that
\begin{align}\label{lema4_s2}
I_{1,2}(s)&\lesssim \int_0^1 |s|^3 t^2 \int_{\Gamma_R}\left|\int_\Omega
\frac{1}{st}\nabla_{\boldsymbol y}\nabla_{\boldsymbol y}^{\top}g_3(\boldsymbol
x, \boldsymbol y) \cdot  \boldsymbol J(\boldsymbol y){\rm d}\boldsymbol
y\right|^2{\rm d}\gamma(\boldsymbol x) {\rm d}t\cr
&\lesssim \int_0^1 |s| \int_{\Gamma_R} \left| \int_\Omega g_3(\boldsymbol
x, \boldsymbol y)\nabla_{\boldsymbol y} \nabla_{\boldsymbol y} \cdot \boldsymbol
J{\rm d}\boldsymbol y\right|^2{\rm d}\gamma(\boldsymbol x) {\rm d}t\cr
&\lesssim  |s|e^{4R|s|}\|\boldsymbol{J}\|_{H^2(B_R)^3}^2.
\end{align}
Combing \eqref{lema4_s1} and \eqref{lema4_s2} yields \eqref{m1}.

Using the Cauchy--Schwarz inequality and \eqref{lema4_s1}, we have
\begin{align*}
I_2(s) &\lesssim \int_0^1 |s|^3 t^2\int_{\Gamma_R}\left|\int_\Omega
g_3(\boldsymbol x, \boldsymbol y; \kappa)\mathbf I_3\cdot \nabla \times
\boldsymbol{J}(\boldsymbol y){\rm
d}\boldsymbol y \right|^2{\rm
d}\gamma(\boldsymbol x){\rm d}t\\
&\lesssim \int_0^1|s|^{3}\left(\int_{B_R}|\nabla \times \boldsymbol{J}
(\boldsymbol y)|^2 { \rm d}\boldsymbol y\right)\int_{\Gamma_R}
\left(\int_\Omega\frac{e^{4R|s|}}{|\boldsymbol x- \boldsymbol y|^2}{\rm
d}\boldsymbol y \right){\rm d}\gamma(\boldsymbol x){\rm d}t\\
&\lesssim |s|^{3}e^{4R|s|}\|\boldsymbol{J}\|_{H^1(B_R)^3}^2,
\end{align*}
which complete the proof of \eqref{m2}.
\end{proof}

\begin{lemm}\label{mhje}
Let $\boldsymbol{J}\in\mathbb J_M(B_R) $. For any $s\ge 1$, the following
estimate holds:
\[
\int_s^{\infty}\kappa^{2}
\|\boldsymbol E\times\boldsymbol\nu\|^2_{\Gamma_R}{\rm
d}\kappa\lesssim s^{-(2m-5)}\| \boldsymbol{J}\|^2_{H^m(B_R)^3}.
\]
\end{lemm}

\begin{proof}
Let
\[
\int_s^{\infty} \int_{\Gamma_R}\kappa^{2}
\|\boldsymbol E\times\boldsymbol\nu\|^2_{\Gamma_R}{\rm
d}\kappa=L_1+L_2,
\]
where
\begin{align*}
L_1&=  \int_s^{\infty} \int_{\Gamma_R}\kappa^{2} |\boldsymbol{E}(\boldsymbol x,
\kappa)\times\boldsymbol \nu(\boldsymbol x)|^2{\rm d}\gamma(\boldsymbol x){\rm
d}\kappa,\\
L_2&=\int_s^{\infty} \int_{\Gamma_R}\kappa^{2}
|\boldsymbol H(\boldsymbol x, \kappa)\times \boldsymbol\nu(\boldsymbol
x))|^2{\rm d}\gamma(\boldsymbol x){\rm d}\kappa.
\end{align*}

First we estimate $L_1$. Using \eqref{es} and noting $s\geq 1$, we obtain
\begin{align*}
 L_1=\int_s^{\infty} \int_{\Gamma_R} \kappa^{2}|\boldsymbol{E}(\boldsymbol x,
\kappa)\times\boldsymbol\nu(\boldsymbol x)|^2 {\rm d}\gamma(\boldsymbol x){\rm
d}\kappa\lesssim L_{1, 1}+L_{1, 2},
\end{align*}
where
\begin{align*}
 L_{1, 1}&=\int_s^{\infty} \int_{\Gamma_R}
\kappa^{4}\bigg|\int_\Omega\frac{e^{{\rm i}\kappa|\boldsymbol
x-\boldsymbol y|}}{|\boldsymbol x-\boldsymbol
y|}\mathbf{I}_3\cdot\boldsymbol{J}(\boldsymbol y){\rm d}\boldsymbol y\bigg|^2
{\rm d}\gamma(\boldsymbol x){\rm d}\kappa,\\
L_{1, 2}&=\int_s^{\infty} \int_{\Gamma_R}
\bigg|\int_\Omega \nabla_{\boldsymbol y}\nabla_{\boldsymbol
y}^{\top}\frac{e^{{\rm i}\kappa|\boldsymbol x-\boldsymbol y|}}{|\boldsymbol
x-\boldsymbol y|}\cdot\boldsymbol{J}(\boldsymbol y){\rm d}\boldsymbol y\bigg|^2
{\rm d}\gamma(\boldsymbol x){\rm d}\kappa.
\end{align*}
Noting ${\rm supp}\boldsymbol{J}=\Omega\subset B_{\hat R}\subset B_R$, and using
integration by parts and the polar coordinates $\rho=|\boldsymbol y-\boldsymbol
x|$ originated at $\boldsymbol x$ with respect to $\boldsymbol y$, we have
\begin{align*}
L_{1,1}=\int_s^{\infty} \int_{\Gamma_R} \kappa^{4}\bigg|\int_0^{2\pi}{\rm
d}\theta\int_0^{\pi}\sin\varphi{\rm d}\varphi\int_{R-\hat
R}^{R+\hat R}\frac{e^{{\rm i}\kappa\rho}}{({\rm
i}\kappa)^m}\mathbf{I}_3\cdot\frac{\partial^m(\boldsymbol{J}\rho)}{
\partial\rho^m} {\rm d}\rho\bigg|^2 {\rm d}\gamma(\boldsymbol x){\rm d}\kappa.
\end{align*}
Consequently,
\begin{align*}
L_{1,1}\le &   \int_s^{\infty} \int_{\Gamma_R}
\kappa^{4}\bigg|\int_0^{2\pi}{\rm
d}\theta\int_0^{\pi}\sin\varphi{\rm d}\varphi\int_{R-\hat R}^{R+\hat R}
\kappa^{-m}\\
&\qquad\bigg(\bigg|\sum_{|\boldsymbol \alpha|=m}\partial_{\boldsymbol
y}^{\boldsymbol \alpha}
\boldsymbol{J}\bigg|\rho+m\bigg|\sum_{ |\boldsymbol
\alpha|=m-1}\partial_{\boldsymbol y}^{\boldsymbol
\alpha}\boldsymbol{J}\bigg|\bigg){\rm d}\rho\bigg|^2 {\rm
d}\gamma(\boldsymbol x){\rm d}\kappa\\
=& \int_s^{\infty} \int_{\Gamma_R} \kappa^{4}\bigg|\int_0^{2\pi}{\rm
d}\theta\int_0^{\pi}\sin\varphi{\rm d}\varphi\int_{R-\hat R}^{R+\hat R}
\kappa^{-m}\\
&\qquad\bigg(\bigg|\sum_{|\boldsymbol \alpha|=m}\partial_{\boldsymbol
y}^{\boldsymbol \alpha}\boldsymbol{J}\bigg|\frac{1}{\rho}
+\bigg|\sum_{|\boldsymbol \alpha|=m-1}\partial_{\boldsymbol
y}^{\boldsymbol \alpha} \boldsymbol{J}\bigg|\frac{m}{\rho^2}\bigg)\rho^2{\rm
d}\rho\bigg|^2 {\rm d}\gamma(\boldsymbol x){\rm d}\kappa\\
\leq& \int_s^{\infty} \int_{\Gamma_R} \kappa^{4}\bigg|\int_0^{2\pi}{\rm
d}\theta\int_0^{\pi}\sin\varphi{\rm d}\varphi\int_{R-\hat R}^{R+\hat R}
\kappa^{-m}\\
&\qquad\bigg(\bigg|\sum_{|\boldsymbol \alpha|=m}\partial_{\boldsymbol
y}^{\boldsymbol \alpha}\boldsymbol{J}\bigg|\frac{1}{(R-\hat R)}
+\bigg|\sum_{|\boldsymbol \alpha|=m-1}\partial_{\boldsymbol y}^{\boldsymbol \alpha}
\boldsymbol{J}\bigg|\frac{m}{(R-\hat R)^2}\bigg)\rho^2{\rm d}\rho\bigg|^2
{\rm d}\gamma(\boldsymbol x){\rm d}\kappa\\
= &\int_s^{\infty} \int_{\Gamma_R}
\omega^{4}\bigg|\int_0^{2\pi}{\rm
d}\theta\int_0^{\pi}\sin\varphi{\rm d}\varphi\int_{0}^{\infty}
\kappa^{-m}\\
&\qquad\bigg(\bigg|\sum_{|\boldsymbol \alpha|=m}\partial_{\boldsymbol
y}^{\boldsymbol \alpha}\boldsymbol{J}\bigg|\frac{1}{(R-\hat R)}
+\bigg|\sum_{|\boldsymbol \alpha|=m-1}\partial_{\boldsymbol
y}^{\boldsymbol \alpha} \boldsymbol{J}\bigg|\frac{m}{(R-\hat
R)^2}\bigg)\rho^2{\rm d}\rho\bigg|^2 {\rm d}\gamma(\boldsymbol x){\rm
d}\kappa.
\end{align*}
Changing back to the Cartesian coordinates with respect to $\boldsymbol y$, we
have
\begin{align}\label{lema5_s1}
L_{1,1} \leq& \int_s^{\infty} \int_{\Gamma_R}
\kappa^{4}\bigg|\int_{\Omega}\kappa^{-m}\\
&\qquad\bigg(\bigg|\sum_{|\boldsymbol \alpha|=m}\partial_{\boldsymbol
y}^{\boldsymbol \alpha}\boldsymbol{J}\bigg|\frac{1}{(R-\hat R)}
+\bigg|\sum_{|\boldsymbol \alpha|=m-1}\partial_{\boldsymbol
y}^{\boldsymbol \alpha} \boldsymbol{J}\bigg|\frac{m}{(R-\hat R)^2}\bigg){\rm
d} \boldsymbol y\bigg|^2 {\rm d}\gamma(\boldsymbol x){\rm d}\kappa\\
\lesssim& m\|\boldsymbol{J}\|^2_{H^m(B_R)^3}\int_s^{\infty}\kappa^{4-2m}{\rm
d}\kappa\\
\lesssim&\left(\frac{m}{2m-5}\right)s^{-(2m-5)}\|\boldsymbol{J}\|^2_{
H^m(B_R)^3}\lesssim s^{-(2m-5)}\|\boldsymbol{J}\|^2_{H^m(B_R)^3}.
\end{align}

For $L_{1,2}$, it follows from the integration by parts and similar steps for
\eqref{lema5_s1} that we obtain
\begin{align}\label{lema5_s2}
L_{1, 2}&\lesssim \int_s^{\infty} \int_{\Gamma_R}
\bigg|\int_{\Omega}\frac{e^{{\rm i}\kappa|\boldsymbol
x-\boldsymbol y|}}{|\boldsymbol x-\boldsymbol y|}\nabla_{\boldsymbol
y}\nabla_{\boldsymbol y}\cdot\boldsymbol{J}(\boldsymbol y){\rm d}\boldsymbol
y\bigg|^2 {\rm d}\gamma(\boldsymbol x){\rm d}\kappa\notag\\
&\lesssim
\int_s^{\infty}\int_{\Gamma_R}\bigg|\int_{\Omega}\frac{1}{\kappa^{m-2}}
\bigg(\bigg|\sum_{ |\boldsymbol \alpha|=m-2}\partial_{\boldsymbol
y}^{\boldsymbol \alpha}(\nabla_{\boldsymbol
y}\nabla_{\boldsymbol y}\cdot\boldsymbol{J})\bigg|\frac{1}{(R-\hat R)}\notag\\
&\qquad+\bigg|\sum_{|\boldsymbol
\alpha|=m-3}\partial_{\boldsymbol
y}^{\boldsymbol \alpha}(\nabla_{\boldsymbol
y}\nabla_{\boldsymbol y}\cdot\boldsymbol{J})\bigg|\frac{(m-2)}{(R-\hat
R)^2}\bigg){\rm d} \boldsymbol y\bigg|^2 {\rm d}\gamma(\boldsymbol x){\rm
d}\kappa\notag\\
&\lesssim (m-2)\|\boldsymbol{J}\|^2_{H^m(B_R)^3}\int_s^{\infty}\kappa^{4-2m}{\rm
d}\kappa\notag\\
&\lesssim\left(\frac{m-2}{2m-5}\right)s^{-(2m-5)} \|\boldsymbol
{J}\|^2_{H^m(B_R)^3}\lesssim
s^{-(2m-5)}\|\boldsymbol{J}\|^2_{H^m(B_R)^3}.
\end{align}
Following from the similar steps as those for \eqref{lema5_s1} and
\eqref{lema5_s2}, we may obtain from \eqref{hs} that
\begin{equation}\label{lema5_s3}
L_{2}\lesssim \int_s^{\infty} \int_{\Gamma_R}
\kappa^{2}\bigg|\int_{\Omega}\frac{e^{{\rm i}\kappa|\boldsymbol
x-\boldsymbol y|}}{|\boldsymbol x-\boldsymbol y|}\mathbf
I_3\cdot\nabla_{\boldsymbol y}\times\boldsymbol{J}(\boldsymbol y){\rm
d}\boldsymbol y\bigg|^2 {\rm d}\gamma(\boldsymbol x){\rm d}\kappa\lesssim
s^{-(2m-5)}\|\boldsymbol{J}\|^2_{H^m(B_R)^3}.
\end{equation}
Combining \eqref{lema5_s1}--\eqref{lema5_s3} completes the proof.
\end{proof}

\begin{lemm}\label{mi12}
 Let $\boldsymbol{f}\in\mathbb J_M(B_R)$. Then there exists a function
$\beta(s)$ satisfying \eqref{beta} such that
\[
 |I_1(s)+I_2(s)|\lesssim M^2 e^{(4R+1)s}\epsilon_4^{2\beta(s)},\quad\forall s\in
(K, ~\infty).
\]
\end{lemm}

\begin{proof}
 It follows from Lemma \ref{mi} that
\[
 |\left(I_1(s)+I_2(s)\right)e^{-(4R+1)|s|}|\lesssim M^2,\quad\forall
s\in\mathcal{V}.
\]
Recalling \eqref{mi1}--\eqref{mi2}, we have
\[
 |\left(I_1(s)+I_2(s)\right)e^{-(4R+1)s}|\leq\epsilon_4^2,\quad s\in [0, ~K].
\]
Using Lemma \ref{caf} shows that there exists a function
$\beta(s)$ satisfying \eqref{beta} such that
\[
 |\left(I_1(s)+I_2(s)\right)e^{-(4R+1)s}|\lesssim
M^2\epsilon_4^{2\beta},\quad\forall s\in
(K, ~\infty),
\]
which completes the proof.
\end{proof}

Now we show the proof of Theorem \ref{mrm}.

\begin{proof}
Let
\[
s=\begin{cases}
\frac{1}{((4R+3)\pi)^{\frac{1}{3}}}K^{\frac{2}{3}}|\ln\epsilon_4|^{\frac{1}{4}},
&
2^{\frac{1}{4}}
((4R+3)\pi)^{\frac{1}{3}}K^{\frac{1}{3}}<|\ln\epsilon_4|^{\frac{1}{4}},\\
K, &|\ln\epsilon_4|\leq 2^{\frac{1}{4}}((4R+3)\pi)^{\frac{1}{3}}K^{\frac{1}{3}}.
 \end{cases}
\]
If
$2^{\frac{1}{4}}((4R+3)\pi)^{\frac{1}{3}}K^{\frac{1}{3}}<|\ln\epsilon_4|^{\frac{
1}{4}}$, then we have from Lemma \ref{mi12} that
\begin{align*}
 |I_1(s)+I_2(s)|&\lesssim M^2 e^{(4R+3)s}
e^{-\frac{2|\ln\epsilon_4|}{\pi}((\frac{s}{K})^4-1)^{-\frac{1}{2}}}\cr
&\lesssim M^2
e^{\frac{(4R+3)}{((4R+3)\pi)^{\frac{1}{3}}}K^{\frac{2}{3}}|\ln\epsilon_4|^{\frac
{1}{4}}-\frac{2|\ln\epsilon_4|}{\pi}(\frac{K}{s})^2}\\
&\lesssim M^2
e^{-2\left(\frac{(4R+3)^2}{\pi}\right)^{\frac{1}{3}}K^{\frac{2}{3}}
|\ln\epsilon_4|^{\frac{1}{2}}\left(1-\frac{1}{2}
|\ln\epsilon_4|^{-\frac{1}{4}}\right)}.
\end{align*}
Noting that $\frac{1}{2} |\ln\epsilon_4|^{-\frac{1}{4}}<\frac{1}{2}$,
$\left(\frac{(4R+3)^2}{\pi}\right)^{\frac{1}{3}}>1$,  we have
\begin{align*}
 |I_1(s)+I_2(s)|&
\lesssim M^2 e^{-K^{\frac{2}{3}}|\ln\epsilon_4|^{\frac{1}{2}}}.
\end{align*}
It follows from the elementary inequality \eqref{ei} that we get
\begin{equation}\label{theo3_s1}
|I_1(s)+I_2(s)|\lesssim\frac{M^2}{\left(\frac{
K^2|\ln\epsilon_4|^{\frac{3}{2}}}{(6m-15)^3}\right)^{2m-5}}.
\end{equation}
If $|\ln\epsilon_4|\leq
2^{\frac{1}{4}}(((4R+3)\pi)^{\frac{1}{3}}K^{\frac{1}{3}}$, then $s=K$. We have
from Lemma \ref{me} and \eqref{mi1}--\eqref{mi2} that
\[
 |I_1(s)+I_2(s)|\leq \epsilon_4^2,
\]
Note that for $s>0$,
\[
I_1(s)+I_2(s)=\int_0^s \kappa^2|\boldsymbol E(\cdot,
\kappa)\times \boldsymbol\nu|^2_{\Gamma_R} {\rm
d}\kappa.
\]
Hence we obtain from Lemma \ref{mhje} and \eqref{theo3_s1} that
\begin{align*}
 &\int_0^\infty \kappa^2|\boldsymbol E(\cdot,
\kappa)\times \boldsymbol\nu|^2_{\Gamma_R}{\rm d}\gamma{\rm
d}\kappa\\
&\leq I_1(s)+I_2(s)+\int_s^\infty \kappa^2|\boldsymbol E(\cdot,
\kappa)\times \boldsymbol\nu|^2_{\Gamma_R}{\rm d}\kappa\\
&\lesssim \epsilon_4^2+\frac{M^2}{\left(\frac{
K^2|\ln\epsilon_4|^{\frac{3}{2}}}{(6m-15)^3}
\right)^{2m-5}}+\frac{M^2}{\left(2^{-\frac{1}{4}}((4R+3)\pi)^{-\frac
{1}{3}}K^{\frac{2}{3}} |\ln\epsilon_4|^{\frac{1}{4}}\right)^{2m-5}}.
\end{align*}
By Lemma \ref{mje}, we have
\[
 \|\boldsymbol{J}\|^2_{L^2(B_R)^3}\lesssim \epsilon_4^2
+\frac{M^2}{\left(\frac{ K^2|\ln\epsilon_4|^{\frac{3}{2}}}{(6m-15)^3}\right)^{
2m-5}}+\frac{M^2}{\left(\frac{K^{\frac{2}
{3}}|\ln\epsilon_4|^{\frac{1}{4}}}{(R+1)(6m-15)^3}\right)^{2m-5}}.
\]
Since $K^{\frac{2}{3}}|\ln\epsilon_4|^{\frac{1}{4}}\leq K^2
|\ln\epsilon_4|^{\frac{3}{2}}$ when $K>1$ and $|\ln\epsilon_4|>1$, we finish
the proof and obtain the stability estimate \eqref{je}.
\end{proof}

\subsection{Stability with discrete frequency data}

First we specify the discrete frequency data. For $\boldsymbol n\in\mathbb
R^d\setminus\{0\}$, let $n=|\boldsymbol n|$, denote the wavenumber
\[
\kappa_n=\frac{n\pi}{R}.
\]
We define the discrete frequency boundary data:
\[
\|\boldsymbol E(\cdot,
\kappa_n)\times\boldsymbol\nu\|^2_{\Gamma_R}=\int_{\Gamma_R}\left(|T_{\rm
M}(\boldsymbol E(\boldsymbol x, \kappa_n)\times\boldsymbol\nu)|^2+|\boldsymbol
E(\boldsymbol x, \kappa_n)\times\boldsymbol\nu|^2\right){\rm d}\gamma(\boldsymbol
x).
\]
Similarly, the Fourier coefficient $\hat{\boldsymbol J}_0$ cannot be recovered
by the discrete frequency data. It is necessary to revise the functional space.
Denote
\[
\tilde{\mathbb J}_M(B_R)=\{\boldsymbol J\in\mathbb J_M(B_R): \int_\Omega
\boldsymbol J(\boldsymbol x){\rm d}\boldsymbol x=0\}.
\]

\begin{prob}[Discrete frequency data for electromagnetic waves]\label{p5} Let
$\boldsymbol J\in \tilde{\mathbb J}_M(B_R) $. The inverse source problem of
electromagnetic waves is to determine $\boldsymbol J$ from the tangential trace
of the electric field $\boldsymbol E(\boldsymbol x, \kappa)\times\boldsymbol\nu$
for $\boldsymbol x\in\Gamma_R, \kappa\in (0, \frac{\pi}{R}]\cup\cup_{n=1}^N
\{\kappa_n\}$, where $1<N\in\mathbb N$.
\end{prob}

The following stability estimate is the main result of Problem \ref{p5}.

\begin{theo}\label{DM}
Let $\boldsymbol{E}$ be the solution of the scattering problem
\eqref{smrc}--\eqref{ef} corresponding to the source $\boldsymbol{J}\in\tilde{\mathbb
J}_M(B_R)$. Then
\begin{align}\label{dje}
\| \boldsymbol{J}\|^2_{L^2(B_R)^d}\lesssim
\epsilon_5^2+\frac{M^2}{\left(\frac{N^{\frac{5}{8}}|\ln\epsilon_6|^{\frac{1}{9}}
}{(6m-12)^3}\right)^{2m-4}},
\end{align}
where
\begin{align*}
\epsilon_5&=\left(\sum_{n\leq N}\|\boldsymbol
E(\cdot, \kappa_n)\|^2_{\Gamma_R}\right)^{\frac{1}{2}},\\
\epsilon_6 &=\sup_{\kappa \in
(0,\frac{\pi}{R}]}\|\boldsymbol E(\cdot, \kappa)\|_{\Gamma_R}.
\end{align*}
\end{theo}

\begin{rema}
The estimate for the discrete frequency data \eqref{dje} is also consistent with
the estimate for the continuous frequency data \eqref{je}. They are analogous to
the relationship between \eqref{dfe} and \eqref{fe} for elastic waves.
\end{rema}

We begin with several useful lemmas.

\begin{lemm}\label{mjed}
Let $\boldsymbol E$ be the solution of \eqref{smrc}--\eqref{ef} corresponding to
the source $\boldsymbol J\in \mathbb X(B_R)$. Then for all
$\boldsymbol n\in\mathbb Z^d\setminus\{0\}$, the Fourier coefficients
of $\boldsymbol J$ satisfy
\[
|\hat{\boldsymbol J}_{\boldsymbol n}|^2\lesssim \|\boldsymbol
E(\cdot, \kappa_n)\|^2_{\Gamma_R}.
\]
\end{lemm}

\begin{proof}
Give any $\boldsymbol n \in \mathbb Z^3$, let $\hat{\boldsymbol n}=\boldsymbol
n/n$. Consider the following electric and magnetic plane waves:
\[
\boldsymbol{E}^{\rm inc}(\boldsymbol x) = \boldsymbol{p} e^{-{\rm
i}\kappa_n \boldsymbol{x} \cdot\hat{\boldsymbol n}}=\boldsymbol{p} e^{-{\rm
i}(\frac{\pi}{R}) \boldsymbol{x} \cdot\boldsymbol{n}}\quad\text{and}\quad
\boldsymbol{H}^{\rm inc}(\boldsymbol x) = \boldsymbol{q} e^{-{\rm
i}\kappa_n\boldsymbol{x} \cdot \hat{\boldsymbol n}}=\boldsymbol{q} e^{-{\rm
i}(\frac{\pi}{R})\boldsymbol{x} \cdot{\boldsymbol n}},
\]
where $\boldsymbol  p$ and $\boldsymbol q$ are chosen such that
$\{\hat{\boldsymbol n}, \boldsymbol p, \boldsymbol q\}$ form an orthonormal
basis in $\mathbb R^3$. It is easy to verify that $\boldsymbol{E}^{\rm inc}$ and
$\boldsymbol{H}^{\rm inc}$ satisfy the Maxwell equations:
\begin{equation}\label{epwd}
\nabla \times (\nabla \times \boldsymbol{E}^{\rm inc}) -
\kappa_n^2 \boldsymbol{E}^{\rm inc} = 0
\end{equation}
and
\[
\nabla \times (\nabla \times \boldsymbol{H}^{\rm inc}) -
\kappa_n^2 \boldsymbol{H}^{\rm inc} = 0.
\]

Multiplying the both sides of \eqref{ef} by $\boldsymbol{E}^{\rm inc}$, using
the integration by parts over $B_R$ and \eqref{epwd}, we obtain
\begin{align*}
{\rm i}n\int_{B_R} (\boldsymbol{p}e^{- {\rm i}(\frac{\pi}{R})\boldsymbol{n}
\cdot \boldsymbol{x}}) \cdot \boldsymbol{J}(\boldsymbol x){\rm d}\boldsymbol x
= -\int_{\Gamma_R} \big({\rm i}\kappa
T_{\rm M}(\boldsymbol{E}(\boldsymbol{x},\kappa_n) \times
\boldsymbol\nu) \cdot \boldsymbol{E}^{\rm inc}\\
+(\boldsymbol{E}(\boldsymbol{x}, \kappa_n) \times
\boldsymbol\nu)\cdot(\nabla \times \boldsymbol{E}^{\rm
inc})\big){\rm d}\gamma.
\end{align*}
A simple calculation yields that
\begin{align*}
\nabla \times \boldsymbol{E}^{\rm inc} = -{\rm i}
\boldsymbol{n}\times \boldsymbol{p}e^{-{\rm i}\kappa_n\boldsymbol x\cdot\hat{\boldsymbol n}},
\end{align*}
which gives
\begin{align*}
|\nabla \times \boldsymbol{E}_{\boldsymbol n}^{\rm inc}| = n.
\end{align*}
Combining the above estimates leads to
\[
 |\boldsymbol p\cdot\hat{\boldsymbol J}_{\boldsymbol n}|^2
\lesssim \int_{\Gamma_R}(|T_{\rm M}(\boldsymbol{E}(\boldsymbol
x, \kappa_n)\times\boldsymbol \nu)|^2+|\boldsymbol{E}(\boldsymbol
x, \kappa_n)\times  \boldsymbol\nu|^2){\rm d}\gamma(\boldsymbol x)\lesssim\|\boldsymbol E(\cdot,
\kappa_n)\|^2_{\Gamma_R}.
\]
Similarly, we have
\[
 |\boldsymbol q_{\boldsymbol n}\cdot\hat{\boldsymbol J}_{\boldsymbol n}|^2
\lesssim \int_{\Gamma_R}(|T_{\rm M}(\boldsymbol{E}(\boldsymbol
x, \kappa_n)\times\boldsymbol \nu)|^2+|\boldsymbol{E}(\boldsymbol x,
\kappa_n)\times \boldsymbol\nu|^2){\rm d}\gamma(\boldsymbol
x)\lesssim\|\boldsymbol E(\cdot,\kappa_n)\|^2_{\Gamma_R}.
\]

On the other hand, since $\boldsymbol J$ has a compact support $\Omega$
contained in $B_R$ and $\boldsymbol J \in \mathbb X(B_R)$, we obtain that
$\boldsymbol{J}$ is a weak solution of the Maxwell system:
\begin{align*}
\nabla \times( \nabla \times \boldsymbol{J}) - \kappa_n^2 \boldsymbol{J} = 0 \quad
{\rm in} ~B_R.
\end{align*}
Multiplying the above equation by $\hat{\boldsymbol n} e^{-{\rm i}\kappa_n
\boldsymbol x\cdot\hat{\boldsymbol n}}$ and using integration by parts, we get
\begin{align*}
\int_{B_R}(\nabla \times \boldsymbol{J}) \cdot (\nabla \times (\hat{\boldsymbol
n} e^{{\rm i}\kappa_n  \boldsymbol x\cdot\hat{\boldsymbol n}}) {\rm
d}\boldsymbol x= \kappa_n^2\hat{\boldsymbol n}\cdot \int_{B_R}
\boldsymbol{J}(\boldsymbol x) e^{-{\rm i}(\frac{\pi}{R})
\boldsymbol x\cdot\boldsymbol n}{\rm d}\boldsymbol x =
\kappa_n^2 \hat{\boldsymbol n}\cdot \hat{\boldsymbol
J}_{\boldsymbol n}.
\end{align*}
Noting $\nabla \times (\hat{\boldsymbol n} e^{-{\rm i}\kappa_n  \boldsymbol
x\cdot\hat{\boldsymbol n}})=-{\rm
i}\kappa_n\hat{\boldsymbol n}\times\hat{\boldsymbol n} e^{-{\rm
i}\kappa_n\boldsymbol x\cdot\hat{\boldsymbol n}}=0$, we
get $\hat{\boldsymbol n} \cdot \hat{\boldsymbol J}_{\boldsymbol n}= 0,$
which yields from the Pythagorean theorem that
\[
|\hat{\boldsymbol J}_{\boldsymbol n}|^2 = |\boldsymbol{p} \cdot \hat{\boldsymbol
J}_{\boldsymbol n}|^2 + |\boldsymbol{q}\cdot \hat{\boldsymbol J}_{\boldsymbol
n}|^2\lesssim \|\boldsymbol E(\cdot,\kappa_n)\|^2_{\Gamma_R},
\]
which completes the proof.
\end{proof}

\begin{lemm}\label{jhfd}
Let $\boldsymbol J \in H^m(B_R)^{3}.$ For any $N_0\in\mathbb N$, the following estimate holds:
\[
\sum_{n= N_0}^{\infty} |\hat{\boldsymbol J}_{\boldsymbol n}|^2
\lesssim N_0^{-(2m-4)}\|\boldsymbol J\|^2_{H^m(B_R)^3}.
\]
\end{lemm}

\begin{proof}
Let $\boldsymbol n=(n_1, n_2, n_3)$ and choose $n_j =\max\{n_1, n_2,
n_3\}$. Then we have $n^2 \leq 3 n_j^2$, which means that $n_j^{-2m}
\leq 3^m n^{-2m}$. Let $\boldsymbol J=(J_1, J_2, J_3)$. Noting ${\rm
supp}\boldsymbol J \subset B_R\subset U_R$ and using integration by parts, we
obtain
\begin{align*}
\left|\int_{B_R} J_1({\boldsymbol x})e^{-{\rm i} (\frac{\pi}{R})\boldsymbol
n\cdot\boldsymbol x} {\rm d} \boldsymbol x\right|^2
\lesssim \left|\int_{B_R} n_j^{-m} e^{-{\rm i}(\frac{\pi}{R}) \boldsymbol
n\cdot\boldsymbol x} \partial^m_{x_j} J_1({\boldsymbol x}) {\rm d} \boldsymbol
x\right|^2\lesssim n^{-2m}\|\boldsymbol J\|^2_{H^m(B_R)^d}.
\end{align*}
Hence
\[
|\hat{\boldsymbol J}_{\boldsymbol n}|^2 \lesssim \left|\int_{B_R} \boldsymbol
J({\boldsymbol x})  e^{-{\rm i}(\frac{\pi}{R}) \boldsymbol n\cdot\boldsymbol
x}{\rm d} \boldsymbol x\right|^2\lesssim n^{-2m}\|\boldsymbol
J\|^2_{H^m(B_R)^d}.
\]
Noting that there are at most $O(n^3)$ elements in \{$\boldsymbol n
\in \mathbb Z^3$, $|\boldsymbol n| = n$\}, we get
\begin{align*}
\sum_{n= N_0}^{\infty} |\hat{\boldsymbol J}_{\boldsymbol
n}|^2 &\lesssim \left(\sum_{n=N_0}^\infty
n^{(3-2m)}\right)\|\boldsymbol J\|^2_{H^m(B_R)^d}\\
&\lesssim \left(\int_0^\infty (N_0+t)^{(3-2m)}{\rm d}t
\right)\|\boldsymbol J\|^2_{H^m(B_R)^d}\\
&\lesssim\frac{N_0^{-(2m-4)}}{(2m-4)}\|\boldsymbol
J\|^2_{H^m(B_R)^d} \lesssim N_0^{-(2m-4)}\|\boldsymbol
J\|^2_{H^m(B_R)^{d}}.
\end{align*}
which completes the proof.
\end{proof}

\begin{lemm}\label{FFM}
Let $\boldsymbol E$ be the solution of \eqref{smrc}--\eqref{ef} corresponding
to the source $\boldsymbol J\in \mathbb X(B_R)$. For any $\kappa
\in (0,\frac{\pi}{R}]$ and  $\boldsymbol d\in\mathbb S^{d-1}$,
the following estimate holds:
\begin{align*}
\left|\int_{B_R} \boldsymbol J(\boldsymbol x) e^{\rm i \kappa \boldsymbol x
\cdot \boldsymbol d}{\rm d} \boldsymbol x\right|^2\lesssim \epsilon_6^2.
\end{align*}
\end{lemm}

\begin{proof}
Let $\boldsymbol p, \boldsymbol q\in\mathbb S^{d-1}$ such that $\boldsymbol p
\cdot \boldsymbol d = 0$ and $\boldsymbol q = \boldsymbol p \times \boldsymbol
d.$ Consider the electric plane wave
$\boldsymbol E^{\rm inc}=\boldsymbol{p} e^{-{\rm
i}\kappa\boldsymbol{x} \cdot \boldsymbol{d}}$ and magnetic plane wave
$\boldsymbol H^{\rm inc}=\boldsymbol{q} e^{-{\rm i}\kappa\boldsymbol{x}\cdot
\boldsymbol{d}}$. Noting ${\rm supp}\boldsymbol J\subset B_R$ and using
similar arguments as those in Lemma \ref{mjed}, we get
\begin{align*}
|\boldsymbol{p} \cdot \hat{\boldsymbol J}(\kappa\boldsymbol d)|^2
&=\left|\boldsymbol p\cdot \int_{B_R} \boldsymbol J(\boldsymbol x) e^{-{\rm
i}\kappa\boldsymbol{x} \cdot \boldsymbol{d}} {\rm d}\boldsymbol x\right|^2\\
&\lesssim \int_{\Gamma_R}(|T_{\rm M}(\boldsymbol{E}(\boldsymbol
x,\kappa)\times\boldsymbol \nu)|^2+|\boldsymbol{E}(\boldsymbol x,\kappa)\times
\boldsymbol\nu|^2){\rm d}\gamma(\boldsymbol x)\lesssim\|\boldsymbol E(\cdot,
\kappa)\|^2_{\Gamma_R},
\end{align*}
and
\begin{align*}
 |\boldsymbol{q}\cdot \hat{\boldsymbol J}(\kappa\boldsymbol d)|^2 &=
 \left|\boldsymbol q\cdot\int_{B_R} \boldsymbol J(\boldsymbol x) e^{-{\rm
i}\kappa\boldsymbol{x} \cdot \boldsymbol{d}}{\rm d}\boldsymbol x\right|^2\\
&\lesssim \int_{\Gamma_R}(|T_{\rm M}(\boldsymbol{E}(\boldsymbol
x,\kappa)\times\boldsymbol \nu)|^2+|\boldsymbol{E}(\boldsymbol x,\kappa)\times
\boldsymbol\nu|^2){\rm d}\gamma(\boldsymbol x)\lesssim\|\boldsymbol E(\cdot,
\kappa)\|^2_{\Gamma_R}.
\end{align*}
Hence we have from the Pythagorean theorem that
\[
|\hat{\boldsymbol J}(\kappa\boldsymbol d)|^2 = |\boldsymbol{p} \cdot
\hat{\boldsymbol J}(\kappa\boldsymbol d)|^2 + |\boldsymbol{q}\cdot
\hat{\boldsymbol J}(\kappa\boldsymbol d)|^2\lesssim \epsilon_6^2,
\]
which completes the proof.
\end{proof}

\begin{lemm}\label{ni12M}
 Let $\boldsymbol{J}\in\tilde{\mathbb J}_M(B_R) $. Then there exists a function
$\beta(s)$ satisfying \eqref{betad} such that
\[
 \left|\int_{B_R} \boldsymbol J(\boldsymbol x) e^{-{\rm i}(\frac{\pi}{R})
\boldsymbol n \cdot \boldsymbol x} {\rm d} \boldsymbol x\right|^2\lesssim M^2
e^{2nR}\epsilon_6^{2n\beta(\frac{n\pi}{R})}, \quad\forall n\in (1, ~\infty).
\]
\end{lemm}

\begin{proof}
We fix $\boldsymbol d\in\mathbb S^{d-1}$ and consider $\boldsymbol
n \in \mathbb Z^3$ which  parallel to $\boldsymbol
d$. Define
\[
I(s) = \left|\int_{B_R} \boldsymbol J(\boldsymbol x) e^{-{\rm i} s \boldsymbol d
\cdot \boldsymbol x}{\rm d} \boldsymbol x\right|^2.
\]
It is easy to show from the Cauchy--Schwarz inequality that there exists a
constant $C$ depending on $R, d$ such that
\[
I(s)\leq C(R,d) e^{2|s|R} M^2, \quad\forall s\in\mathcal{V},
\]
which gives
\[
e^{-2|s|R}I(s) \lesssim M^2, \quad \forall s\in\mathcal{V}.
\]
Using Lemma \ref{FFM} yields
\[
e^{-2|s|R}\left|\int_{B_R}\boldsymbol J(\boldsymbol x) e^{-{\rm i} s \boldsymbol
d \cdot \boldsymbol x} {\rm d} \boldsymbol x\right|^2 \leq
\epsilon_6^2, \quad \forall s\in [0, ~\frac{\pi}{R}].
\]
An direct application of Lemma \ref{caf} shows that there exists a function
$\beta(s)$ satisfying \eqref{betad} such that
\[
 |I(s)e^{-2sR}|\lesssim M^2\epsilon_6^{2\beta},\quad\forall
s\in (\frac{\pi}{R}, ~\infty).
\]
Hence
\[
 |I(s)|\lesssim M^2e^{2sR}\epsilon_6^{2\beta},\quad\forall
s\in (\frac{\pi}{R}, ~\infty).
\]
Noting that the constant $C(R,d)$ does not depend on $\boldsymbol
d$, we have obtained that for all $\boldsymbol n\in\mathbb Z^3$ with
$n > 1$ such that
\[
\left|\int_{B_R} \boldsymbol J(\boldsymbol x) e^{-{\rm i}(\frac{\pi}{R})
\boldsymbol n \cdot \boldsymbol x}  {\rm d} \boldsymbol x\right|^2 =
\left|\int_{B_R}\boldsymbol J(\boldsymbol x)  e^{-{\rm i}(\frac{n\pi}{R})
\hat{\boldsymbol n} \cdot \boldsymbol x}
{\rm d} \boldsymbol x\right|^2
\lesssim M^2 e^{2nR}\epsilon_6^{2n\beta(\frac{n\pi}{R})},
\]
which completes the proof.
\end{proof}

The proof of Theorem \ref{DM} is similar to that for Theorem \ref{mrnd}. We
briefly present it for completeness.

\begin{proof}
Applying Lemma \ref{fst} and the Parseval identity, we have
\begin{align*}
\int_{B_R} |\boldsymbol J|^2{\rm d}\boldsymbol x
\lesssim\sum_{n=0}^{N_0} |\hat{\boldsymbol J}_{\boldsymbol n} |^2
+\sum_{n=N_0+1}^\infty |\hat{\boldsymbol J}_{\boldsymbol n} |^2.
\end{align*}
Let
\[
N_0=
 \begin{cases}
  [N^{\frac{3}{4}}|\ln\epsilon_6|^{\frac{1}{9}}], &
N^{\frac{3}{8}}<\frac{1}{2^{\frac{5}{6}}
\pi^{\frac{2}{3}}}|\ln\epsilon_6|^{\frac{1}{9}},\\
N,&N^{\frac{3}{8}}\geq \frac{1}{2^{\frac{5}{6}}
\pi^{\frac{2}{3}}}|\ln\epsilon_6|^{\frac{1}{9}}
 \end{cases}.
\]
Using Lemma \ref{ni12M} leads to
\begin{align*}
& \Bigl|\int_{B_R}\boldsymbol J(\boldsymbol x)e^{-{\rm i}(\frac{\pi}{R})\boldsymbol n
\cdot \boldsymbol x} {\rm d}\boldsymbol x\Bigr|^2\lesssim M^2
e^{2nR}\epsilon_2^{2n\beta}\lesssim M^2
e^{2nR}e^{2n\beta|\ln\epsilon_6|}\\
&\qquad\lesssim M^2 e^{2nR}e^{-\frac{2}{\pi}(n^4-1)^{-\frac{1}{2}}|\ln\epsilon_6|}\lesssim
M^2 e^{2nR-\frac{2}{\pi}n^{-2}|\ln\epsilon_6|}\\
&\qquad\lesssim M^2 e^{-\frac{2}{\pi}n^{-2}|\ln\epsilon_6|(1-2\pi n^3|\ln\epsilon_6|^{
-1})},\quad \forall\,n\in (2^{\frac{1}{4}},~\infty).
\end{align*}
Therefore
\begin{equation}\label{c2m}
 \Bigl|\int_{B_R}\boldsymbol J(\boldsymbol x) e^{-{\rm i}(\frac{\pi}{R})\boldsymbol n
\cdot \boldsymbol x} {\rm d}\boldsymbol x\Bigr|^2\lesssim M^2
e^{-\frac{2}{\pi^3}N_0^{-2}|\ln\epsilon_6|(1-2\pi^4 N_0^3|\ln\epsilon_6|^{
-1})},\quad\forall~n\in (2^{\frac{1}{4}},~N_0\pi].
\end{equation}

If $N^{\frac{3}{8}}<\frac{1}{2^{\frac{5}{6}}
\pi^{\frac{2}{3}}}|\ln\epsilon_6|^{\frac{1}{9}}$, then $2\pi^4
N_0^3|\ln\epsilon_6|^{-1}<\frac{1}{2}$ and
\begin{equation}\label{c3m}
 e^{-\frac{2}{\pi^3}\frac{|\ln\epsilon_6|}{N_0^2}}\leq
e^{-\frac{2}{\pi^3}\frac{|\ln\epsilon_6|}{N^{\frac{3}{2}}|\ln\epsilon_6|^{\frac{
2}{9}}}}\leq
e^{-\frac{2}{\pi^3}\frac{|\ln\epsilon_6|^{\frac{7}{9}}}{N^\frac{3}{2}}}\leq
e^{-\frac{2}{\pi^3}\frac{2^5\pi^4
|\ln\epsilon_6|^{\frac{1}{9}} N^{\frac{9}{4}}}{N^{\frac{3}{2}}}}= e^{-64\pi
|\ln\epsilon_6|^{\frac{1}{9}}N^{\frac{3}{4}}}.
\end{equation}
Combining \eqref{c2m} and \eqref{c3m}, we obtain
\begin{align*}
  \Bigl|\int_{B_R}\boldsymbol J(\boldsymbol x) e^{-{\rm i}(\frac{\pi}{R})\boldsymbol n
\cdot \boldsymbol x} {\rm d}\boldsymbol x\Bigr|^2&\lesssim M^2
e^{-\frac{2}{\pi^3}N_0^{-2}|\ln\epsilon_6|(1-2\pi^4 N_0^3|\ln\epsilon_6|^{
-1})}\\
&\lesssim M^2 e^{-\frac{1}{\pi^3}N_0^{-2}|\ln\epsilon_6|}\lesssim M^2 e^{-32\pi
|\ln\epsilon_6|^{\frac{1}{9}} N^{\frac{3}{4}}},\quad\forall n\in
(2^{\frac{1}{4}},~N_0\pi].
\end{align*}
Using \eqref{ei}, we have
\[
  \Bigl|\int_{B_R}\boldsymbol J(\boldsymbol x) e^{-{\rm i}(\frac{\pi}{R})\boldsymbol n
\cdot \boldsymbol x} {\rm d}\boldsymbol x\Bigr|^2\lesssim
M^2\frac{1}{\left(\frac{|\ln\epsilon_6|^{\frac{1}{3}}N^{\frac{9}{4}}}{(6m-12)^3
}\right)^{2m-4}},
\quad n=1, \dots, N_0.
\]
Consequently,
\begin{align*}
& \sum_{n=0}^{N_0}  \Bigl|\int_{B_R}\boldsymbol J(\boldsymbol x)
e^{-{\rm i}(\frac{\pi}{R})\boldsymbol n \cdot \boldsymbol x} {\rm d}\boldsymbol x\Bigr|^2\lesssim
\frac{M^2 N_0}{\left(\frac{|\ln\epsilon_6|^{\frac{1}{3}}N^{\frac{9}{4}}}{(6m-12)^3
}\right)^{2m-4}}\\
&\lesssim \frac{M^2 N^{\frac{3}{4}}|\ln\epsilon_6|^{\frac{1}{9}}}{\left(\frac{
|\ln\epsilon_6|^{\frac{1}{3}}N^{\frac{9}{4}}}{(6m-12)^3}\right)^{2m-4}}\lesssim
\frac{M^2}{\left(\frac{|\ln\epsilon_6|^{\frac{2}{9}}N^{\frac{3}{2}}}{(6m-12)^3}
\right)^{2m-4}}\lesssim \frac{M^2}{\left(\frac{|\ln\epsilon_6|^{\frac{1}{9}}N^{\frac{3}{2}}}{(6m-12)^3}
\right)^{2m-4}}.
\end{align*}
Here we have noted that $|\ln\epsilon_6|>1$ when
$N^{\frac{3}{8}}<\frac{1}{2^{\frac{5}{6}}
\pi^{\frac{2}{3}}}|\ln\epsilon_6|^{\frac{1}{9}}$.
If $N^{\frac{3}{8}}<\frac{1}{2^{\frac{5}{6}}
\pi^{\frac{2}{3}}}|\ln\epsilon_6|^{\frac{1}{9}}$, we have
\[
\left(\bigl[|\ln\epsilon_2|^{\frac{1}{9}}N^{\frac{3}{4}}\bigr]
+1\right)^{2m-4}\geq\left(|\ln\epsilon_2|^{\frac{1}{9}}N^{\frac{3}{4}}\right)^{2m-4}.
\]

If $N^{\frac{3}{8}}\geq \frac{1}{2^{\frac{5}{6}}
\pi^{\frac{2}{3}}}|\ln\epsilon_6|^{\frac{1}{9}}$, then $N_0=N$. It follows from
Lemma \ref{mjed} that
\[
 \sum_{n=0}^{N_0}  \Bigl|\int_{B_R}\boldsymbol J(\boldsymbol x)e^{-{\rm
i}(\frac{\pi}{R})\boldsymbol n \cdot \boldsymbol x} {\rm d}\boldsymbol
x\Bigr|^2\lesssim\epsilon_5^2.
\]
Combining the above estimates and Lemma \ref{jhfd}, we obtain
\begin{align*}
 \sum_{n=0}^{\infty}  \Bigl|\int_{B_R}\boldsymbol J(\boldsymbol
x)e^{-{\rm i}(\frac{\pi}{R})\boldsymbol n \cdot \boldsymbol x} {\rm d}\boldsymbol
x\Bigr|^2\lesssim \epsilon_5^2+
\frac{M^2}{\left(\frac{|\ln\epsilon_6|^{\frac{1}{9}}N^{\frac{3}{2}}}{(6m-12)^3}
\right)^{2m-4}}\\
+\frac{M^2}{\bigl(|\ln\epsilon_6|^{\frac{1}{9}}N^{\frac{3}{4}}
\bigr)^{2m-4}}+\frac{M^2(2^{\frac{5}{6}}\pi^{\frac{2}{3}})^{2m-4}}{\left(|\ln\epsilon_6|^{\frac{1}
{9}}N^{\frac{5}{8}}\right)^{2m-4}}.
\end{align*}
Noting that $N^{\frac{5}{8}}\le N^{\frac{3}{4}} \le N^{\frac{3}{2}}$ and
$2^{\frac{5}{6}}\pi^{\frac{2}{3}}\le (6m-12)^3$, $\forall m\geq 3$.
The proof is completed by combining the above estimates.
\end{proof}

\section{Conclusion}

We have presented a unified stability theory of the inverse source
problems for both elastic and electromagnetic waves. For elastic waves, the
increasing stability is achieved to reconstruct the external force. For
electromagnetic waves, the increasing stability is obtained to reconstruct the
radiating electric current density. The analysis requires the Dirichlet data
only at multiple frequencies. The stability estimates consist of the data
discrepancy and the high frequency tail. The result shows that the ill-posedness
of the inverse source problems decreases as the frequency increases for the
data. A possible continuation of this work is to investigate the stability with
a limited aperture data, i.e., the data is only available on a part of the
boundary. Since the Neumann data cannot be represented via the limited Dirichlet
data by using the DtN map, a new technique must be developed, and both the
Dirichlet and Neumann data are required in order to obtain the increasing
stability. Another interesting direction is to study the stability in the
inverse source problems for inhomogeneous media, where the analytical Green
tensors are not available any more and the present method may not be directly
applicable. We also point out even more challenging inverse medium
and obstacle scattering problems. These nonlinear problems are largly open.

\appendix

\section{differential operators}\label{do}

We list the notations for differential operators used in
this paper.

First we introduce the notation in two-dimensions. Let $\boldsymbol x=(x_1,
x_2)^\top$. Let $u$ and $\boldsymbol u=(u_1, u_2)^\top$ and be a scalar and
vector function, respectively. We introduce the gradient and the Jacobi matrix:
\[
\nabla u=(\partial_{x_1}u, \partial_{x_2}u)^\top,\quad
 \nabla\boldsymbol u=\begin{bmatrix}
                      \partial_{x_1} u_1 & \partial_{x_2} u_1\\
                      \partial_{x_1} u_2 & \partial_{x_2} u_2
                     \end{bmatrix}
\]
and the scalar curl and the vector curl:
\[
 {\rm curl}\boldsymbol u=\partial_{x_1} u_2 - \partial_{x_2}u_1,\quad
{\bf curl} u=(\partial_{x_2}u, -\partial_{x_1}u)^\top.
\]
It is easy to verify that
\[
\nabla \nabla^{\top}u=\begin{bmatrix}
\partial_{x_1x_1} u& \partial_{x_1x_2}u \\
\partial_{x_2x_1} u& \partial_{x_2x_2}u
\end{bmatrix}
\]
and
\[
 \nabla\nabla\cdot\boldsymbol u=\begin{bmatrix}
\partial_{x_1x_1} u_1 + \partial_{x_1x_2}u_2 \\
\partial_{x_2x_1} u_1 + \partial_{x_2x_2}u_2
\end{bmatrix}.
\]

Next we introduce the notation in three-dimensions. Let $\boldsymbol x=(x_1,
x_2, x_3)^\top$. Let $u$ and $\boldsymbol u=(u_1, u_2, u_3)^\top$ and be a
scalar and vector function, respectively. We introduce the gradient, the curl,
and the Jacobi matrix:
\[
\nabla u=(\partial_{x_1}u, \partial_{x_2}u, \partial_{x_3}u)^\top,\quad
\nabla\times\boldsymbol u=\begin{bmatrix}
\partial_{x_2}u_3-\partial_{x_3}u_2\\
\partial_{x_3}u_1-\partial_{x_1}u_3\\
\partial_{x_1}u_2-\partial_{x_2}u_1
\end{bmatrix},
\quad
 \nabla\boldsymbol u=\begin{bmatrix}
                      \partial_{x_1} u_1 & \partial_{x_2} u_1 & \partial_{x_3}
u_1\\
                      \partial_{x_1} u_2 & \partial_{x_2} u_2 &
\partial_{x_3}u_2\\
                      \partial_{x_1} u_3 & \partial_{x_2} u_3 &
\partial_{x_3}u_3
                     \end{bmatrix}.
\]
It can be also verified that
\[
\nabla\nabla^{\top}u=\begin{bmatrix}
\partial_{x_1x_1} u& \partial_{x_1x_2}u & \partial_{x_1x_3} u\\
\partial_{x_2x_1} u& \partial_{x_2x_2}u & \partial_{x_2x_3} u\\
\partial_{x_3x_1} u& \partial_{x_3x_2}u & \partial_{x_3x_3}u
\end{bmatrix}
\]
and
\[
\nabla\nabla\cdot\boldsymbol{u}=
\begin{bmatrix}
\partial_{x_1x_1}u_1 + \partial_{x_1x_2}u_2 + \partial_{x_1x_3}u_3 \\
\partial_{x_2x_1}u_1 + \partial_{x_2x_2}u_2 + \partial_{x_2x_3}u_3 \\
\partial_{x_3x_1}u_1 + \partial_{x_3x_2}u_2 + \partial_{x_3x_3}u_3
\end{bmatrix}.
\]

\section{Helmholtz decomposition}\label{hd}

We present the Helmholtz decomposition for the displacement
which is used to introduce the Kupradze--Sommerfeld radiation condition
in Section 2. Since the source $\boldsymbol f$ has a compact support $\Omega$,
the elastic wave equation \eqref{ne} reduces to
\begin{equation}\label{ext}
\mu\Delta\boldsymbol{u}+ (\lambda + \mu)\nabla\nabla\cdot\boldsymbol{u} +
\omega^2\boldsymbol{u} =0\quad\text{in}~\mathbb{R}^d\setminus\bar{\Omega}.
\end{equation}

First we introduce the Helmholtz decomposition in the two-dimensions. For any
solution $\boldsymbol u$ of \eqref{ext}, we let
\begin{equation}\label{hd2}
 \boldsymbol u=\nabla\phi + {\bf curl}\psi,
\end{equation}
where $\phi$ and $\psi$ are scalar potential functions. Substituting
\eqref{hd2} into \eqref{ext} gives
\[
 \nabla((\lambda+2\mu)\Delta \phi+\omega^2\phi)+{\bf
curl}(\mu\Delta\psi+\omega^2\psi)=0,
\]
which is fulfilled if $\phi$ and $\psi$ satisfy the Helmholtz equations:
\begin{equation}\label{he2}
 \Delta\phi+\kappa^2_{\rm p}\phi=0,\quad \Delta\psi+\kappa^2_{\rm s}\psi=0.
\end{equation}
It follows from \eqref{hd2} and \eqref{he2} that we get
\[
 \nabla\cdot\boldsymbol u=\Delta\phi=-\kappa^2_{\rm p}\phi,\quad {\rm
curl}\boldsymbol u=-\Delta\psi=\kappa^2_{\rm s}\psi.
\]
Using \eqref{hd2} again yields
\[
 \boldsymbol u=\boldsymbol u_{\rm p} +\boldsymbol u_{\rm s},
\]
where $\boldsymbol u_{\rm p}$ and $\boldsymbol u_{\rm s}$ are the compressional
part the shear part, respectively, given by
\[
 \boldsymbol u_{\rm p}=-\frac{1}{\kappa_{\rm p}^2}\nabla \nabla\cdot\boldsymbol
u,\quad \boldsymbol u_{\rm s}=\frac{1}{\kappa^2_{\rm s}}{\bf curl}{\rm
curl}\boldsymbol u.
\]

Next we introduce the Helmholtz decomposition in the three-dimensions. For any
solution $\boldsymbol u$ of \eqref{ext}, the Helmholtz decomposition reads
\begin{equation}\label{hd3}
 \boldsymbol u=\nabla\varphi + \nabla\times\boldsymbol\psi,\quad
\nabla\cdot\boldsymbol\psi=0,
\end{equation}
where $\varphi$ is a scalar potential function and $\boldsymbol\psi$ is a
vector potential function. Substituting \eqref{hd3} into \eqref{ext} gives
\[
 \nabla((\lambda+2\mu)\Delta\varphi+\omega^2\varphi)
+\nabla\times(\mu\Delta\boldsymbol\psi+\omega^2\boldsymbol\psi)=0,
\]
which implies that $\varphi$ and $\boldsymbol\psi$ satisfy the Helmholtz
equations:
\begin{equation}\label{he3}
 \Delta\varphi+\kappa^2_{\rm p}\varphi=0,\quad
\Delta\boldsymbol\psi+\kappa^2_{\rm s}\boldsymbol\psi=0.
\end{equation}
Similarly, we have from \eqref{hd3} and \eqref{he3} that
\[
 \boldsymbol u=\boldsymbol u_{\rm p} +\boldsymbol u_{\rm s},
\]
where
\[
 \boldsymbol u_{\rm p}=-\frac{1}{\kappa_{\rm p}^2}\nabla \nabla\cdot\boldsymbol
u,\quad \boldsymbol u_{\rm s}=\frac{1}{\kappa^2_{\rm
s}}\nabla\times(\nabla\times\boldsymbol u).
\]

\section{Sobolev spaces}\label{ss}

Denote by $L^2(B_R)$ the Hilbert space of square integrable functions. Denote
by $H^m(B_R), m\in\mathbb N$ the Sobolev space which consists of square
integrable weak derivatives up to $m$th order and has the norm characterized by
\[
 \|u\|^2_{H^m(B_R)}=\sum_{|\alpha|\leq m}\int_{B_R}|D^\alpha u(\boldsymbol
x)|{\rm d}\boldsymbol x.
\]
Introduce the Sobolev space
\[
 H({\rm curl}, B_R)=\{\boldsymbol u\in L^2(B_R)^3, ~ \nabla\times\boldsymbol
u\in L^2(B)^3\},
\]
which is equipped with the norm
\[
 \|\boldsymbol u\|_{H({\rm curl}, B_R)}=\left(\|\boldsymbol
u\|^2_{L^2(B_R)^3}+\|\nabla\times\boldsymbol u\|^2_{L^2(B_R)^3}\right)^{1/2}.
\]

Let $H^s(\Gamma_R), s\in\mathbb R$ be the standard trace functional
space. Given $u(\boldsymbol x)\in L^2(\Gamma_R), \boldsymbol x\in\mathbb R^2$,
it has the Fourier expansion
\[
 u(R, \theta)=\sum_{n\in\mathbb Z}\hat{u}_n e^{{\rm i}n\theta},\quad
\hat{u}_n=\frac{1}{2\pi}\int_0^{2\pi}u(R, \theta)e^{-{\rm i}n\theta}{\rm
d}\theta.
\]
The $H^s(\Gamma_R)$-norm is characterized by
\[
 \|u\|^2_{H^s(\Gamma_R)}=\sum_{n\in\mathbb Z}(1+n^2)^s |\hat{u}_n|^2.
\]
Given $u(\boldsymbol x)\in L^2(\Gamma_R), \boldsymbol x\in\mathbb R^3$,
it has the Fourier expansion
\[
 u(R, \theta, \varphi)=\sum_{n=0}^\infty\sum_{m=-n}^n \hat{u}_n^m Y_n^m(\theta,
\varphi), \quad \hat{u}_n^m =\int_{\Gamma_R}u(R, \theta,
\varphi)\bar{Y}_n^m(\theta, \varphi){\rm d}\gamma,
\]
where $Y_n^m$ is the spherical harmonics of order $n$. The
$H^s(\Gamma_R)$-norm is characterized by
\[
 \|u\|^2_{H^s(\Gamma_R)}=\sum_{n=0}^\infty\sum_{m=-n}^m
(1+n(n+1))^s|\hat{u}_n^m|^2.
\]
Define a tangential trace functional space
\[
 H^{-1/2}({\rm curl}, \Gamma_R)=\{\boldsymbol u\in H^{-1/2}(\Gamma_R)^3:
\boldsymbol u\cdot\boldsymbol\nu=0~\text{on}~\Gamma_R, ~ {\rm
curl}_{\Gamma_R}\boldsymbol u\in H^{-1/2}(\Gamma_R)\},
\]
where $\boldsymbol\nu$ is the unit outward normal vector on $\Gamma_R$ and ${\rm
curl}_{\Gamma_R}$ is the surface scalar curl on $\Gamma_R$.

Below is a classical result from the theory of Fourier analysis.

\begin{lemm}\label{fst}
 Let $U_R=(-R, R)^d\subset\mathbb R^d$ be a box. For $\boldsymbol f\in
L^2(U_R)^d$, define the Fourier coefficients
\[
 \hat{\boldsymbol f}_{\boldsymbol
n}=\frac{1}{(2R)^d}\int_{U_R}\boldsymbol f(\boldsymbol x)e^{-{\rm
i}(\frac{\pi}{R})\boldsymbol x\cdot\boldsymbol n}{\rm d}\boldsymbol
x, \quad \boldsymbol n\in \mathbb Z^d.
\]
Then $\boldsymbol f$ has the Fourier series expansion
\[
 \boldsymbol f(\boldsymbol x)=\sum_{\boldsymbol n\in \mathbb
Z^d}\hat{\boldsymbol f}_{\boldsymbol n}
e^{{\rm i}(\frac{\pi}{R})\boldsymbol x\cdot\boldsymbol n}
\]
in the $L^2$-sense, i.e.,
\[
\int_{U_R} \Big|\boldsymbol f(\boldsymbol x) - \sum_{|\boldsymbol n|
\leq N} \hat{\boldsymbol f}_{\boldsymbol n}e^{{\rm i}(\frac{\pi}{R})
\boldsymbol x\cdot\boldsymbol n}\Big|^2{\rm d}\boldsymbol x \to 0, \quad N \to
\infty.
\]
Moreover,
\[
 \|\boldsymbol f\|^2_{L^2(U_R)}=(2R)^d\sum_{\boldsymbol n\in\mathbb
Z^d}|\hat{\boldsymbol f}_{\boldsymbol n}|^2.
\]
\end{lemm}

The following lemma (cf. \cite[Lemma 3.2]{CIL-JDE16}) gives a link between the
values of an analytical function for small and large arguments.

\begin{lemm}\label{caf}
Let $p(z)$ be analytic in the sector
\[
 \mathcal{V}=\{z\in\mathbb C: -\frac{\pi}{4}<{\rm arg}z<\frac{\pi}{4}\}
\]
and continuous in $\bar{\mathcal{V}}$ satisfying
\[
 \begin{cases}
  |p(z)|\leq\epsilon, & z\in (0, ~ K],\\
  |p(z)|\leq M, & z\in\mathcal{V},\\
  |p(0)|=0, & z=0,
 \end{cases}
\]
where $\epsilon, K, M$ are positive constants. Then there exits a function
$\beta(z)$ satisfying
\[
 \begin{cases}
  \beta(z)\geq\frac{1}{2},  & z\in(K, ~ 2^{\frac{1}{4}}K),\\
  \beta(z)\geq \frac{1}{\pi}((\frac{z}{K})^4-1)^{-\frac{1}{2}}, & z\in
(2^{\frac{1}{4}}K, ~ \infty),
 \end{cases}
\]
such that
\[
|p(z)|\leq M\epsilon^{\beta(z)}, \quad\forall z\in (K, ~ \infty).
\]
\end{lemm}

\end{document}